\documentclass[11pt]{amsart} 
\usepackage[bottom=1in]{geometry}
\usepackage{esint} 
\usepackage{graphicx, amssymb, hyperref, MnSymbol}
\usepackage[mathscr]{euscript}
\usepackage{todonotes}
\usepackage[english]{babel}

\newtheorem{theorem}{Theorem}

\newtheorem{definition}{Definition}
\newcommand{\ssize}{\text{size}\,}
\newcommand{\eenergy}{\text{energy}\,}
\newtheorem{lemma}[theorem]{Lemma}
\newtheorem{corollary}[theorem]{Corollary}
\newtheorem{proposition}[theorem]{Proposition}

\newtheorem*{remark}{Remark:}
\newtheorem*{notation}{Notation:}

\newcommand{\sssize}{\widetilde{\text{size}\,}}
\newcommand{\ave}{\text{ave\,}}

\newcommand{\one}{\mathbf{1}}
\newcommand{\dist}{\text{ dist }}
\newcommand{\rr}{\mathbb}
\newcommand{\bbf}{\mathbf}
\newcommand{\ii}{\mathscr}
\newcommand{\ci}{\tilde{\chi}}
\newcommand{\loc}{\text{loc}}
\newcommand{\ds}{\displaystyle}

\newtheorem{question*}{Question}
\newtheorem*{main*}{\underline{Induction statement}}

\newcommand{\n}{\mathfrak{n}}

\newcommand{\ic}{\mathcal}
\newcommand{\g}{\mathfrak{g}}

\def\Xint#1{\mathchoice
   {\XXint\displaystyle\textstyle{#1}}%
   {\XXint\textstyle\scriptstyle{#1}}%
   {\XXint\scriptstyle\scriptscriptstyle{#1}}%
   {\XXint\scriptscriptstyle\scriptscriptstyle{#1}}%
   \!\int}
\def\XXint#1#2#3{{\setbox0=\hbox{$#1{#2#3}{\int}$}
     \vcenter{\hbox{$#2#3$}}\kern-.5\wd0}}

\def\aver#1{\Xint-_{#1}}

\author{Cristina Benea}
\address{Cristina Benea, Universit\'{e} de Nantes, Laboratoire Jean Leray, Nantes 44322, France}
\email{cristina.benea@univ-nantes.fr}

\author[Camil Muscalu]{Camil Muscalu*}
\thanks{$^*$The author is also a Member of the ``Simion Stoilow" Institute of Mathematics of the Romanian Academy}
\address{Camil Muscalu, Department of Mathematics, Cornell University, Ithaca, NY 14853, USA}
\email{camil@math.cornell.edu}

 \title{Sparse domination via the helicoidal method
} 
\begin{document}

\begin{abstract}
Using exclusively the localized estimates upon which the helicoidal method was built in \cite{vv_BHT}, we show how sparse estimates can also be obtained. This approach yields a sparse domination for multiple vector-valued extensions of operators as well. We illustrate these ideas for an $n$-linear Fourier multiplier whose symbol is singular along a $k$-dimensional subspace of $\Gamma=\lbrace \xi_1+\ldots+\xi_{n+1}=0  \rbrace$, where $k<\dfrac{n+1}{2}$, and for the variational Carleson operator.
\end{abstract}

\maketitle

\section{Introduction}

The helicoidal method, which we previously developed in \cite{vv_BHT}, \cite{quasiBanachHelicoid}, represents a new iterative method founded upon very precise local estimates for (quasi)linear and (quasi)multilinear operators and their multiple vector-valued extensions. It allowed us to give positive answers to several questions in harmonic analysis that had been open previously. In particular, we proved that in the local $L^2$ range, all the multiple vector-valued extensions of the bilinear Hilbert transform $BHT$ from \cite{initial_BHT_paper} have the same range as the scalar operator itself. Before that, it was not known whether there exist any vector-valued extensions which behave in a similar way to the scalar bilinear Hilbert transform. For more details, see \cite{vv_BHT}.

In short, the efficiency of the method relies in the sharp localization result, which is rendered global (in a multiple vector-valued form) by an additional stopping time, i.e. an additional decomposition of each of the functiones involved. Through the local estimates and the control of the local operator norm (which are themselves obtained through a stopping time) we can access any triple of Lebesgue exponents from the domain of the bilinear Hilbert transform (including those outside the local $L^2$ range), obtain the localized vector-valued estimates while keeping track of the operator norm, and eventually convert all these into global estimates. Thus, we are making use of iterated stopping times, and this is one of the main differences between the helicoidal method and the previous approaches to vector-valued extensions for $BHT$, which had only produced partial results (see \cite{vv_bht-Prabath}, \cite{vv_quartile}).

In the present work, we will show how the helicoidal method can be naturally adjusted for proving sparse estimates, both for the operator $T$ in question and for any of its multiple vector-valued extension $\vec T$, which takes values in the vector space $X$ (for us, $X$ will consist of iterated Lebesgue spaces). That is, we prove sparse estimates for $\ds \| \vec T \|_{L^q(\rr R; X)(v)}^q$, for any $0<q<\infty$, and any positive, locally integrable function $v$.

An interesting consequence of such sparse estimates are Fefferman-Stein inequalities for $T$ and its vector-valued extensions, which state that the operator $T$ is controlled by a certain maximal operator in $\| \cdot\|_{L^q(v)}$ (quasi)-norms. Consequently, vector-valued weighted estimates can be deduced.

Our method applies to many operators in harmonic analysis. We will illustrate these ideas for the natural generalization of the bilinear Hilbert transform, which is $T_k$, an $n$-linear Fourier multiplier whose symbol is singular along a $k$-dimensional subspace of $\Gamma=\lbrace \xi_1+\ldots+\xi_{n+1}=0  \rbrace$, where $k<\dfrac{n+1}{2}$. For the time being, no other strategy of proving sparse estimates applies to these operators, while the localization principle and the helicoidal method apply straightaway. We also give a special consideration to $T_0$, which corresponds to a multilinear Calder\'on-Zygmund operator. Specifically, we prove multiple vector-valued (and sparse) inequalities even in the situation when $L^\infty$ spaces are involved, which was not known before. Estimates involving $L^\infty$ played an important role in proving mixed-norm estimates for multi-parameter Calder\'on-Zygmund operators and for deducing Leibniz rules in mixed-norm spaces in \cite{vv_BHT}.

Another case study is the variational Carleson operator from \cite{variational_Carleson}, for which the localization principle produces an extra simplification in the proof: more precise local estimates (which are themselves obtained through stopping times) allow for a clean-cut exceptional set and a simpler stopping time which produces the global estimate. See also the expository paper \cite{expository-hel}.

We will see that the \emph{local} property of an operator (or of its bilinear/ multilinear form) is at the heart of ``sharp" vector-valued and weighted extensions. The two types of extensions are similar in spirit, in the sense that they both amount to changing the measure space and the Lebesgue exponents in the domain of the operator. For linear operators, it is known that there is a deeper connection between vector-valued and weighted extensions (see \cite{RFCuervaBook}). 

In the multilinear case, especially for operators which are not of Calder\'on-Zygmund type, weighted estimates and the appropriate extrapolation theory were not known until very recently, hence the search for vector-valued extensions was calling for a new approach. For the bilinear Hilbert transform, which should be thought of as the least complex of such operators, partial results were obtained in \cite{vv_quartile} combining time-frequency analysis and UMD spaces techniques, and also in \cite{vv_bht-Prabath}, where ideas regarding vector-valued estimates from \cite{BatemanThiele2011} were incorporated in the time-frequency analysis method.

In \cite{myphdthesis}, the interest in vector-valued extensions for the bilinear Hilbert transform was prompted by a \emph{bilinear Rubio de Francia operator for iterated Fourier integrals}:
\[
RF_r(f, g)(x)= \Big(  \sum_k \Big|  \int\limits_{a_k < \xi_1 <\xi_2 < b_k} \hat{f}(\xi_1) \hat{g}(\xi_2) e^{2 \pi i x \left( \xi_1+\xi_2\right)} d \xi_1 d \xi_2   \Big|^r   \Big)^\frac{1}{r},
\]
which is associated to an arbitrary collection $\lbrace [a_k, b_k]  \rbrace_k$ of intervals with bounded overlap, and a Lebesgue exponent $1 \leq r \leq \infty$. Given that in its multiplier form $BHT$ can be written as
\[
BHT(f, g)(x):=  \int\limits_{ \xi_1 <\xi_2 } \hat{f}(\xi_1) \hat{g}(\xi_2) e^{2 \pi i x \left( \xi_1+\xi_2\right)} d \xi_1 d \xi_2,
\] 
the \emph{non-linear} operator $RF_r$ can be regarded as a \emph{linear} $\ell^r$-valued extension of the bilinear Hilbert transform. The natural restrictions for $r$, in the bilinear setting (and which proved to be accurate), seemed to be $1 \leq r \leq \infty$. This should the compared to the linear Rubio de Francia $\ell^r$-valued operator associated to arbitrary Fourier projections onto a family of intervals of bounded overlap, where the condition $2 \leq r \leq \infty$ is necessary. For the linear operator, $\ell^\infty$ estimates for the sharp Fourier projections are deduced from the boundedness of Carleson operator; for the bilinear operator $RF_r$, the case $r=\infty$ can be studied by using the bi-Carleson operator of \cite{bi-Carleson} (and again, the Carleson operator).

In dealing directly with the operator $RF_r$, we developed a proof based on \emph{spatial localizations}, the principle being that in order to control arbitrary frequency scales which are dictated by a quantitative size (a vector-valued square function), a local maximal operator which only retains the spatial information is needed. Stopping times (i.e. a decomposition procedure of the operator according to varying averages) and the local control of the operator by local maximal averages imply the boundedness of the operator.  

The same method, i.e. a local control of the operator by localized ``sizes" (or local maximal averages), produced in \cite{vv_BHT} and \cite{quasiBanachHelicoid} multiple vector-valued extensions for $BHT$. With an alteration of the stopping time procedure, we will show that it also yields sparse estimates even in the multiple vector-valued setting. More generally, this will apply to operators that admit a decomposition into wave packets indexed after a certain collection $\rr P$, in the sense that they can be represented as
\begin{equation}
\label{eq:representation}
T_{\rr P}(f_1, \ldots, f_{n})(x)=\sum_{P \in \rr P} \vert I_P \vert^{-\frac{n-1}{2}} \langle f_1, \phi_P^1 \rangle \cdot \ldots \langle f_{n}, \phi_P^{n} \rangle \phi_{P}^{n+1}(x).
\end{equation}
We claim that sparse and vector-valued estimates are implied by a suitable local estimate for the spatial localization $T_{\rr P(I_0)}$ of the operator (we only consider tiles $P \in \rr P$ with the spatial interval $I_P$ contained inside $I_0$), which can be written as
{\fontsize{9}{10}\begin{equation}
\label{eq:local-est-impo}
\big\| T_{\rr P(I_0)} (f_1, \ldots, f_n)  \big\|_{L^q(v^q)}^q \lesssim \prod_{j=1}^n  \Big( \sup_{P \in \rr P \left( I_0 \right)^+} \frac{1}{|I_P|} \int_{\rr R}\vert f_j (x)\vert^{s_j}  \cdot \ci_{I_P}^{M} dx \Big)^\frac{q}{s_j} \cdot \Big( \sup_{P \in \rr P \left( I_0 \right)^+} \frac{1}{|I_P|} \int_{\rr R}\vert v (x)\vert^{s_{n+1}}  \cdot \ci_{I_P}^{M} dx \Big)^\frac{q}{s_{n+1}} \cdot |I_0|.
\end{equation}} 
Above, $0 < q \leq 1$ (on subadditivity grounds), $1 \leq s_1, \ldots, s_{n+1} <\infty$ are certain Lebesgue exponents which reflect the operator's properties, and $v^q$ is a positive, locally integrable function. The case $q=1$ corresponds to the study of the $(n+1)$-linear form, and such local estimates were used in \cite{vv_BHT}. For $q<1$, the local estimate appeared in \cite{quasiBanachHelicoid}, for particular functions that are bounded above by characteristic functions of sets of finite measure. In fact, if the range of exponents $(s_1, \ldots, s_{n+1})$ for which we want to prove the sparse domination is open, local estimates for restricted type functions are sufficient (see Section \ref{sec:restricted-type-to-general}). 

We note that our approach yields sparse estimates by making use solely of the local character of the operator (which is an innate property): this is distinct from Lerner's use of \emph{local mean oscillation} from \cite{Lerner-simplerA_2}, or from Lacey's or Lerner's reiterative slicing from \cite{Lacey-sparse} or \cite{Lerner-max-op-sparse}. 

What is more, we obtain a global Fefferman-Stein inequality: provided $v^q$ is a weight satisfying a \emph{reverse H\"older}-$\frac{s_{n+1}}{q}$ inequality, we have that
\begin{equation}
\label{eq:fefferman0stein-intro}\tag{FS}
\big\|  T( f_1, \ldots, \vec f_n) \big\|_{L^q(v^q)} \lesssim  \big\|  \vec {\ic M}_{s_1, \ldots, s_n} (f_1, \ldots, f_n)  \big\|_{L^q(v^q)},
\end{equation}
where $\vec {\ic M}_{s_1, \ldots, s_n}$ is the multi(-sub)linear maximal operator 
\begin{equation}
\label{eq:def-func-max-multilin}
\vec {\ic M}_{s_1, \ldots, s_n}(f_1, \ldots, f_n)(x):=\sup_{Q \ni x} \prod_{j=1}^n  \big(  \frac{1}{|Q|}  \int_{Q}  |f_j(y)|^{s_j} dy  \big)^\frac{1}{s_j}.
\end{equation}
That is, $\vec {\ic M}_{s_1, \ldots, s_n}$ (and as a result the product of $n$ Hardy-Littlewood maximal functions $\ic M_{s_1} \cdot \ldots \cdot \ic M_{s_n}$), controls the operator $T$ in any $L^q(v^q)$ norm, under certain conditions on $s_1, \ldots , s_n$.

The first instance of a \eqref{eq:fefferman0stein-intro} inequality available for all $0<q<\infty$  appeared in \cite{FeffStein-RealHpSpaces}, and the weighted case in \cite{coifman-fefferman}; sometimes they are referred to as Coifman-Fefferman inequalities in the literature. Commonly, the Fefferman-Stein inequalities are obtained as a consequence of \emph{good lambda} inequalities, which also yield weighted estimates. In the context of Calder\'on-Zygmund theory, the good lambda inequalities assert that level sets of the maximal function control level sets of Calder\'on-Zygmund operators (or of their oscillations/ maximal truncations).

Broadly speaking, the philosophy behind our approach for Fefferman-Stein inequalities or for sparse domination is somewhat similar: first, we know that locally $T$ is controlled by the product of $n+1$ maximal averages. A stopping time will identify the intervals (and associated collections of tiles) where \emph{all} of the maximal averages are uniformly controlled by averages. Since all these intervals chosen through the stopping time correspond to level sets of localized maximal functions, they are relatively ``sparse". This last property allows to recover the global $L^q$ norm of $\vec {\ic M}_{s_1, \ldots, s_n, s_{n+1}}(f_1, \ldots, f_n, v)$.

On the other hand, for vector-valued estimates of the type $T: L^{p_1}( \ell^{r_1} ) \times \ldots \times L^{p_n}( \ell^{r_n} ) \to L^{p}( \ell^{r} )$, i.e. an inequality such as
\[
\big\| \big( \sum_k \vert  T(f_1^k, \ldots, f_n^k)  \vert^r  \big)^\frac{1}{r} \big\|_p \lesssim \prod_{j=1}^n \big\| \big( \sum_k \vert f_j^k \vert^{r_j}  \big)^\frac{1}{r_j} \big\|_{p_j},
\]
each of the the maximal $L^{s_j}$ averages of $f_j^k$ will be converted through a stopping time into $L^{r_j}$ averages, while making sure to save the ``spare" information, which is set aside as an operatorial norm. The $L^{r_j}$ averages are summed up via H\"older, and transformed again into $L^{s_j}$ maximal averages of $\big( \sum_k \vert f_j^k \vert^{r_j}  \big)^\frac{1}{r_j}$, eventually localized. The stopping times, as the Calder\'on-Zygmund decomposition, play an important role since they allow to transform maximal averages into averages.

Hence, all of the above apply equally to multiple vector-valued extensions of the operator $T$: in this setting too, we can obtain sparse domination and Fefferman-Stein inequalities. All of these will be made precise shortly.

We recall the linear theory first. If $T$ is a linear operator so that $T: L^p(dx) \to L^p(dx)$ for $p \in Range(T)$, and $(\ii W, \mu)$ is a $\sigma$-finite measure space, we want to find $Range(\vec T_r)$, the range of exponents $r$ so that 
\begin{equation}
\label{eq:VVE} \tag{VVE}
T: L^p(\rr R; L^r(\ii W, \mu)) \to L^p(\rr R; L^r(\ii W, \mu)).
\end{equation}
In the weighted theory, we want to characterize the class of \emph{weights} $w$ (that is, positive, locally integrable functions) for which 
\begin{equation}
\label{eq:WE}\tag{WE}
T: L^p(w \,dx) \to L^p(w \, dx).
\end{equation}

For Calder\'on-Zygmund operators, the collection of weights for which \eqref{eq:WE} holds is the $A_p$ class containing weights $w$ that satisfy
\begin{equation}
\label{eq:A_p-class}
\sup_{I \text{ interval }} \aver{I} w(x) dx  \cdot \Big(  \aver{I} w(x)^{1-p'} dx \Big)^{p-1} <\infty.
\end{equation}
The quantity in \eqref{eq:A_p-class} represents the \emph{$A_p$ characteristic} of $w$ and is denoted $[w]_{A_p}$. Alternatively, one can describe these weights as being exactly those for which the Hardy-Littlewood maximal function $\ic M$ is a bounded operator on $L^p(w\, dx)$, for $1<p<\infty$. 

The class $A_p$ turned out to be especially suited for extrapolation: more exactly, if we know that \emph{$T: L^p(w \,dx) \to L^p(w \,dx)$ for all $w \in A_p$}, then 
\[
T: L^s(w\, dx) \to L^s(w\, dx) \quad \text{for all} \quad  w \in A_s,
\]
and moreover, we can also obtain the vector-valued extension:
\[
T: L^s(\rr R; L^r(w\, dx)) \to L^s(\rr R; L^r(w\, dx)) \quad \text{for all} \quad  w \in A_s \text{   and any   } 1<r<\infty.
\]
As a consequence, vector-valued estimates can be obtained through extrapolation, once weighted estimates in a suitable class are known. 


A well-known question in harmonic analysis had been the $A_2$ conjecture, which predicted a linear dependency of $[w]_{A_2}$ in \eqref{eq:WE}, for Calder\'on-Zygmund operators, if $p=2$. Although this was answered positively in \cite{Hytonen-A2}, simpler proofs were provided (\cite{Lerner-simplerA_2}, \cite{Lacey-sparse}, \cite{Lerner-max-op-sparse}), based on an underlying sparse structure that had appeared previously in \cite{Lernerfirst}, \cite{Hytonen-A2} and \cite{CUMP}.

As already announced, we will illustrate how the \emph{helicoidal method}, a technique developed in \cite{myphdthesis}, \cite{vv_BHT} for showing vector-valued estimates for the bilinear Hilbert transform operator $BHT$, can be used also for obtaining sparse estimates and multiple vector-valued sparse estimates. The singularity of the multiplier of $BHT$, which is defined by
\[
(f, g) \mapsto  \int\limits_{ \xi_1 <\xi_2 } \hat{f}(\xi_1) \hat{g}(\xi_2) e^{2 \pi i x \left( \xi_1+\xi_2\right)} d \xi_1 d \xi_2,
\]
differentiates it from general Calder\'on-Zygmund operators. A weighted theory for operators similar to $BHT$ was only recently investigated in \cite{weighted_BHT}, \cite{extrap-BHT}, \cite{martell-kangwei-mulilinear-weights-extrapolation}.

The multiple vector-valued spaces considered here are mixed norm $L^p$ spaces. Let $m \geq 1$ be a natural number, $\lbrace(\ii W_j, \Sigma_j, \mu_j )\rbrace_{1 \leq j \leq m} $ are totally $\sigma$-finite measure spaces and $P=(p_1, \ldots, p_m)$ is an $m$-tuple. Then on the product space $\ds (\ii W, \Sigma, \mu):= (\prod_{j=1}^m \ii W_j, \prod_{j=1}^m \Sigma_j, \prod_{j=1}^n \mu_j)$ the mixed norm is defined by
\[
\|  f \|_{P}:=\big( \int_{\ii W_1}\ldots  \big( \int_{\ii W_m } \vert f(w_1, \ldots, w_m)      \vert^{p_m} d \mu_m(w_m)  \big)^{\frac{p_{m-1}}{p_m}} \ldots   d \mu_1(w_1) \big)^{\frac{1}{p_1}}
\]

In \cite{vv_BHT}, estimates for $BHT$ on such spaces were necessary in the study of multi-parameter operators, such as $BHT \otimes \Pi \otimes \ldots \otimes \Pi$ or $\Pi\otimes \Pi \otimes \ldots \otimes \Pi$, where $\Pi$ is a paraproduct \textbackslash bilinear Calder\'on-Zygmund operator.

As mentioned before, one of our case studies is the $n$-linear operator $T_k$ given by a multiplier singular along a $k$-dimensional subspace of $\Gamma:=\lbrace \xi_1+\ldots+\xi_{n+1}=0   \rbrace$, where $k<\dfrac{n+1}{2}$. Such operators were studied in \cite{multilinearMTT}. We note that the case $n=2, k=1$ corresponds to the $BHT$ operator of \cite{LaceyThieleBHTp>2}, \cite{initial_BHT_paper}, while the case $n=3, k=2$, which does not satisfy our sufficient condition $k<\dfrac{n+1}{2}$, corresponds to the trilinear Hilbert transform and its boundedness is still an open question. More recently, in \cite{Robert-unbdd}, it was proved that no $L^p$ estimates can hold for general multipliers if $k \geq \dfrac{n+3}{2}$, and consequently, the only undecided case is $\ds k= \big\lceil   \frac{n+1}{2} \big\rceil$.

Another situation of particular interest is that when $k=0$; then the multilinear operator, denoted $T_0$, is given by a multiplier that is singular at the origin. Hence we recover multilinear Fourier multipliers (sometimes called ``paraproducts"), which are particular cases of multilinear Calder\'on-Zygmund operators. Sparse estimates for these were studied in \cite{Nazarov-Lerner-DyadicCalculus}, and in Section \ref{sec:k=0} we discuss multiple vector-valued generalizations.

For an operator as above, we prove the following vector-valued and sparse estimates:
\begin{theorem}
\label{thm:main-thm}
Let $\Gamma'$ be a subspace of $\Gamma:=\lbrace \xi \in \rr R^{n+1}: \xi_1+\ldots+\xi_{n+1}=0 \rbrace$ of dimension $k$ where $1 \leq k < \frac{n+1}{2}$. Assume that $\Gamma'$ is non-degenerate in the sense of \cite{multilinearMTT}, and that $m$ satisfies the estimate
\[
\vert \partial_\xi^\alpha m(\xi)  \vert \lesssim \dist (\xi, \Gamma')^{-\vert \alpha \vert}
\]
for all partial derivatives $\partial_\xi^\alpha$ on $\Gamma$ up to some finite order.
Let $T_k$ be an $n$-linear operator associated to the $(n+1)$-linear form
\begin{equation}
\label{eq:def-op-T_m}
\Lambda(f_1, \ldots, f_{n+1})=\int_{\xi_1+\ldots+\xi_{n+1}=0} m(\xi_1, \ldots, \xi_{n+1}) \hat{f}_1(\xi_1) \cdot \ldots \hat{f}_{n+1}(\xi_{n+1}) d \xi_1 \ldots d \xi_{n+1}.
\end{equation}
Then the operator $T_k$ admits multiple vector-valued extensions of depth $m$
\begin{equation*}
T_k: L^{p_1}\big(  \rr R; L^{R_1}(\ii W, \mu) \big) \times \ldots \times L^{p_n}\big(  \rr R; L^{R_n}(\ii W, \mu) \big)  \to L^{p_{n+1}'}\big(  \rr R; L^{{R'}_{n+1}}(\ii W, \mu) \big) 
\end{equation*}
for $(p_1, \ldots, p_n, p_{n+1})$ and $m$-tuples $(R_1, \ldots, R_n, R_{n+1})$ for which there exist a tuple $(\alpha_1, \ldots , \alpha_{n+1})$ with $\alpha_j \in (0, \frac{1}{2})$ as in \eqref{def:exp-alpha_j}, so that the following conditions are satisfied:
\begin{equation}
\label{eq:Holder-tuples-for-multilinear-op}
1 < p_j , {r_j^l} \leq \infty \quad \forall 1 \leq j \leq n, \quad \frac{1}{n} < p'_{n+1}, \big( r^l_{n+1} \big)' <\infty \quad \forall 1 \leq l \leq m,
\end{equation}
\begin{equation}
\label{eq:Holder-again}
\frac{1}{p_1}+\ldots +\frac{1}{p_n}+\frac{1}{p_{n+1}}=1, \qquad \frac{1}{r_1^l}+\ldots + \frac{1}{r_{n}^l}+\frac{1}{r_{n+1}^l}=1, \quad \forall 1 \leq l \leq m,
\end{equation}
\begin{equation}
\label{eq:cond-Leb-exp-vv-T}
\frac{1}{p_j}<1 -\alpha_j, \quad \frac{1}{r_j^l} <1 -\alpha_j \quad \text{for all     } 1\leq j \leq n+1, 1 \leq l \leq m.
\end{equation}

Moreover, $T_k$ admits a multiple vector-valued sparse domination in any $L^q$ space, for any $0< q< \infty$. If $X_j$ denotes the space $ L^{R_j}(\ii W, \mu)$ for all $1 \leq j \leq n$, $X'_{n+1}=L^{R'_{n+1}}(\ii W, \mu)$, and if there exists a tuple $(\alpha_1, \ldots , \alpha_{n+1})$ with $\alpha_j \in (0, \frac{1}{2})$ as in \eqref{def:exp-alpha_j} for which 
\begin{equation}
\label{eq:cond-mvvs}
\frac{1}{r_j^l} <1 -\alpha_j \quad \text{for all     } 1\leq j \leq n+1, 1 \leq l \leq m,
\end{equation}
then, given Lebesgue exponents $s_j$ with $\frac{1}{s_j}<1-\alpha_j$ for all $1 \leq j \leq n$ and $\frac{1}{s_{n+1}}<\frac{1}{q}-\alpha_{n+1}$, any vector-valued functions $\vec f_1, \ldots, \vec f_{n}$ so that $\| \vec f_j(x, \cdot) \|_{X_j}$ are locally integrable, and $v$ a locally $q$-integrable function, we can construct a sparse collection $\ic S$ depending on the functions $\vec f_j$ and $v$ and the exponents $s_j$ and $q$ for which
{\fontsize{10}{10}\begin{equation}\label{eq:sparse-T_k-all-tau}
\big\| \big\|T_k(\vec f_1, \ldots, \vec f_n)\big\|_{X'_{n+1}}  \cdot v \big\|_q^q \lesssim \sum_{Q \in \ic S} \prod_{j=1}^{n}\big( \frac{1}{|Q|} \int_{\rr R} \| \vec f_j(x, \cdot) \|_{X_j}^{s_j} \cdot \ci_{Q}^M dx \big)^{q /{s_j}} \big( \frac{1}{|Q|} \int_{\rr R} |v(x)|^{s_{n+1}} \cdot \ci_{Q}^M dx \big)^{q /{s_{n+1}}} |Q|.
\end{equation}}
\end{theorem}

\begin{remark}
\begin{itemize}
\item[(i)] In the Banach case (i.e. when $1< r_j^l \leq \infty$ for all $1 \leq j \leq n+1, 1 \leq l \leq m$), if $q=1$, the sparse estimate above is equivalent to a multiple vector-valued sparse domination for the $n+1$-linear form.
\item[(ii)] We call the estimate in Theorem \ref{thm:main-thm} a ``depth\,-$m$" inequality, meaning that the vector spaces correspond to $m$-tuples. We refer to the classical scalar-valued inequality as a`` $0$-depth" inequality. 
\item[(iii)] Both the multiple vector-valued and the sparse multiple vector-valued estimates associated to $(R_1, \ldots, R_{n+1})$ are conditioned by the existence of a tuple $(\alpha_1, \ldots, \alpha_{n+1})$ given by \eqref{def:exp-alpha_j}, for which condition \eqref{eq:cond-mvvs} holds. Then the Lebesgue exponents $p_j$ in the case of the vector-valued extension (which satisfy \eqref{eq:Holder-again}), and the Lebesgue exponents $s_j$ in the case of the sparse domination (which we are trying to minimize), must also verify $\frac{1}{p_j} <1 -\alpha_j$ and $\frac{1}{s_j} <1 -\alpha_j$, respectively (with the exception of $s_{n+1}$, which should satisfy $\frac{1}{s_{n+1}}< \frac{1}{q}-\alpha_{n+1}$).
\end{itemize}
\end{remark}

As a consequence of the sparse estimate \eqref{eq:sparse-T_k-all-tau}, we deduce the following
\begin{corollary}[Fefferman-Stein inequality for $T_k$]
\label{cor:fefferman-stein-general}
Let $(\alpha_1, \ldots, \alpha_{n+1})\in (0, \frac{1}{2})^{n+1}$ be a tuple defined by \eqref{def:exp-alpha_j}, and $0<p<\infty$. For any $ s_1, \ldots, s_n, s_{n+1}$ satisfying $\ds \frac{1}{s_j} <1-\alpha_j$ for all $1 \leq j \leq n$, $\frac{1}{s_{n+1}} <\frac{1}{p}-\alpha_{n+1}$ and any $m$-tuples $(R_1, \ldots, R_n, R'_{n+1})$ satisfying \eqref{eq:cond-mvvs}, we have
\[
\big\|  \| T_k( \vec f_1, \ldots, \vec f_n)  \|_{L^{R'_{n+1}}(\ii W, \mu)} \big\|_p \lesssim \| \prod_{j=1}^n \ic M_{s_j} ( \| \vec f_j (x,\cdot)  \|_{L^{R_j}(\ii W, \mu)} )   \|_p.
\]

Moreover, if $w$ is a weight in $RH_{\frac{s_{n+1}}{p}}$, we have that 
\begin{equation}
\label{eq:maximal-bht-weights}
\big\|  \| T_k( \vec f_1, \ldots, \vec f_n)  \|_{L^{R'_{n+1}}} \big\|_{L^p(w)} \lesssim  \big\|  \vec {\ic M}_{s_1, \ldots, s_n} ( \| \vec f_1 (x,\cdot)  \|_{L^{R_1}}, \ldots, \| \vec f_n (x,\cdot)  \|_{L^{R_n}} )  \big\|_{L^p(w)},
\end{equation}
where $\vec {\ic M}_{s_1, \ldots, s_n}$ is the multi(sub-)linear maximal operator defined in \eqref{eq:def-func-max-multilin}. The implicit constant in \eqref{eq:maximal-bht-weights} depends on the weight $w$.
\end{corollary}
For multilinear Calder\'on-Zygmund operators, a similar result is due to Coifman and Meyer \cite{CoifmanMeyer-Commutators}. A weighted estimate involving products of maximal functions appeared in \cite{GrafTorres-maxop-multilin}, and in \cite{NewMaxFnMultipleWeights} an inequality as above using the multilinear maximal function was obtained; in both cases, the condition on the weight is that $w \in A_\infty$. We sketch the proof of this result in Section \ref{sec:carleson_op-only}.

Together with the weighted estimates that we will prove for $\vec {\ic M}_{s_1, \ldots, s_n}$ in Proposition \ref{prop-strong-est-weighted-multi-max}, we deduce the following weighted estimates for $T_k$:
\begin{corollary}
\label{cor:main-weights-vector-weights}
Let $(\alpha_1, \ldots, \alpha_{n+1})\in (0, \frac{1}{2})^{n+1}$ be a tuple defined by \eqref{def:exp-alpha_j}, $1 < s_1, \ldots, s_{n+1} <\infty$ are exponents satisfying $\ds \frac{1}{s_j} <1-\alpha_j$ for all $1 \leq j \leq n$, $\frac{1}{s_{n+1}} <\frac{1}{q}-\alpha_{n+1}$, while the $m$-tuples $(R_1, \ldots, R_n, R'_{n+1})$ comply with condition \eqref{eq:cond-mvvs} of Theorem \ref{thm:main-thm}. Then for any $q_1,\ldots, q_n, q$ so that $s_j < q_j \leq \infty$ for all $1 \leq j \leq n$ and $\frac{1}{q_1}+ \ldots+\frac{1}{q_{n}}=\frac{1}{q}$, the operator $T_k$ satisfies the following vector-valued weighted estimate:
{\fontsize{10}{10}\begin{equation*}
T_k: L^{q_1}\big(  \rr R; L^{R_1}(\ii W, \mu) \big)(w_1^{q_1}) \times \ldots \times L^{q_n}\big(  \rr R; L^{R_n}(\ii W, \mu) \big)(w_n^{q_n})  \to L^{q}\big(  \rr R; L^{{R'}_{n+1}}(\ii W, \mu) \big)(w^{q}),
\end{equation*}}
where $w=w_1 \cdot \ldots \cdot w_n$ and the vector weight $\vec w=(w_1^{q_1}, \ldots, w_n^{q_n})$ satisfies the condition
\begin{equation}
\label{eq:joint-weight-condition}
\sup_{Q} \big( \aver{Q} w^{s_{n+1}} \big)^\frac{1}{s_{n+1}} \, \prod_{j=1}^n \Big(  \aver{Q} w_j^{- \frac{s_j \, q_j}{q_j-s_j}} \Big)^{\frac{1}{s_j}-\frac{1}{q_j}}<+\infty.
\end{equation}
Note that we allow $q_j=\infty$ for some of the $j$, but in that case $w_j \equiv 1$.
\end{corollary}

Recently, weighted estimates were obtained through extrapolation in \cite{extrap-BHT}, for the bilinear Hilbert transform operator (though the classes of weights involved are smaller, and $L^\infty$ spaces do not appear). The method of the proof doesn't allow to recover all the vector-valued estimates from \cite{vv_BHT} and \cite{quasiBanachHelicoid}. Even more recently, an extrapolation theory for vector weights was developed in \cite{martell-kangwei-mulilinear-weights-extrapolation}, which, together with the weighted estimates from \cite{weighted_BHT}, allow to recover all the vector-valued estimates from \cite{vv_BHT} and \cite{quasiBanachHelicoid} (except for the case when $L^\infty$ spaces are involved). A more comprehensive discussion and a proof of Corollary \ref{cor:main-weights-vector-weights} will follow in Section \ref{sec:weighted-est}.

\begin{remark}
The condition \eqref{eq:joint-weight-condition} on the vector weight is the same as that in \cite{martell-kangwei-mulilinear-weights-extrapolation} or \cite{weighted_BHT}, but written somehow differently. If $q \geq 1$, the condition on $s_{n+1}$ is that $\ds \frac{1}{s_{n+1}} <\frac{1}{q}-\alpha_{n+1}$, which is the same as $\ds \frac{1}{s_{n+1}}< \big( 1-\alpha_{n+1} \big) -\frac{1}{q'}$.

In Corollary \ref{cor:main-weights-vector-weights}, we don't keep track of the ``qualitative" dependency on the vector weight, but it can be traced directly from the sparse form \eqref{eq:sparse-T_k-all-tau}.
\end{remark}

The exact range of exponents for the multiple vector-valued inequalities depends closely on the values of the tuple $(\alpha_1, \ldots, \alpha_{n+1}) \in (0, \frac{1}{2})^{n+1}$. However, if all the exponents $(r^l_1, \ldots, r^l_{n+1})$ for $1 \leq l \leq m$, are ``local $L^2$" exponents (that is, $2 \leq r_j^l \leq \infty$ for all $1 \leq j \leq n+1, 1 \leq l \leq m$), then the range of the depth-$m$ vector-valued extension of $T_k$ corresponding to the tuple $(R_1, \ldots, R_{n+1})$ coincides with the range of the operator in the scalar case. This is because the constraint in \eqref{eq:cond-Leb-exp-vv-T} reduces to a constraint on the $p_j$ only: $\frac{1}{r_j^l} < 1-\alpha_j$ is automatically satisfied. An immediate consequence, which can be obtained by using the tools developed in \cite{vv_BHT}), is the following result:
\begin{corollary}
For any $n \geq 1$ and any $0 \leq k <\dfrac{n+1}{2}$, the multi-parameter operator $T_k \otimes T_0 \otimes \ldots \otimes T_0$ satisfies the same $L^p$ estimates as the operator $T_k$ itself:
\[
T_k \otimes T_0 \otimes \ldots \otimes T_0 : L^{p_1}(\rr R^m) \times \ldots L^{p_n}(\rr R^m) \to L^{p_{n+1}'}(\rr R^m),
\]
for any $(p_1, \ldots, p_n, p_{n+1}') \in Range(T_k)$.
\end{corollary}

This extends our previous results from \cite{vv_BHT}, corresponding to the multi-parameter operator $BHT \otimes \Pi \otimes \ldots \otimes \Pi$.

Also, for $T_0$ (i.e. a multilinear H\"ormander-Mikhlin multiplier) we obtain the following sparse domination:
\begin{theorem}
\label{thm:k=0case}
Let $(R_1, \ldots, R_n, R'_{n+1})$ be $m$-tuples so that $1<r_j^l \leq \infty$ for all $1 \leq l \leq m, 1 \leq j \leq n$, $\frac{1}{2} <  (r^l_{n+1})' <\infty$. Let $\epsilon>0$. For any such tuples $R_j$, any vector-valued functions $\vec f_1, \ldots, \vec f_{n}$ so that $\| \vec f_j(x, \cdot) \|_{R_j}$ are locally integrable, and any $v$ a locally $q$-integrable function, we can construct a sparse collection $\ic S$ depending on the functions $\vec f_j$ and $v$ and the exponents $q$ and $\epsilon$ for which
{\fontsize{9}{10}\begin{equation*}
\big\|  \big\|T_0(\vec f_1, \ldots, \vec f_n)\big\|_{L^{R'_{n+1}}} \cdot v \big\|_q^q \lesssim \sum_{Q \in \ic S} \prod_{j=1}^{n} \big( \frac{1}{\vert Q \vert} \int_{\rr R} \big\|  \vec f_j(x, \cdot) \big\|_{\tilde X_j}^{(1+\delta_j)} \cdot \ci_{Q}^{M} dx  \big)^\frac{q}{1+\delta_j} \big( \frac{1}{\vert Q \vert} \int_{\rr R} \big| v(x) \big|^{q+\epsilon_q} \cdot \ci_{Q}^{M} dx  \big)^\frac{q}{q+\epsilon_q}  \cdot  \vert Q \vert.
\end{equation*}}
Above, $\delta_j, \epsilon_q \in \lbrace 0, \epsilon  \rbrace$ with at most $m$ of the $j \neq 0$. If $ \|  \cdot \|_{L^{R'_{n+1}}}^q$ is subadditive, we can take $\epsilon_q=0$.
\end{theorem}

In \cite{Nazarov-Lerner-DyadicCalculus}, a pointwise sparse domination was proved for $T_0$ in the scalar case. The method seems to extend to the case of multiple vector-valued spaces, provided $L^\infty$ spaces are not involved.

The natural Fefferman-Stein inequality, which in the non-weighted scalar case is due to \cite{CoifmanMeyer-Commutators}, and to \cite{multiple-weights-multilinear-op} in the weighted scalar case, has the following formulation:  

\begin{corollary}
\label{cor:Fefferman-Stein-k=0case}
Let $(R_1, \ldots, R_n, R'_{n+1})$ be $m$-tuples so that $1<r_j^l \leq \infty$ for all $1 \leq l \leq m, 1 \leq j \leq n$, $\frac{1}{2} <  (r^l_{n+1})' <\infty$. Let $\epsilon>0$ and $\delta_j \in \lbrace 0, \epsilon  \rbrace$ with at most $m$ of the $j \neq 0$. If $w$ is an $A_\infty$ weight and $0  <q <\infty$, then
{\fontsize{10}{10}\begin{align*}
\big\| \big\|T_0(\vec f_1, \ldots, \vec f_n)\big\|_{L^{R'_{n+1}}}^q \big \|_{L^q(w)} \lesssim \big\| \vec {\ic M}_{1+\delta_1, \ldots, 1+\delta_n}(\big\|  \vec f_1(x, \cdot) \big\|_{L^{R_1}}, \ldots, \big\|  \vec f_n(x, \cdot) \big\|_{L^{R_n}})(x) \big\|_{L^q(w)}.
\end{align*}}
\end{corollary}

As a second case study, we obtain (multiple) vector-valued and sparse estimates for the variational Carleson operator:
\[
\ic C^{var, r} (f)(x):= \sup_{K} \sup_{n_0< \ldots < n_k} \big( \sum_{\ell=1}^K \vert  \int_{a_{n_{\ell-1}}}^{a_{n_\ell}} \hat{f}(\xi) e^{2 \pi i x \xi}  d \xi  \vert^r \big)^{\frac{1}{r}},
\]
where $r>2$.

\begin{theorem}
\label{thm:multi-var-Carleson}
If $R=(r^1, \ldots, r^m)$ is an $m$-tuple with $r'< r^j<\infty$, then we have
\[
\ic C^{var, r}: L^p(\rr R; L^R(\ii W, \mu)) \to L^p(\rr R; L^R(\ii W, \mu))
\]
for all $r' <p<\infty$.

Further, given $0< q < \infty $, $\epsilon, \epsilon_q >0$, $\vec f$ a multiple vector-valued function with $\| f(x, \cdot) \|_{L^R(\ii W, \mu)}$ locally integrable and $v$ a locally $q$-integrable function, there exists a sparse collection $\ic S$ of dyadic intervals (depending on the preceding parameters) so that
{\fontsize{10}{10}\[
\big\| \|\ic C^{var, r}\|_{L^R(\ii W, \mu)} \cdot v  \big\|_q^q \lesssim \sum_{Q \in \ic S} \big( \frac{1}{\vert Q \vert} \int_{\rr R} \| f(x, \cdot) \|_{L^R(\ii W, \mu)}^{r'+\epsilon} \cdot \ci_{Q}^M(x) dx \big)^{\frac{q}{r'+\epsilon}} \cdot \big( \frac{1}{\vert Q \vert} \int_{\rr R} |v(x)|^{q+\epsilon_q} \cdot \ci_{Q}^M(x) dx \big)^{\frac{q}{q+\epsilon_q}} \cdot \vert Q \vert.
\]}

If $q \leq 1$, we can allow for $\epsilon_q=0$.
\end{theorem}

We note that the scalar sparse domination result of Theorem \ref{thm:multi-var-Carleson}, when $q=1$, was already obtained in \cite{sparse-var-Carleson}. Also, the scalar case for $n=2, k=1, q=1$ of Theorem \ref{thm:main-thm} was proved in \cite{weighted_BHT}. The ``localized outer $L^p$ embeddings", formulated in the language of outer measures of \cite{outer_measures}, sit at the core of the two papers, and are similar in spirit to the localization principle presented herein. Strictly speaking, this localization principle was first published in \cite{myphdthesis}; it was developed in C.B.'s PhD thesis at Cornell University, and afterwards refined in \cite{vv_BHT}.

As usual, the sparse estimates of Theorem \ref{thm:multi-var-Carleson} imply the following weighted inequalities:
\begin{corollary}
For any $\delta>0$, and any $r'+\delta < p < \infty$, we have that
\[
\ic C^{var, r} :L^p(w) \to L^p(w),
\]
for any weight $w \in A_{\frac{p}{r'+\delta}}$. Moreover, the operatorial norm is bounded by
\begin{equation}
\label{eq:op-norm-weights}
\| \ic C^{var, r}\|_{L^p(w) \to L^p(w)} \lesssim \Big(  [w]_{A_{\frac{p}{r'+\delta}}}  \Big)^{\max \big( \frac{1}{p-r' -\delta},1 \big)}.
\end{equation}

Similar estimates hold in the vector-valued case:
\[
\big\| C^{var, r} f(x, \cdot) \big\|_{L^p(\rr R; w dx; L^R(\ii W, \mu))} \lesssim \Big(  [w]_{A_{\frac{p}{r'+\delta}}}  \Big)^{\max \big( \frac{1}{p-r' -\delta},1 \big)} \big\|f(x, \cdot) \big\|_{L^p(\rr R; w dx; L^R(\ii W, \mu))}.
\]
\end{corollary}

The scalar case was already obtained in \cite{sparse-var-Carleson}, as a consequence of similar sparse estimates. The vector-valued weighted estimates follow through extrapolation.

\begin{remark}
If $p>r'$ is fixed, we obtain weighted estimates for all weights
\[
w \in \bigcup_{r'<p_0 <p}  A_{\frac{p}{p_0}} ,
\]
which coincides with the class $\ds A_{\frac{p}{r'}}$.
\end{remark}

A Fefferman-Stein inequality is available, in the multiple vector-valued setting, for the variational Carleson as well:
\begin{corollary}
For any $0<p<\infty$, $\epsilon>0$, and any $m$-tuple $R=(r^1, \ldots, r^m)$ with $r'< r^j<\infty$, we have 
\[
\big\| \big\| \ic C^{var, r} f(x, \cdot) \big\|_{ L^R(\ii W, \mu)}  \big\|_p \lesssim \big\|  \ic M_{r'+\epsilon} \big( \| f(x, \cdot)\|_{L^R(\ii W, \mu)} \big)\big\|_p.
\]

Moreover, if $w \in A_{\infty}$, we have that 
\[
\big\| \big\| \ic C^{var, r} f(x, \cdot) \big\|_{ L^R(\ii W, \mu)}  \big\|_{L^p(w)} \lesssim \big\|  \ic M_{r'+\epsilon} \big( \| f(x, \cdot)\|_{L^R(\ii W, \mu)} \big)\big\|_{L^p(w)}.
\]
\end{corollary}

\begin{remark}
The implicit constant in the inequality above depends on the weight $w$ or on its ${A_\infty}$ characteristic. We don't track that dependence here.
\end{remark}

As mentioned previously, the sparse domination of $\| T \cdot v \|_q^q$, where $T$ allows for a wave packet decomposition as in \eqref{eq:representation}, is implied by a local property of the operator. We present the statement for $T_k$ when $n=2$ and $q=1$, as the hypothesis had already appeared in \cite{vv_BHT}, for a special class of functions.
\begin{proposition}
\label{prop:local->sparse}
Let $\rr P$ be a collection of tri-tiles and $\Lambda_{\rr P}$ the trilinear form associated to $T$. Assume that 
{\fontsize{10}{10}
\begin{align*}
\vert \Lambda_{\rr{P}(I_0)}(f, g, h) \vert & \lesssim  \big(\sup_{P \in \rr P \left( I_0 \right)^+} \frac{1}{|I_P|} \int_{\rr R}\vert f (x)\vert^{s_1}  \cdot \ci_{I_P}^{M} dx\big)^{1/{s_1}} \cdot \big(\sup_{P \in \rr P \left( I_0 \right)^+} \frac{1}{|I_P|}  \int_{\rr R}\vert g (x)\vert^{s_2}  \cdot \ci_{I_P}^{M} dx\big)^{1/{s_2}} \\
& \cdot \big(\sup_{P \in \rr P \left( I_0 \right)^+} \frac{1}{|I_P|}  \int_{\rr R}\vert h (x)\vert^{s_3}  \cdot \ci_{I_P}^{M} dx\big)^{1/{s_3}} \cdot \vert I_0 \vert
\end{align*}}holds for any dyadic interval $I_0$ and any locally integrable functions $f, g$ and $h$. Then there exists $\ic S$ a sparse family of intervals depending on the functions $f, g, h$ and on the Lebesgue exponents $s_1, s_2, s_3$, so that 
{\fontsize{10}{10}\begin{align*}
\label{eq:sparse-BHT-non-restr}
\big| \Lambda_{\rr P}(f, g, h) \big|\lesssim \sum_{Q \in \ic S} \big( \frac{1}{|Q|} \int_{\rr R} \vert f \vert^{s_1} \cdot \ci_{Q}^{M-1} dx \big)^{1/{s_1}}  \big( \frac{1}{|Q|} \int_{\rr R} \vert g \vert^{s_2} \cdot \ci_{Q}^{M-1} dx \big)^{\frac{1}{s_2}}\big( \frac{1}{|Q|} \int_{\rr R} \vert h \vert^{s_3} \cdot \ci_{Q}^{M-1} dx \big)^{1/{s_3}} \cdot |Q|.
\end{align*}}
\end{proposition}


The proof of this result will be presented in Section \ref{sec:local->sparse}.

We now briefly describe the helicoidal method (as used in \cite{vv_BHT} and in \cite{quasiBanachHelicoid}). When proving vector-valued inequalities, the sparse property of the three collections produced through the stopping times is built-in, and it doesn't play a role in itself. The study of $BHT$ involves a coupling of $L^2$ and $L^1$ information, and this motivated the restriction to functions that are bounded above by characteristic functions of sets of finite measure.

An immediate application of Proposition 6.12 in \cite{multilinear_harmonic} is the following estimate for the trilinear form associated to the $BHT$ operator:
\begin{equation}
\label{eq:BHT-size-en}
\vert \Lambda_{BHT(\rr P)}(f_1, f_2, f_3)  \vert \lesssim \prod_{j=1}^3  \big( \sup_{P \in \rr P} \frac{1}{\vert I_P \vert} \int_{\rr R} \vert f_j(x) \vert \cdot \ci_{I_P}^M(x)  dx  \big)^{\theta_j} \cdot \| f_j \|_2^{1-\theta_j},
\end{equation}
where $0 \leq \theta_1, \theta_2, \theta_3<1$ and $\theta_1+\theta_2+\theta_3=1$. The first part is an $L^1$ quantity, similar to a maximal operator, while the second part is just the $L^2$ norm. We want to show that $BHT: L^p \times L^q \to L^s$, where $\dfrac{1}{p}+\dfrac{1}{q}=\dfrac{1}{s}$. Invoking multilinear interpolation, we can assume that $\vert f_j(x) \vert \leq \one_{E_j}$ for $1 \leq j \leq 3$, and it remains to verify that
\[
\vert \Lambda_{BHT(\rr P)}(f_1, f_2, f_3)  \vert \lesssim \vert E_1 \vert^{\alpha_1} \cdot \vert E_2 \vert^{\alpha_2} \cdot  \vert E_3 \vert^{\alpha_3},
\]
for $(\alpha_1, \alpha_2, \alpha_3)$ arbitrarily close to $(\frac{1}{p}, \frac{1}{q}, \frac{1}{s'})$, with $\alpha_1+\alpha_2+\alpha_3=1$ (if we consider $s<1$, which is natural for bilinear operators, the interpolation statement becomes more involved and will not be presented in the introduction).

This reduces the problem to finding the good combination of $\theta_1, \theta_2$ and $\theta_3$ in \eqref{eq:BHT-size-en} for which we can obtain an expression similar to $\ds \vert E_1 \vert^{\frac{1}{p}} \cdot \vert E_2 \vert^{\frac{1}{q}} \cdot  \vert E_3 \vert^{\frac{1}{s'}}$. A careful inspection reveals a constraint on the Lebesgue exponents: for the proof strategy to work, we require that
\begin{equation}
\label{eq:L2-restriction}
\big\vert \frac{1}{p}-\frac{1}{q}  \big\vert \leq \frac{1}{2},
\end{equation}
and in consequence, the adjoint operators need to be considered as well. In particular, estimates close to $L^1 \times L^\infty \to L^1$ cannot be obtained directly, and multilinear interpolation between adjoint operators is needed.

Similarly, the vector-valued inequality reduces to proving
\begin{equation}
\label{eq:BHT-tril-form}
\big| \sum_k \Lambda_{BHT(\rr P)}(f_k, g_k, h_k)  \big| \lesssim |F|^{\frac{1}{p}} \, |G|^{\frac{1}{q}} \, |H|^{\frac{1}{s'}},
\end{equation}
whenever $\vec f= \lbrace f_k \rbrace, \vec g= \lbrace g_k \rbrace, \vec h= \lbrace h_k \rbrace$ are vector-valued functions satisfying $\big( \sum\limits_k | f_k |^{r_1}   \big)^{\frac{1}{r_1}} \leq \one_F$, $\big( \sum\limits_k |  g_k |^{r_2}   \big)^{\frac{1}{r_2}} \leq \one_G$, and $\big( \sum\limits_k |  h_k |^{r'}   \big)^{\frac{1}{r'}} \leq \one_{H}$, with $\ds \frac{1}{r_1}+\frac{1}{r_2}=\frac{1}{r}$.

This is achieved through \emph{localization}: the estimate in \eqref{eq:BHT-size-en} is remodeled and localized in order to obtain the sharp inequality
\[
\big|  \Lambda_{\rr P \left( I_0 \right)}(f_k \cdot \one_F, g_k \cdot \one_G, h_k \cdot \one_{H'})  \big| \lesssim \big\| \Lambda^{F, G, H'}_{\rr P \left( I_0 \right)} \big\|\|f_k \cdot \ci_{I_0}^M \|_{r_1} \|g_k \cdot \ci_{I_0}^M \|_{r_2} \|h_k \cdot \ci_{I_0}^M \|_{r'}, 
\]
where $\big\| \Lambda^{F, G, H'}_{\rr P \left( I_0 \right)} \big\| $ represents the operatorial norm that will be described later.

Then we can sum in $k$ via H\"older's inequality the expressions $\|f_k \cdot \ci_{I_0}\|_{r_1}$, which  are further transformed into $L^p$ norms (thus the ``change of measure space") of the function $\one_F$. In addition, in order to obtain the largest possible range of exponents for the vector-valued extensions, the constraint \eqref{eq:L2-restriction} had to be removed and specific \emph{sharp} estimates were needed. 

As a byproduct of the localization (the $0$-depth inequality), the constraint \eqref{eq:L2-restriction} which confined us to the region $\vert \frac{1}{p}-\frac{1}{q}  \vert <\frac{1}{2}$ can be removed by using the local estimate
\begin{equation}
\label{eq:bht-local-max-op}
\Lambda_{BHT;\rr P(I_0)}(f_1, f_2, f_3) \lesssim \prod_{j=1}^3 \big( \sup_{P \in \rr P} \frac{1}{\vert I_P \vert} \int_{\rr R} \vert \one_{E_j}(x) \vert \cdot \ci_{I_P}^M(x)  dx  \big)^{\frac{1+\theta_j}{2}-\epsilon}   \cdot |I_0|,
\end{equation}
and an additional stopping time. In this way, we can obtain all the known $L^p$ estimates for $BHT$ without using interpolation of adjoint trilinear forms. This was already contained in \cite{myphdthesis} and \cite{vv_BHT}.

Note that \eqref{eq:bht-local-max-op} is precisely the estimate needed for deducing sparse domination in Proposition \ref{prop:local->sparse}. Although it initially appeared in the context of restricted-type functions (we assume $|f_j(x)| \leq \one_{E_j}(x)$), it extends to general functions through an argument resembling interpolation (see Proposition \ref{prop:averages-restr-gen}).

The paper is organized as follows: after introducing some basic notions in Section \ref{sec:preliminaries}, we show in Section \ref{sec:local->sparse} how the local estimate implies almost immediately the sparse domination of the multilinear form and how to remove the restricted-type assumption on the functions. The stopping time that yields the sparse estimates is compared to the one used in \cite{vv_BHT} for obtaining vector-valued estimates in Section \ref{sec:review-hel-mthd}. In Sections \ref{sec:rank-k} and \ref{sec:proof_MainThm} we present the details of the proof of Theorem \ref{thm:main-thm}, proving inductively the multiple vector-valued sparse estimates. Weighted estimates that can be deduced directly from the sparse domination are discussed in Section \ref{sec:weighted-est}. Finally, in Section \ref{sec:Carleson-op} we illustrate our method for the Carleson and variational Carleson operators.

\subsection*{Acknowledgments}

C. B. wishes to express her gratitude to Fr\'ed\'eric Bernicot and Teresa Luque for several discussions on sparse domination and weighted theory.

C. B. was partially supported by NSF grant DMS 1500262  and ERC project FAnFArE no. 637510. C. M. was partially supported by NSF grant DMS 1500262; he also acknowledges partial support from a grant from the Ministry of Research and Innovation of Romania, CNCS--UEFISCDI, project PN--III--P4--ID--PCE--2016--0823 within PNCDI--III.

During the Spring Semester of 2017, C. M. was a member of the MSRI in Berkeley, as part of the Program in Harmonic Analysis, and during the Fall semester of 2017, he was visiting the Mathematics Department of the Universit\'e Paris-Sud Orsay, as a Simons Fellow. He is very grateful to both institutions for their hospitality, and to the Simons Foundation for their generous support.

The authors are equally grateful to Jos\'e Mar\'ia Martell for many clarifying discussions on extrapolation and weighted theory.

\section{Preliminaries}
\label{sec:preliminaries}
\subsection{Multi-tiles and sizes}
We recall some notions pertaining to time-frequency analysis, yet trying to avoid the prominent technicalities associated to the field. The essential step in our approach to sparse or vector-valued inequalities consists in proving a local estimate, such as the one mentioned earlier in Proposition \ref{prop:local->sparse}. This will be carried out in Section \ref{sec:rank-k}, and more definitions will be introduced at that point.

\begin{definition}
\label{def:Holder-tuple}
We call a \emph{H\"older tuple} any tuple $(p_1, \ldots, p_n, p_{n+1})$ of exponents satisfying
\begin{equation}
\label{eq:Holder-tuple}
\frac{1}{p_1}+\ldots +\frac{1}{p_n}=\frac{1}{p_{n+1}}, \quad \text{where     } 1<p_1, \ldots, p_n \leq \infty, \frac{1}{n}<p_{n+1}<\infty.
\end{equation}
\end{definition}

The operators fitting our approach are those that allow for a wave packet decomposition, that is, a decomposition in the time-frequency plane.

\begin{definition}
A \emph{tile} is a rectangle $P=I_P \times \omega_P$ of area one, with the property that $I_P$ is a dyadic interval, and $\omega_P$ is contained in a certain translate of the dyadic grid. A \emph{multi-tile} is a tuple $P=(P_1, \ldots, P_{n+1})$ where each $P_j$ is a tile, and so that $I_{P_j}=I_P$ for all $1 \leq j \leq n+1$ (that is, all the tiles share the same spatial interval $I_P$).
\end{definition}

\begin{definition}
A \emph{wave packet} associated to a tile $P=I_P \times \omega_P$ is a smooth function $\phi_P$ so that $ supp  \,\hat{\phi}_P \subseteq \dfrac{9}{10} \omega_P$ and $\phi_P$ is \emph{$L^2$-adapted} to $I_P$ in the sense that
\begin{equation}
\label{eq:L^2-adapted}
\vert  \phi_P^{\left( k \right)} (x) \vert \leq C_{k, M} \frac{1}{\vert I_P  \vert^{\frac{1}{2}+k}} \big(  1+ \frac{\dist (x, I_P) }{\vert I_P \vert} \big)^{-M}
\end{equation}
for sufficiently many derivatives $k$ and any $M>0$.
\end{definition}

With the above notation, we can study $n$-linear operators that allow for a model $(n+1)$-linear form representable as
\begin{equation}
\label{eq:gen-model-op}
\sum_{P \in \rr P} \vert I_P \vert^{-\frac{n-1}{2}} a_{P_1} \cdot \ldots \cdot a_{P_{n+1}}.
\end{equation}

For the operator introduced in Theorem \ref{thm:main-thm}, the coefficients are given by $\ds a_{P_j}=\langle f_j, \phi_{P_j}  \rangle$, for $1 \leq j \leq n+1$. The Carleson operator, defined by
\begin{equation*}
\ic C f(x):=\sup_{N} \big\vert \int_{\xi < N} \hat{f}(\xi) e^{2 \pi i x \xi} \, d \xi  \big\vert,
\end{equation*}
brings about a measurable function $N(x)$ which attains the supremum in the expression above, hence $a_{P_1}=\langle f_1, \phi_{P_1}\rangle$, $a_{P_2}=\langle f_2, \phi_{P_2} \cdot \one_{ \lbrace N(x) \in \omega_{P_2} \rbrace} \rangle$. In the same way, the variational Carleson operator of Theorem \ref{thm:multi-var-Carleson} involves several functions $\xi_1(x), \ldots, \xi_K(x)$, and in that case we have $n=1$ and $a_{P_1}=\langle f_1, \phi_{P_1}\rangle$, $a_{P_2}=\langle f_2, \phi_{P_2} \cdot \one_{\omega_{P_2}}(\xi_k(\cdot))  \rangle$ for some $1 \leq k \leq K$.

There is no typical way of treating directly these operators, and their study varies greatly upon the properties of their invariants. However, local estimates for these operators can be obtained by looking carefully at the classical boundedness proof and adjusting it accordingly: the goal is now to gather as much information as possible, although at a local level. When passing to the global estimate, much consideration is given to keeping track of the \emph{localized operatorial norm}.

Now we introduce some notations that will be used throughout the paper. In particular, we keep on using the notion of \emph{``size"} in spite of it representing henceforth just a maximal average over dyadic intervals (as opposed to its common meaning of a square function over some subcollection of multi-tiles: see Definition \ref{def:size-m-k}, and its use in \cite{initial_BHT_paper}, \cite{multilinearMTT}, \cite{multilinear_harmonic}, \cite{biest}).

\begin{notation}
Given $I \subseteq \rr R$ an interval, $\one_I$ denotes its characteristic function, while $\ci_I$ is a function $L^\infty$-adapted to $I$. For example, we can set
\[
\ci_I(x):= \big( 1+\frac{\dist(x, I)}{\vert I \vert}  \big)^{-100}.
\] 

Then we define the \emph{weighted average} of a locally integrable function $f$ on $I$ as
\begin{equation}
\label{eq:def-average}
\ave_{I}(f):= \frac{1}{|I|} \int_{\rr R} \vert f(x)  \vert \cdot \ci_I(x) dx.
\end{equation}

\end{notation}

\begin{notation}
A collection of multi-tiles will be usually denoted by $\rr P$. If $I_0$ is a fixed dyadic interval and $\rr P$ a collection of multi-tiles, we use the notations 
\begin{equation}
\label{eq:def-local-coll}
\rr P(I_0):=\lbrace P \in \rr P : I_P \subseteq I_0 \rbrace \qquad \text{and} \quad \rr P (I_0)^+:= \rr P (I_0) \cup P_{I_0}, 
\end{equation}
where $P_{I_0}$ is some multi-tile (not necessarily contained in $\rr P$)  with spatial interval $I_0$.
\end{notation}

Finally, the most important notion is that of \emph{size}:
\begin{definition}
\label{def:size}
If $\rr P$ is a collection of multi-tiles, then we define its \emph{$\sssize_{\rr P}$} with respect to the function $f$ by
\begin{equation}
\label{eq:def-zise-eq}
\sssize_{\rr P}(f):= \sup_{P \in \rr P} \frac{1}{\vert I_P \vert} \int_{\rr R} \vert f(x) \vert \cdot \ci_{I_P}^M(x) dx:= \sup_{P \in \rr P} \ave_{I_P} (f).
\end{equation}

Then $\sssize_{\rr P(I_0)}$ and $\sssize_{\rr P(I_0)^+}$ denote the sizes associated to the collections $\rr P(I_0)$ and $\rr P(I_0)^+$ respectively. Given an interval $I_0$, if the collections $\rr P(I_0)$ is implicitly defined in an unambiguous way, we simply use the notation
\[
\sssize_{I_0}(f)=\sssize_{\rr P\left(I_0\right)^+}(f)=\max \Big( \frac{1}{\vert I_0\vert} \int_{\rr R} \vert f(x) \vert \cdot \ci_{I_0}^M(x) dx,  \,\sssize_{\rr P \left( I_0 \right)} (f) \Big).
\]
\end{definition}

\begin{notation}
We will also use the notations 
\begin{equation}
\label{eq:def-L^p-size}
\sssize_{\rr P\left( I_0 \right)}^{p}(f):= \sup_{P \in \rr P \left( I_0 \right)^+}\big( \frac{1}{|I_P|} \int_{\rr R} | f |^{p} \cdot \ci_{I_P}  dx  \big)^{\frac{1}{p}}, \qquad \ave_{I_0}^{p}(f):= \big( \frac{1}{|I_0|} \int_{\rr R} | f |^{p} \cdot \ci_{I_0}  dx  \big)^{\frac{1}{p}}.
\end{equation}
\end{notation}

\subsection{Sparse collections}
\label{sec:sparse}

Even though the notion of ``sparseness" had appeared in various contexts, with various meanings, in the field of harmonic analysis, here it represents the geometric property of a collections of intervals which doesn't allow for too much stacking. Somehow this is a very good property for obtaining sharp weighted estimates.

\begin{definition}
\label{def-sparse-disj}
Let $0< \eta <1$. A collection $\ic S$ of dyadic intervals is called \emph{$\eta$-sparse} if one can choose pairwise disjoint measurable sets $E_Q \subseteq Q$ with $\vert E_Q \vert \geq \eta \vert Q \vert$ for all $Q \in \ic S$. 
 \end{definition}

\begin{definition}
\label{def-Carleson}
Let $\Lambda >1$. A family $\ic S$ is called \emph{$\Lambda$-Carleson} if for any dyadic interval $Q$ we have
\[
\sum_{\substack{P \in \ic S \\ P \subseteq Q}} \vert P \vert \leq \Lambda \vert Q \vert.
\]
\end{definition}

As it turns out, the two definitions are equivalent:

\begin{lemma}[Lemma 6.3 \cite{Nazarov-Lerner-DyadicCalculus}]
A collection $\ic S$ is $\eta$-sparse if and only if it is $\dfrac{1}{\eta}$-Carleson.
\end{lemma}

Another notion of sparse, which implies the one in Definition \ref{def-sparse-disj}, is the following:
\begin{definition}
\label{def-sparse-desc}
Let $0< \eta<1$. A collection of dyadic intervals $\ic S$ is said to be \emph{$\eta$-sparse} if for each $Q \in \ic S$ we have
\[
\sum_{P \in ch_{\ic S}(Q)} \vert P \vert \leq (1-\eta) \vert Q \vert,
\]
where $ch_{\ic S}(Q)$ is the collection of direct descendants of $Q$ in $\ic S$: the maximal elements of $\ic S$ that are strictly contained in $Q$.
\end{definition}

It is this last notion that is most commonly verified in practice, for example in \cite{Lacey-sparse}, \cite{BernicotFreyPetermichl}, \cite{Lerner-max-op-sparse}.

\section{Sparse Domination via the helicoidal method : local estimates imply sparse domination}
\label{sec:local->sparse}

In this section we show how the local estimate given by the scalar $\ic P(0)$ statement (equation \eqref{eq:localBHT} below) of the helicoidal method from \cite{vv_BHT} implies the expected scalar sparse domination of the multilinear form (corresponding to the case $q=1$). This procedure is quite general, as can be seen from the proof of Theorem \ref{thm:local->sparse} below, but will be illustrated first for trilinear forms. 

Later on, provided we have subadditivity, we prove in Proposition \ref{prop:local->sparse->subadditive} how a local estimate in (quasi-) norm (such as $\ic P(0)$ of the helicoidal method applied to quasi-Banach spaces in \cite{quasiBanachHelicoid}) implies a sparse domination in (quasi-) norm (estimate \eqref{eq:sparse-T_k-all-tau} of Theorem \ref{thm:main-thm}).

\subsection{The stopping time for the sparse collection}
\label{sec:stopping-time-local->sparse}
Let $I_0$ be a fixed dyadic interval, $\rr P$ a rank-1 collection of tri-tiles, and $0 \leq \theta_1, \theta_2, \theta_3 <1$ with $\theta_1+\theta_2+\theta_3=1$. It was proved in \cite{vv_BHT} that 
\begin{equation}
\label{eq:localBHT}
\vert \Lambda_{BHT;\rr{P}(I_0)}(f, g, h) \vert \lesssim  \left(\sssize_{I_0} \one_F \right)^{\frac{1+\theta_1}{2}} \cdot \left(\sssize_{I_0} \one_G \right)^{\frac{1+\theta_2}{2}} \cdot \left(\sssize_{I_0} \one_{H'} \right)^{\frac{1+\theta_3}{2}} \cdot \vert I_0 \vert,
\end{equation}
whenever $f, g$ and $h$ are functions having the property that $\vert f(x) \vert \leq \one_F(x), \vert g(x) \vert \leq \one_G(x)$ and $\vert h(x) \vert \leq \one_{H'}(x)$, where $F, G$ and $H'$ are subsets of $\rr R$ of finite measure.

Using exclusively this local estimate, we obtain a sparse domination for the globally-defined trilinear form $\Lambda_{BHT; \rr P}$. More exactly, we will construct a sparse family $\ic S$ of intervals (depending of course on the functions $\one_F, \one_G, \one_{H'}$ and the exponents $\theta_1, \theta_2$ and $\theta_3$) so that 
{\fontsize{10}{10}
\begin{align*}
\label{eq:sparse-BHT}
\big| \Lambda_{BHT; \rr P}(f, g, h) \big|\lesssim \sum_{Q \in \ic S} \big( \frac{1}{|Q|} \int_{\rr R} \one_F\cdot \ci_{Q}^M dx \big)^{\frac{1+\theta_1}{2}}  \big( \frac{1}{|Q|} \int_{\rr R} \one_G \cdot \ci_{Q}^M dx \big)^{\frac{1+\theta_2}{2}}\big( \frac{1}{|Q|} \int_{\rr R} \one_{H'} \cdot \ci_{Q}^M dx \big)^{\frac{1+\theta_3}{2}} \cdot |Q|,
\end{align*}}whenever $f, g, h$ are as above. Since the \emph{restricted weak-type condition} on the functions can be removed (see Proposition \ref{prop:averages-restr-gen} below), we will show in the general context, and for an arbitrary family $\rr P$ of tiles, how to obtain the sparse domination result from the local estimate.  

\begin{theorem}
\label{thm:local->sparse}
Let $\rr P$ be a collection of tiles and assume that 
\begin{equation}
\label{eq:localBHT-non-restr}
\vert \Lambda_{\rr{P}(I_0)}(f, g, h) \vert \lesssim  \left(\sssize_{I_0} \,\vert f \vert^{s_1} \right)^{1/{s_1}} \cdot \left(\sssize_{I_0} \,\vert g \vert^{s_2} \right)^{1/{s_2}} \cdot \left(\sssize_{I_0} \, \vert h\vert^{s_3} \right)^{1/{s_3}} \cdot \vert I_0 \vert,
\end{equation}
holds for any dyadic interval $I_0$ and any locally integrable functions $f, g$ and $h$. Then there exists a sparse family $\ic S$ of intervals so that 
\begin{align*}
\label{eq:sparse-BHT-non-restr}
\big| \Lambda_{\rr P}(f, g, h) \big|\lesssim \sum_{Q \in \ic S} \big( \frac{1}{|Q|} \int_{\rr R} \vert f \vert^{s_1} \cdot \ci_{Q}^{M-1} dx \big)^{\frac{1}{s_1}}  \big( \frac{1}{|Q|} \int_{\rr R} \vert g \vert^{s_2} \cdot \ci_{Q}^{M-1} dx \big)^{\frac{1}{s_2}}\big( \frac{1}{|Q|} \int_{\rr R} \vert h \vert^{s_3} \cdot \ci_{Q}^{M-1} dx \big)^{\frac{1}{s_3}} \cdot |Q|.
\end{align*}

\begin{remark}\begin{itemize}
\item[1.] We note that the scalar-valued functions can be replaced by more general vector-valued norms: a local estimate
{\fontsize{10}{10}
\begin{equation*}
\label{eq:localBHT-non-restr-vv}
\vert \Lambda_{\rr{P}(I_0)}(\vec f, \vec g, \vec h) \vert \lesssim  \left(\sssize_{I_0} \| \vec f(x, \cdot) \|_{X_1}^{s_1} \right)^{1/{s_1}} \cdot \left(\sssize_{I_0} \| \vec g(x, \cdot) \|_{X_2}^{s_2}\right)^{1/{s_2}} \cdot \left(\sssize_{I_0} \|\vec h(x, \cdot) \|_{X_3}^{s_3} \right)^{1/{s_3}} \cdot \vert I_0 \vert
\end{equation*}}
will imply in an analogous way the vector-valued sparse domination
{\fontsize{9}{10}
\begin{align*}
\label{eq:sparse-BHT-non-restr-vv}
\big| \Lambda_{\rr P}(\vec f, \vec g, \vec h) \big|\lesssim \sum_{Q \in \ic S} \big( \frac{1}{|Q|} \int_{\rr R} \| \vec f(x, \cdot) \|_{X_1}^{s_1} \cdot \ci_{Q}^M dx \big)^{1/{s_1}}  \big( \frac{1}{|Q|} \int_{\rr R} \| \vec g(x, \cdot) \|_{X_2}^{s_2} \cdot \ci_{Q}^M dx \big)^{1/{s_2}}\big( \frac{1}{|Q|} \int_{\rr R} \| \vec h(x, \cdot) \|_{X_3}^{s_3} \cdot \ci_{Q}^M dx \big)^{1/{s_3}} \cdot |Q|.
\end{align*}}
\item[2.] Similarly, a local estimate for a (vector-valued) $( n+1)$-linear form will imply a (vector-valued) sparse domination of the $\left( n+1 \right)$-linear form.
\end{itemize}
\end{remark}
\begin{proof}

The sparse estimation does not only involve a localization in space onto a certain interval $I_0$, but also a restriction to a specific subset of tiles $\rr P_{I_0} \subseteq\rr P(I_0)$. We will construct $\ic S$ a sparse family of intervals, where $\ds \ic S =\bigcup_{k \geq 0} \ic S_k$ and for every $k \geq 0$ we have
\[
\ic S_{k+1}=\bigcup_{Q_0 \in \ic S_{k}} ch_{\ic S}(Q_0);
\]
that is, the intervals from the $k+1$ level $\ic S_{k+1}$ are precisely the descendants of the intervals of the $k$-th level $\ic S_k$.

We start by defining $\ic S_0$ as the collection of maximal dyadic intervals $I$ so that $I=I_P$ for some tri-tile $P \in \rr P$.  Then we show how to construct $\ic S_{k+1}$, assuming that $\ic S_k$ has already been constructed: for every $Q_0 \in \ic S_{k}$, the descendants $ch_{\ic S}(Q_0)$ are maximal dyadic intervals $Q \subseteq Q_0$ so that there exists at least one tri-tile $P \in \rr P$ with $I_P \subset Q$ and so that one of the following holds:
\begin{equation}
\label{eq:large-ave-f}
\big( \frac{1}{\vert Q \vert} \int_{\rr R} \vert f(x) \vert^{s_1} \cdot \ci_{Q}^M(x) dx  \big)^{1/s_1} > C \cdot \big( \frac{1}{\vert Q_0 \vert} \int_{\rr R} \vert f(x) \vert^{s_1} \cdot \ci_{Q_0}^{M-1}(x) dx  \big)^{1/s_1} \qquad \text{or}
\end{equation} 
\begin{equation}
\label{eq:large-ave-g}
\big( \frac{1}{\vert Q \vert} \int_{\rr R} \vert g(x) \vert^{s_2} \cdot \ci_{Q}^M(x) dx  \big)^{1/s_2} > C \cdot \big( \frac{1}{\vert Q_0 \vert} \int_{\rr R} \vert g(x) \vert^{s_2} \cdot \ci_{Q_0}^{M-1}(x) dx  \big)^{1/s_2} \qquad \text{or}
\end{equation}  
\begin{equation}
\label{eq:large-ave-h}
\big( \frac{1}{\vert Q \vert} \int_{\rr R} \vert h(x) \vert^{s_3} \cdot \ci_{Q}^M(x) dx  \big)^{1/s_3} > C \cdot \big( \frac{1}{\vert Q_0 \vert} \int_{\rr R} \vert h(x) \vert^{s_3} \cdot \ci_{Q_0}^{M-1}(x) dx  \big)^{1/s_3}.
\end{equation} 

The collection $ch_{\ic S}(Q_0)$ of such intervals is disjoint due to the maximality condition, and it is not difficult to check that
\[
\sum_{Q \in ch_{\ic S}(Q_0)} \vert Q \vert \leq \frac{1}{2}\vert Q_0 \vert,
\]
provided the constant $C$ is chosen to be large enough. Indeed, we can see that all the intervals satisfying \eqref{eq:large-ave-f} are mutually disjoint and they are contained inside the set
\[
\big\lbrace \ic M_{s_1} \big( f \cdot \ci_{Q_0}^{\frac{M-1}{s_1}} \big)(x) > C \cdot \big( \frac{1}{\vert Q_0 \vert} \int_{\rr R} \vert f(x) \vert^{s_1} \cdot \ci_{Q_0}^{M-1}(x) dx  \big)^{1/s_1}  \big\rbrace.
\]

Using the $L^{s_1} \mapsto L^{s_1, \infty}$ boundedness of $\ic M_{s_1}$, we can estimate the measure of the set above by
\begin{align*}
& \vert \big\lbrace \ic M_{s_1} \Big( f \cdot \ci_{Q_0}^{\frac{M-1}{s_1}} \Big)(x) > C \cdot \big( \frac{1}{\vert Q_0 \vert} \int_{\rr R} \vert f(x) \vert^{s_1} \cdot \ci_{Q_0}^{M-1}(x) dx  \big)^{1/s_1}  \big\rbrace \vert \lesssim \\
&\lesssim  C^{-s_1} \cdot \big( \frac{1}{\vert Q_0 \vert} \int_{\rr R} \vert f(x) \vert^{s_1} \cdot \ci_{Q_0}^{M-1}(x) dx \big)^{-1} \cdot \big\|  \ic M_{s_1} \big( f \cdot \ci_{Q_0}^{\frac{M-1}{s_1}}\big) \big\|_{s_1, \infty}^{s_1} \lesssim C^{-s_1} \vert Q_0 \vert.
\end{align*}
This proves the sparse property of the collection $\ic S$; now we need to prove the sparse domination estimate. To this end, for every $Q \in \ic S$ we define $\rr P_Q$ to be the collection of tri-tiles $P \in \rr P(Q)$ so that $I_P \subseteq Q$, but $I_P$ is not contained in any other descendant of $Q$ in $\ic S$. This implies in particular that every such tri-tile satisfies 
\[
\big( \frac{1}{\vert I_P \vert} \int_{\rr R} \vert f(x) \vert^{s_1} \cdot \ci_{I_P}^M(x) dx  \big)^{1/s_1} \leq  C \cdot \big( \frac{1}{\vert Q \vert} \int_{\rr R} \vert f(x) \vert^{s_1} \cdot \ci_{Q}^{M-1}(x) dx  \big)^{1/s_1},
\]
and similarly for $g$ and $h$. We obtain in this way
\[
\big(\sssize_{\rr P_{Q}} \vert f \vert^{s_1} \big)^{1/{s_1}} \lesssim C \cdot \big( \frac{1}{\vert Q \vert} \int_{\rr R} \vert f(x) \vert^{s_1} \cdot \ci_{Q}^{M-1}(x) dx  \big)^{1/s_1},
\]
and likewise, 
{\fontsize{9}{9}
\[
\big(\sssize_{\rr P_{Q}} \vert g \vert^{s_2} \big)^{1/{s_2}} \lesssim C \cdot \big( \frac{1}{\vert Q \vert} \int_{\rr R} \vert g(x) \vert^{s_2} \cdot \ci_{Q}^{M-1}(x) dx  \big)^{1/s_2}, \quad \big(\sssize_{\rr P_{Q}} \vert h \vert^{s_3} \big)^{1/{s_3}} \lesssim C \cdot \big( \frac{1}{\vert Q \vert} \int_{\rr R} \vert h(x) \vert^{s_3} \cdot \ci_{Q}^{M-1}(x) dx  \big)^{1/s_3}.
\]}

Noting that $\ds \rr P=\bigcup_{Q \in \ic S} \rr P_Q$, we deduce the sparse domination of the trilinear form.

\end{proof}
\end{theorem}

Eventually, the sparse domination we want to obtain doesn't concern the multilinear form, but rather an expression of the form $\ds \big\|  \|T( \vec f_1, \ldots, \vec f_n)\|_{X'_{n+1}} \cdot v \big\|_q^q$ (estimate \eqref{eq:sparse-T_k-all-tau} of Theorem \ref{thm:main-thm}), where $T$ is an $n$-(sub)linear operator, $\vec f_1, \ldots, \vec f_n$ are vector-valued functions and $v$ is a locally $q$-integrable function, and $0< q< \infty$. In this case too, the local estimates imply a sparse domination, and \emph{subadditivity} is essential.

\begin{proposition}
\label{prop:local->sparse->subadditive}
Let $\| \cdot \|_{X_{n+1}}$ be a (quasi-)norm so that $ \| \cdot \|_{X_{n+1}}^q$ is subadditive for some $q>0$. Let $T$ be a $n$-(sub)linear operator determined by a collection $\rr P$ of multi-tiles, which satisfies the multiple vector-valued local estimate: there exist $s_1, \ldots, s_{n+1} \in (0, \infty)$ so that for any dyadic interval $I_0$,
\[
\big\|  \|T_{\rr P(I_0)}( \vec f_1, \ldots, \vec f_n)\|_{X'_{n+1}} \cdot v \big\|_{q}^q \lesssim \prod_{j=1}^n  \big( \sssize_{\rr P(I_0)}^{s_j} \| \vec f_j (x, \cdot) \|_{X_j}\big)^q \cdot \big( \sssize_{\rr P(I_0)}^{s_{n+1}} v  \big)^q \cdot |I_0|.
\]
Then there exists a sparse family $\ic S$ of dyadic intervals, depending on the functions $\vec f_1, \ldots, \vec f_n, v$ and the Lebesgue exponents $s_1, \ldots , s_{n+1}, q$ so that 
{\fontsize{10}{10}\begin{align*}
\big\|  \|T_{\rr P}( \vec f_1, \ldots, \vec f_n)\|_{X'_{n+1}} \cdot v \big\|_{q}^q \lesssim \sum_{Q \in \ic S} \prod_{j=1}^n \big( \frac{1}{|Q|} \int_{\rr R} \| \vec f_j (x, \cdot) \|^{s_j}_{X_j} \cdot \ci_{Q}^{M-1} dx \big)^{\frac{q}{s_j}} \cdot  \big( \frac{1}{|Q|} \int_{\rr R} \vert v \vert^{s_{n+1}} \cdot \ci_{Q}^{M-1} dx \big)^{\frac{q}{s_{n+1}}} \cdot |Q|.
\end{align*}}
\begin{proof}
The stopping time is identical to the one in the proof of Theorem \ref{thm:local->sparse}. Once the collection $\ic S$ is defined in an analogous way, and  once we have defined for every $Q \in \ic S$ the collection of descendants $ch_{\ic S}(Q)$ and the collection of tiles $\rr P_{Q}$, we only need to notice that
\begin{align*}
&\big\|  \|T_{\rr P}( \vec f_1, \ldots, \vec f_n)\|_{X'_{n+1}} \cdot v \big\|_{q}^q \lesssim \sum_{Q \in \ic S} \big\|  \|T_{\rr P_Q}( \vec f_1, \ldots, \vec f_n)\|_{X'_{n+1}} \cdot v \big\|_{q}^q \\
& \lesssim \prod_{j=1}^n  \big( \sssize_{\rr P_Q}^{s_j} \| \vec f_j (x, \cdot) \|_{X_j}\big)^q \cdot \big( \sssize_{\rr P_Q}^{s_{n+1}} v  \big)^q \cdot |Q| \\
&\lesssim \sum_{Q \in \ic S} \prod_{j=1}^n \big( \frac{1}{|Q|} \int_{\rr R} \| \vec f_j (x, \cdot) \|^{s_j}_{X_j} \cdot \ci_{Q}^{M-1} dx \big)^{\frac{q}{s_j}} \cdot  \big( \frac{1}{|Q|} \int_{\rr R} \vert v \vert^{s_{n+1}} \cdot \ci_{Q}^{M-1} dx \big)^{\frac{q}{s_{n+1}}} \cdot |Q|.
\end{align*}

\end{proof}
\end{proposition}

\subsection{From restricted-type to general functions}
\label{sec:restricted-type-to-general}
Here we present a technical lemma, which allows to pass from local estimates for restricted-type functions to local estimates for general functions.

Let $I_0$ be a fixed dyadic interval and $\rr P$ a finite collection of tritiles; recall the notations 
\[
\rr P(I_0):=\lbrace P \in \rr P : I_P \subseteq I_0 \rbrace \quad \text{and} \quad \rr P (I_0)^+:= \rr P (I_0) \cup P_{I_0}, 
\]
where $P_{I_0}$ is some tritile (not necessarily contained in $\rr P$)  with spatial interval $I_0$.

\begin{proposition}[Mock interpolation]
\label{prop:averages-restr-gen}
Assume $\Lambda$ is a trilinear form associated to the collection $\rr P$, satisfying
{\fontsize{10}{10}\begin{align*}
&\big| \Lambda_{\rr P\left( I_0 \right)}(f, g, h)\big| \lesssim \\
&\sup_{P \in \rr P \left( I_0 \right)^+}\Big( \frac{1}{|I_P|} \int_{\rr R} \one_F \cdot \ci_{I_P} dx  \Big)^{\alpha_1} \cdot  \sup_{P \in \rr P \left( I_0 \right)^+}\Big( \frac{1}{|I_P|} \int_{\rr R} \one_G \cdot \ci_{I_P} dx  \Big)^{\alpha_2} \cdot  \sup_{P \in \rr P \left( I_0 \right)^+}\Big( \frac{1}{|I_P|} \int_{\rr R} \one_{H'} \cdot \ci_{I_P} dx  \Big)^{\alpha_3} \cdot |I_0| \nonumber,
\end{align*}}for any  $F, G$ and $H'$ sets of finite measure and any functions $f, g, h$ such that $|f(x)| \leq \one_F(x), |g(x)| \leq \one_G(x)$ and $|h(x)| \leq \one_{H'}(x)$. Then for any triple $\left( p_1, p_2, p_3 \right)$ satisfying $p_i \alpha_i >1$, and any functions $f, g$ and $h$ locally integrable, we have{\fontsize{10}{10}
\begin{align*}
&\big| \Lambda_{\rr P\left( I_0 \right)}(f, g, h)\big| \lesssim \\
&\sup_{P \in \rr P \left( I_0 \right)^+}\Big( \frac{1}{|I_P|} \int_{\rr R} \big|f \cdot \ci_{I_P}\big|^{p_1} dx  \Big)^{\frac{1}{p_1}} \cdot  \sup_{P \in \rr P \left( I_0 \right)^+}\Big( \frac{1}{|I_P|} \int_{\rr R} \big| g \cdot \ci_{I_P} \big|^{p_2} dx  \Big)^{\frac{1}{p_2}} \cdot  \sup_{P \in \rr P \left( I_0 \right)^+}\Big( \frac{1}{|I_P|} \int_{\rr R} \big| h \cdot \ci_{I_P}\big|^{p_3} dx  \Big)^{\frac{1}{p_3}} \cdot |I_0| \nonumber.
\end{align*}}
The implicit constants in both inequalities above may depend on the $\alpha_i$ and $p_i$, but are independent of the collection of tritiles and of the interval $I_0$.
\begin{proof}

Any function can be decomposed according to its level sets; more exactly, we have
\begin{equation}
\label{eq:dec-f}
f=\sum_{k_1} f \cdot \one_{\left\lbrace 2^{k_1-1} \leq \vert f \vert \leq 2^{k_1}\right\rbrace}:=\sum_{k_1} 2^{k_1} f_{k_1}:=\sum_{k_1} 2^{k_1} f_{k_1} \cdot \one_{F_{k_1}}.
\end{equation}
The above notation means that $F_{k_1}=\lbrace 2^{k_1-1} \leq \vert f \vert \leq 2^{k_1} \rbrace$, and we note that the functions $f_{k_1}$ satisfy $\vert f_{k_1}(x)\vert \leq   \one_{F_{k_1}}(x)$.

Likewise, $g(x)=\sum_{k_2} 2^{k_2} g_{k_2} \cdot \one_{G_{k_2}}$ and $h(x)=\sum_{k_3} 2^{k_3} h_{k_2} \cdot \one_{H_{k_3}}$.

We estimate the trilinear form, using the multilinearity property and the hypotheses:
\begin{align*}
&\big| \Lambda_{\rr P \left( I_0 \right)}(f, g, h) \big| \leq \sum_{k_1, k_2, k_3}2^{k_1}2^{k_2}2^{k_3} \big|  \Lambda_{\rr P \left( I_0 \right)} (f_{k_1}, g_{k_2}, h_{k_3}) \big|  \\
&\lesssim \sum_{k_1, k_2, k_3}2^{k_1}2^{k_2}2^{k_3} \sup_{P \in \rr P \left( I_0 \right)^+}\big( \frac{1}{|I_P|} \int_{\rr R} \one_{F_{k_1}} \cdot \ci_{I_P} dx  \big)^{\alpha_1} \\
& \cdot  \sup_{P \in \rr P \left( I_0 \right)^+}\big( \frac{1}{|I_P|} \int_{\rr R} \one_{G_{k_2}} \cdot \ci_{I_P} dx  \big)^{\alpha_2} \cdot  \sup_{P \in \rr P \left( I_0 \right)^+}\big( \frac{1}{|I_P|} \int_{\rr R} \one_{H_{k_3}} \cdot \ci_{I_P} dx  \big)^{\alpha_3} \cdot |I_0|.
\end{align*}

Now we are left with proving
\[
\sum_{k} 2^k  \sup_{P \in \rr P \left( I_0 \right)^+}\big( \frac{1}{|I_P|} \int_{\rr R} \one_{F_{k}} \cdot \ci_{I_P} dx  \big)^{\alpha_1} \lesssim \sup_{P \in \rr P \left( I_0 \right)^+}\big( \frac{1}{|I_P|} \int_{\rr R} \big| f \cdot \ci_{I_P} \big|^{p_1} dx  \big)^{\frac{1}{p_1}}.
\]

Each of the expressions on the left hand side are $L^{\frac{1}{\alpha_1}}$ maximal averages, and they are summable at the expense of loosing some information; in the end, we will have an $L^{p_1}$ maximal average, with $p_1>\frac{1}{\alpha_1}$.

Let $k_0$ be so that 
\[
\sup_{P \in \rr P \left( I_0 \right)^+}\big( \frac{1}{|I_P|} \int_{\rr R} \big| f \cdot \ci_{I_P} \big|^{p_1} dx  \big)^{\frac{1}{p_1}} \sim 2^{k_0}.
\]

We note the following: since $ 2^{k-1} \cdot \one_{F_k} \leq |f(x)|  \cdot \one_{F_k}(x) \leq  2^{k}\cdot \one_{F_k} $, we have $\one_{F_k}(x) \lesssim 2^{-kp_1} |f(x)|^{p_1}$ and hence
\begin{equation}
\label{eq:ineq-restr-sup}
\frac{1}{\vert I_P \vert} \int_{\rr R }\one_{F_k} \cdot \ci_{I_P} dx  \lesssim 2^{-k p_1} \frac{1}{\vert I_P \vert} \int_{\rr R }\big| f(x) \cdot \ci_{I_P}\big|^{p_1} dx \lesssim 2^{\left( k_0 -k \right)p_1}. 
\end{equation}

Here we need to be more cautious; the adapted weights $\ci_{I_P}$ appearing on the left and right hand side above are not exactly the same, but one of them is a power of the other. Nevertheless, the wave packets associated to the tritiles in $\rr P$ can have arbitrary decay, and for that reason we will not worry about the possibly different $\ci_{I_P}$.

On the other hand, we have the trivial inequality 
\[
\frac{1}{\vert I_P \vert} \int_{\rr R }\one_{F_k} \cdot \ci_{I_P} dx \lesssim 1.
\]
Taking the sup over $P \in \rr P (I_0)^+$ in the expression above, as well as in \eqref{eq:ineq-restr-sup}, we obtain that 
\[
\sup_{P \in \rr P \left( I_0 \right)^+}\big( \frac{1}{|I_P|} \int_{\rr R} \one_{F_{k}} \cdot \ci_{I_P} dx  \big)^{\alpha_1} \lesssim \min \big( 1, 2^{\left( k_0 -k \right)p_1\alpha_1} \big).
\]
Then we are ready for the final estimate:
\begin{align*}
&\sum_{k} 2^k  \sup_{P \in \rr P \left( I_0 \right)^+}\big( \frac{1}{|I_P|} \int_{\rr R} \one_{F_{k}} \cdot \ci_{I_P} dx  \big)^{\alpha_1}\\
&  \lesssim \sum_{k \leq k_0} 2^k  \sup_{P \in \rr P \left( I_0 \right)^+}\big( \frac{1}{|I_P|} \int_{\rr R} \one_{F_{k}} \cdot \ci_{I_P} dx  \big)^{\alpha_1} + \sum_{k >k_0} 2^k  \sup_{P \in \rr P \left( I_0 \right)^+}\big( \frac{1}{|I_P|} \int_{\rr R} \one_{F_{k}} \cdot \ci_{I_P} dx  \big)^{\alpha_1}\\
&\lesssim \sum_{k \leq k_0} 2^k + \sum_{k >k_0} 2^k 2^{\left( k_0 -k \right)p_1\alpha_1} \lesssim 2^{k_0}+2^{k_0} \sum_{k>k_0} 2^{\left(k-k_0\right)\left( 1- p_1 \alpha_1 \right)} \lesssim 2^{k_0}.
\end{align*}
provided $1-p_1\alpha_1 <0$.
\end{proof}
\end{proposition}

In our application to $BHT$, we will have $\alpha_j=\frac{1+\theta_j}{2}$ and $\frac{1}{p_j}=\frac{1+\theta_j}{2} -\epsilon <\alpha_j$, hence the conditions in Proposition \ref{prop:averages-restr-gen} are satisfied.

Using the notation in \eqref{eq:def-L^p-size}, the above proposition and \eqref{eq:localBHT} imply the following local estimate for the model operator of the bilinear Hilbert transform:
\begin{proposition}
\label{prop:sizes-fns-gen}
If $I_0$ is a fixed dyadic interval and $\Lambda_{\rr P(I_0)}$ is a model trilinear form associated to the $BHT$ operator, we have
\[
\big| \Lambda_{\rr P \left( I_0 \right)} (f, g, h) \big| \lesssim  \left(\sssize_{\rr P\left( I_0 \right)}^{p_1} f  \right) \cdot \left(\sssize_{\rr P\left( I_0 \right)}^{p_2} g  \right) \cdot \left(\sssize_{\rr P\left( I_0 \right)}^{p_3} h  \right) \cdot |I_0|,
\]
where $\frac{1}{p_j}=\frac{1+\theta_j}{2} -\epsilon$, for any $0 \leq \theta_1, \theta_2, \theta_3 <1$ with $ \theta_1+ \theta_2+\theta_3 =1$ and any $\epsilon >0$ small enough.
\end{proposition}

\begin{remark}
The result of Proposition \ref{prop:averages-restr-gen} can be easily extended to a more general setting: if $T$ is a vector-valued operator so that $\| \cdot  \|_{L^{R'_{n+1}}}^q$ and $\| \cdot \|_q^q$ are subadditive, for which we have
\begin{align*}
 \big\| \| T_{\rr P(I_0)}(\vec f_1, \ldots, \vec f_{n}) \|_{L^{R'_{n+1}}} \cdot \one_{\tilde E_{n+1}} \big\|_q^q  \lesssim \prod_{j=1}^{n} \big( \sssize_{I_0} \one_{E_j}    \big)^{\alpha_j} \cdot \big( \sssize_{I_0} \one_{\tilde E_{n+1}}    \big)^{\alpha_{n+1}} \cdot \vert I_0 \vert
\end{align*}
for any sets $E_1, \ldots, E_n, \tilde E_{n+1}$ of finite measure, any vector-valued functions $\vec f_1, \ldots, \vec f_{n}$ such that $\|\vec f_j (x)\|_{X_j} \leq \one_{E_j}$ for all $1 \leq j \leq n$, and $v$ a locally $q$-integrable function so that $|v(x)| \leq \one_{\tilde E_{n+1}}$ then also
{\fontsize{10}{10}\begin{align*}
\big\|  \| T_{\rr P(I_0)}(\vec f_1, \ldots, \vec f_{n})\|_{L^{R'_{n+1}}} \cdot v\|_q^q  \lesssim \prod_{j=1}^{n} \big( \sssize_{I_0}^{p_j} \| \vec f_j  \|_{X_j}  \big)^q \cdot  \big( \sssize_{I_0}^{p_{n+1}} (v)  \big)^q  \cdot \vert I_0 \vert,
\end{align*}}
for any vector-valued functions $\vec f_1, \ldots, \vec f_{n}$, any locally $q$-integrable function $v$, and any Lebesgue exponents $p_1, \ldots, p_{n+1}$ satisfying for all $1 \leq j \leq n+1$,  $p_j > \dfrac{q}{\alpha_j}$.

\end{remark}

In order to obtain the localized estimates for multiple vector-valued extensions, we make use of the helicoidal method, as presented in the following Section \ref{sec:review-hel-mthd}.

\subsection{Additional details on sparse domination via the helicoidal method}
\label{subsec:additional:details}

With the purpose of making clear and comprehensible the ideas behind our method, we take another look at the $BHT$ example. We've seen in Section \ref{sec:stopping-time-local->sparse} that sparse domination for restricted-type functions is implied by the fundamental local estimate from \cite{vv_BHT}:
\begin{equation}
\label{eq:fundam-local-est}
\vert \Lambda_{\rr P(I_0)} (f_1, f_2, f_3)  \vert \lesssim \prod_{j=1}^3 \big( \sssize_{I_0} \, \one_{E_j} \big)^\frac{1+\theta_j}{2} \cdot |I_0|
\end{equation}
which holds for all fixed dyadic intervals $I_0 \subseteq \rr R$ and all functions satisfying $\vert f_j  \vert \leq \one_{E_j}$ with $E_j$ measurable subsets of the real line having finite measure.

This yields a sparse collection $\ic S$ of dyadic intervals (depending on the functions $f_1, f_2$ and $f_3$), for which
\begin{equation}
\label{eq:fundam-sparse}
\vert  \Lambda (f_1, f_2, f_3)  \vert \lesssim \sum_{Q \in \ic S} \prod_{j=1}^3 \big( ave_{Q}^{p_j} \one_{E_j} \big) \cdot |Q|,
\end{equation}
where $\frac{1}{p_j}=\frac{1+\theta_j}{2}$ for some $0 \leq \theta_1, \theta_2, \theta_3<1$ with $\theta_1+\theta_2+\theta_3=1$. 

Conversely, assuming \eqref{eq:fundam-sparse}, one can observe that 
\begin{align*}
\vert \Lambda_{\rr P (I_0)} (f_1, f_2, f_3) \vert \lesssim \sum_{\substack{Q \in \ic S \\ Q \subseteq I_0}} \prod_{j=1}^3 \big( ave_{Q}^{p_j} \one_{E_j} \big)\cdot |Q| \lesssim  \prod_{j=1}^3 \big( \sssize_{I_0} \, \one_{E_j} \big)^\frac{1+\theta_j}{2} \cdot |I_0|.
\end{align*}

This is an easy consequence of the sparseness property of the collection $\ic S$: $\ds  \sum_{\substack{Q \in \ic S \\ Q \subseteq I_0}}|Q| \lesssim |I_0|$.

In other words, the two estimates \eqref{eq:fundam-local-est} and \eqref{eq:fundam-sparse} are essentially equivalent to each other. In particular, in order to prove \eqref{eq:fundam-sparse} for arbitrary functions, all one has to do is prove the local inequality \eqref{eq:fundam-local-est} for arbitrary functions. Namely, it suffices to prove
\begin{equation}
\label{eq:fundam-local-gen-fnc}
\vert \Lambda_{\rr P(I_0)} (f_1, f_2, f_3)  \vert \lesssim \prod_{j=1}^3 \big( \sssize^{p_j}_{I_0} \, f_j \big) \cdot |I_0|.
\end{equation}

Let us recall now how \eqref{eq:fundam-local-est} has been proved in \cite{vv_BHT}: start with the generic ``size and energy" estimate from \cite{biest}
\begin{equation}
\label{eq:fund-gen-size-energy}
\vert \Lambda_{\rr P} (f_1, f_2, f_3) \vert \lesssim \prod_{j=1}^3 \big( \ssize_{\rr P} \, f_j \big)^{\theta_j} \cdot \big( \eenergy_{\rr P}\, f_j  \big)^{1-\theta_j} ,
\end{equation}
valid for all rank-$1$ collections of tiles $\rr P$ and $0 \leq \theta_1, \theta_2, \theta_3 <1$ so that $\theta_1+\theta_2+\theta_3=1$. Above, $\ds \ssize_{\rr P} \, f$ is a maximal square function which satisfies $\ds \ssize_{\rr P} \, f \lesssim \sssize_{\rr P} \, f$, while $\eenergy_{\rr P} f \lesssim \|f\|_2$ for any function $f$ nice enough (see Definition 6.1 in \cite{biest}).

Subsequently, localize \eqref{eq:fund-gen-size-energy} to a fixed dyadic interval $I_0$ to get 
\begin{equation}
\label{eq:fund-gen-size-energy-local}
\vert \Lambda_{\rr P (I_0)} (f_1, f_2, f_3) \vert \lesssim \prod_{j=1}^3 \big( \sssize_{\rr P(I_0)} \, f_j \big)^{\theta_j} \cdot \big( \eenergy_{\rr P(I_0)}\, f_j  \big)^{1-\theta_j}.
\end{equation}

Hence the estimate \eqref{eq:fundam-local-est}, which is central to proving vector-valued and sparse estimates, follows from the localized ``size and energy" inequality \eqref{eq:fund-gen-size-energy-local} above, once we prove for all $1 \leq j \leq 3$ and all functions $f_j$ satisfying $|f_j| \leq \one_{E_j}$ that 
\begin{equation}
\label{eq:fund-eq-six}
\big( \sssize_{\rr P(I_0)} \, f_j \big)^{\theta_j} \cdot \big( \eenergy_{\rr P(I_0)}\, f_j  \big)^{1-\theta_j} \lesssim \big( \sssize^{p_j}_{I_0} \, \one_{E_j} \big)^\frac{1+\theta_j}{2} \cdot |I_0|^\frac{1-\theta_j}{2}.
\end{equation}

This follows from the fact that the $\eenergy _{\rr P}$ localizes well:
\[
\eenergy_{\rr P(I_0)} \, f_j \lesssim \|f_j \cdot \ci_{I_0}\|_{L^2} \lesssim \| \one_{E_j} \cdot \ci_{I_0}\|_{L^2} = \big( \ave_{I_0} \one_{E_j}  \big)^\frac{1}{2} \cdot \vert I_0 \vert^\frac{1}{2},  
\]
an estimate that had already appeared in a somehow different form in \cite{multilinearMTT}, and which we prove later on in Lemma \ref{lemma:dec-lemma}.

Returning now to \eqref{eq:fundam-local-gen-fnc}, the local estimate for general functions, two comments are in order. First, if \eqref{eq:fundam-local-gen-fnc} is true for the indices $\ds p_j:= \frac{2}{1+\theta_j}$, then it also holds for $\ds p_j > \frac{2}{1+\theta_j}$, as an application of H\"older's inequality. And second, in order to prove sparse domination, one does not need \eqref{eq:fundam-local-gen-fnc} in its full generality, but only in the particular case when 
\begin{equation}
\label{eq:fundam-maximizer-size}
\ave_{I_P}^{p_j} f_j \lesssim \ave_{I_0}^{p_j} f_j, \quad \text{for all } P \in \rr P(I_0), \quad \text{i.e.} \quad \sssize_{I_0}^{p_j} f_j \lesssim \ave_{I_0}^{p_j} f_j
\end{equation}
for all $1 \leq j \leq 3$. This is a simple consequence of the earlier stopping time argument. The reader familiar with the terminology of outer measures of \cite{outer_measures} would recognize that 
\[
\ssize_{\rr P} f_j \simeq \| F_j  \|_{\ic L^\infty}, \quad \text{while} \quad \eenergy_{\rr P} \, f_j \simeq \| F_j  \|_{\ic L^{2,\infty}}
\]
when $F_j : \rr P \to \rr C$ is defined by $\ds F_j(P):= \langle f_j, P \rangle$ for all $P \in \rr P$. The outer measure spaces $\ic L^q$ (for $2<q<\infty$) of \cite{outer_measures} are defined precisely so as to generalize this correspondence.

In particular, the quantity $\ds \big( \ssize_{\rr P(I_0)} \, f_j \big)^{\theta_j} \cdot \big( \eenergy_{\rr P(I_0)}\, f_j  \big)^{1-\theta_j}$ of \eqref{eq:fund-gen-size-energy} interpolates naturally between these two spaces and because of this we will denote it by $\ds \| F_j \|_{\ic L^{q_j}_{mock}}$, where $\frac{1}{q_j}:=\frac{1-\theta_j}{2}+\frac{\theta_j}{\infty}$. Then the original size and energy estimate \eqref{eq:fund-gen-size-energy} can be written as 
\begin{equation}
\label{eq:fund-gen-size-energy-outer-mock}
\vert \Lambda_{\rr P} (f_1, f_2, f_3) \vert \lesssim \prod_{j=1}^3  \| F_j \|_{\ic L^{q_j}_{mock}},
\end{equation}
while \eqref{eq:fund-eq-six} becomes, for all $1 \leq j \leq 3$ and all $p_j > \frac{2}{1+\theta_j}$,
\begin{equation}
\label{eq:localized-mock-norms}
\| F_j \|_{\ic L^{q_j}_{mock}; I_0} \lesssim \big( \sssize_{\rr P(I_0)}^{p_j} \one_{E_j}  \big)  \cdot \vert I_0 \vert^\frac{1}{q_j}.
\end{equation} 

Equivalently, to prove \eqref{eq:fundam-local-est}, one could have used the estimate
\begin{equation*}
\vert \Lambda_{\rr P}(f_1, f_2, f_3)   \vert \lesssim \prod_{j=1}^3 \| F_j \|_{\ic L^{q_j}}
\end{equation*}
from \cite{outer_measures} instead of \eqref{eq:fund-gen-size-energy-outer-mock}, and also the analogue of \eqref{eq:localized-mock-norms}
\begin{equation}
\label{eq:localized-outer-norms}
\| F_j \|_{\ic L^{q_j}; \rr P( I_0)} \lesssim \big( \sssize_{\rr P(I_0)}^{p_j} \one_{E_j}  \big)  \cdot \vert I_0 \vert^\frac{1}{q_j}
\end{equation}
for $1 \leq j \leq 3$.

Note that \eqref{eq:localized-outer-norms} is a consequence of \eqref{eq:localized-mock-norms} from \cite{vv_BHT}, since 
\[
\| F_j \|_{\ic L^{q_j}; \rr P} \lesssim \| F_j \|_{\ic L^{q_j}_{mock} ; \rr P}.
\]
A brief proof of this inequality can be found in Proposition 4 of \cite{expository-hel}.

To sum up, sparse domination for arbitrary functions would follow from inequality \eqref{eq:localized-outer-norms} under the assumption \eqref{eq:fundam-maximizer-size}, namely that
\begin{equation}
\label{eq:fundam-stopping-fit}
\| F_j \|_{\ic L^{q_j}; \rr P( I_0)} \lesssim \big( \ave_{\rr P(I_0)}^{p_j} f_j  \big)  \cdot \vert I_0 \vert^\frac{1}{q_j}.
\end{equation} 

The proof of \eqref{eq:fundam-stopping-fit} for general functions has been worked out in \cite{weighted_BHT}, Proposition 4.1. As explained before, this implies sparse domination as in \eqref{eq:fundam-sparse} for indices $p_j$ that satisfy $p_j > \frac{2}{1+\theta_j}$ for all $1 \leq j \leq 3$.
Our main observation is that \eqref{eq:fundam-local-gen-fnc} follows from the original one \eqref{eq:fundam-local-est} via a very short interpolation argument. This fact has very important consequences; in particular, it allows us to prove sparse domination also in the multiple vector-valued case essentially without any extra effort, since it is known form \cite{vv_BHT} that \eqref{eq:fundam-local-est} is true in this case as well.  For more connections with the theory of outer measures, the reader is referred to Section 2.3 of \cite{expository-hel}.

\section{Vector-valued estimates via the helicoidal method: a review}
\label{sec:review-hel-mthd}

Although we present the helicoidal method in the special context of the bilinear Hilbert transform operator, it generalizes to many other operators whose multilinear form can be represented as in \eqref{eq:representation}.

\subsection{The stopping time for vector-valued estimates}
\label{sec:stopping-time-helicoidal-method}
Here we recall some technical aspects from \cite{vv_BHT}, drawing attention to the stopping times involved in proving the multiple vector-valued estimates, which will be later compared to the stopping times used for proving sparse domination (see the following Section \ref{sec:conclusions-comp}). The goal is to show that $BHT:L^p(\ell^{r_1}) \times L^q(\ell^{r_2}) \to L^s(\ell^{r})$, where $(p, q, s)$ and $(r_1, r_2, r)$ are H\"older tuples, with $r \geq 1$. Since we want to illustrate how the local estimate implies the vector-valued result, we turn away form the slightly more technical cases $r_1=\infty, r_2=\infty$ or $r<1$. Hence, we want to show 
\[
\big\| \big( \sum_k | BHT(f_k, g_k) |^r \big)^{\frac{1}{r}}\big\|_s \lesssim \big\| \big( \sum_k |  f_k |^{r_1}   \big)^{\frac{1}{r_1}} \big\|_p \, \big\| \big( \sum_k | g_k |^{r_2}   \big)^{\frac{1}{r_2}} \big\|_q.
\]

By vector-valued restricted weak type interpolation (see, for example \cite{vv_BHT}), it is sufficient to prove the following statement: 

\emph{ for any given $F, G$ and $H$ sets of finite measure, there exists $H' \subseteq H$ major subset such that for any vector-valued functions $\vec f =\lbrace f_k  \rbrace, \vec g =\lbrace g_k  \rbrace$ and $\vec h =\lbrace h_k  \rbrace$ satisfying 
\[
\big(  \sum_k \vert f_k(x)\vert^{r_1}\big)^{\frac{1}{r_1}} \leq \one_F, \quad \big(  \sum_k \vert g_k(x)\vert^{r_1}\big)^{\frac{1}{r_2}} \leq \one_G, \quad  \big(  \sum_k \vert h_k(x)\vert^{r'}\big)^{\frac{1}{r'}} \leq \one_{H'}, \quad 
\]
we have 
\begin{equation}
\label{eq:restr-weak-type-1}
\vert \sum_k \Lambda_{BHT; \rr P}(f_k, g_k, h_k)   \vert \lesssim \vert F \vert^{\alpha_1} \cdot \vert G \vert^{\alpha_2} \cdot \vert H \vert^{\alpha_3},
\end{equation}
for $(\alpha_1, \alpha_1, \alpha_3)$ a tuple satisfying $\alpha_1+\alpha_2+\alpha_3=1$, arbitrarily close to $(\frac{1}{p}, \frac{1}{q}, \frac{1}{s'})$.
}

Most often, the major subset $H'$ is the part of $H$ where $\vec f, \vec g$ are under control: if 
\[
\Omega:= \big\lbrace x: \ic M \one_F(x) > C \,\frac{\vert F \vert}{ \vert H \vert} , \, \ic M \one_G(x) > C\,\frac{\vert G \vert}{ \vert H \vert}  \big\rbrace,
\]
then we set $H':=H \setminus \Omega$.

Upon obtaining a proper local estimate, a triple stopping time will be performed, according to the sizes of the functions. For this reason we assume that the tiles satisfy, for $d \geq 0$,
\[
1+\frac{\dist(I_P, \Omega^c)}{\vert I_P \vert} \sim 2^d,
\]
and we need to obtain a certain decay $2^{-10 d}$ in the restricted-type inequality \eqref{eq:restr-weak-type-1}.

The next step consists in proving, for $1<r_1, r_2, r'< \infty$, the local estimate
\begin{equation}
\label{eq:bht-local-fns}
\big|  \Lambda_{\rr P \left( I_0 \right)}(f_k \cdot \one_F, g_k \cdot \one_G, h_k \cdot \one_{H'})  \big| \lesssim \big\| \Lambda^{F, G, H'}_{\rr P \left( I_0 \right)} \big\|\|f_k \cdot \ci_{I_0}^M \|_{r_1} \|g_k \cdot \ci_{I_0}^M \|_{r_2} \|h_k \cdot \ci_{I_0}^M \|_{r'}, 
\end{equation}
where $\big\| \Lambda^{F, G, H'}_{\rr P \left( I_0 \right)} \big\| $ represents the operatorial norm of the localized trilinear form where the extra information is retained, and it equals
\[
 \big\| \Lambda^{F, G, H'}_{\rr P \left( I_0 \right)} \big\|=\left( \sssize_{\rr P \left(I_0\right)} \one_F  \right)^{\frac{1+\theta_1}{2}-\frac{1}{r_1}-\epsilon} \left( \sssize_{\rr P \left(I_0\right)}  \one_G \right)^{\frac{1+\theta_2}{2}-\frac{1}{r_2}-\epsilon} \left( \sssize_{\rr P \left(I_0\right)} \one_{H'}  \right)^{\frac{1+\theta_3}{2}-\frac{1}{r'}-\epsilon} .
\]
Then we use H\"older to recover the Lebesgue norms of $\one_F, \one_G, \one_{H'}$. If we ignored the operatorial norm $\big\| \Lambda^{F, G, H'}_{\rr P \left( I_0 \right)} \big\| $,  we could only retrieve $|F|^{\frac{1}{r_1}} \cdot |G|^{\frac{1}{r_2}} \cdot |H|^{\frac{1}{r'}}$, which can be very different from the desired expression $ |F|^{\frac{1}{p}} |G|^{\frac{1}{q}} |H|^{\frac{1}{s'}}$. So in order to obtain all the possible vector-valued estimates, we had to take into account $\big\| \Lambda^{F, G, H'}_{\rr P \left( I_0 \right)} \big\| $, and moreover, to secure largest possible exponents for the sizes.

Overall, the constraint we obtained for $r_1, r_2, r'$ and $p, q, s'$, reduces to the existence of $0\leq\theta_1, \theta_2, \theta_3<1$ with $\theta_1+ \theta_2+\theta_3=1$, so that we have simultaneously
\[
\frac{1+\theta_1}{2}>\max \left( \frac{1}{r_1}, \frac{1}{p} \right), \qquad \frac{1+\theta_2}{2}>\max \left( \frac{1}{r_2}, \frac{1}{q} \right) \quad \text{and} \quad \frac{1+\theta_3}{2}>\max \left( \frac{1}{r'}, \frac{1}{s'} \right).
\]

In order to prove \eqref{eq:bht-local-fns}, we use restricted-type interpolation and the local estimate \eqref{eq:bht-local-max-op} that has already appeared in the introduction. Assuming that $\vert f_k(x) \vert \leq \one_{E_1}, \vert g_k(x) \vert \leq \one_{E_2}$ and $\vert h_k(x) \vert \leq \one_{E'_3}$ and that $I \subseteq I_0$ is a fixed dyadic interval, we have
{\fontsize{10}{10}\[
\vert \Lambda_{\rr{P}(I)}(f_k \cdot \one_F, g_k \cdot \one_G, h_k \cdot \one_{H'}) \vert \lesssim  \left(\sssize_{I_0} \one_F \cdot \one_{E_1} \right)^{\frac{1+\theta_1}{2}-\epsilon} \cdot \left(\sssize_{I_0} \one_G \cdot \one_{E_2} \right)^{\frac{1+\theta_2}{2} -\epsilon} \cdot \left(\sssize_{I_0} \one_{H'} \cdot \one_{E'_3} \right)^{\frac{1+\theta_3}{2}-\epsilon} \cdot \vert I \vert.
\]}

Through a triple stopping time, we can recover the $L^{r_1}, L^{r_2}$ and $L^{r'}$ norms of the functions $f_k, g_k$ and $h_k$ respectively, while $\big\| \Lambda_{\rr P(I_0)}^{F,G, H'}  \big\|$ represents a ``surplus" that will be used later on, when deciding the range of exponents $p, q, s$. This last part in which we retrieve the $L^{r_1}$ norm of $f_k$ and $L^{r_2}$ norm of $g_k$ uses the same argument as the proof of boundedness of $BHT$ from \cite{biest}: the amelioration, which will allow us to obtain the vector-valued estimate, relies in the localization of the $\eenergy$. 

Applying H\"older in \eqref{eq:bht-local-fns}, we obtain a vector-valued local estimate corresponding to a vector-valued version of the fundamental local result \eqref{eq:fundam-local-est}:
\begin{equation}
\label{eq:bht-local-fns-restr}
\big| \sum_k \Lambda_{\rr P \left( I_0 \right)}(f_k , g_k , h_k )  \big| \lesssim \left(\sssize_{I_0} \one_F  \right)^{\frac{1+\theta_1}{2}-\epsilon} \cdot \left(\sssize_{I_0} \one_G \right)^{\frac{1+\theta_2}{2} -\epsilon} \cdot \left(\sssize_{I_0} \one_{H'} \right)^{\frac{1+\theta_3}{2}-\epsilon} \cdot \vert I_0 \vert.
\end{equation}
This will be used, together with a stopping time that will be described shortly, in order to obtain the general vector-valued inequality. The stopping times are, in some sense, reversing the localization procedure: we need to find the good intervals that allow us to sum up the sizes, so that $\sssize_{I_0} \one_F$ should be related somehow to $|F|$, and the same for the functions $\one_G$ and $\one_{H'}$.

The triple stopping time used in \cite{vv_BHT} yields three collections of intervals $\ii I_j^{n_j}$ (one for each function $\one_F, \one_G$ and $\one_{H'}$), and associated to each of these, a subcollection $\rr P_{I_j} \subseteq \rr P$ of tritiles, where $1\leq j \leq 3$. Once we have these, for every $I_0=I_1\cap I_2 \cap I_3$, where  $I_j \in \ii I_j^{n_j}$, we need to consider $\Lambda_{\rr P_{I_0}}(f, g, h)$. Here $\ds \rr P_{I_0}:= \bigcap_j \rr P_{I_j}$.

In fact, the trilinear form $\Lambda_{\rr P}(f_k, g_k, h_k)$ is bounded above by
\[
\big| \Lambda_{\rr P}(f_k, g_k, h_k)  \big| \lesssim \sum_{n_1, n_2, n_3} \sum_{\substack{I_0=I_1\cap I_2 \cap I_3 \\ I_j \in \ii I_j^{n_j}}} \big| \Lambda_{\rr P_{I_0}}(f_k, g_k, h_k)  \big|.
\]

The intervals $I_1 \in \ii I_1^{n_1}$ and the collections $\rr P_{I_1} \subseteq \rr P(I_1)$ of tritiles are chosen so that $\ds \sssize_{\rr P_{I_1}} \one_F \sim 2^{-n_1} \sim \ave_{I_1} \one_F $, and similarly for the sizes of $\one_G$ and $\one_{H'}$ respectively. Since we will be considering the intersection $\ds \rr P_{I_0}:= \bigcap_j \rr P_{I_j}$, and the $\sssize$ is a maximal function at the level of the tritiles, we want to make sure that
\[
\sssize_{\rr P'_{I_1}} \one_F \lesssim \sssize_{\rr P_{I_1}} \one_F  \sim \ave_{I_1} \one_F 
\] 
whenever $\rr P'_{I_1}$ is a subcollection of $\rr P_{I_1}$.

For this reason, we  set $\rr P_{Stock}:=\rr P$, and start with the maximal possible $\sssize$ for $\one_F$ (which is going to be bounded nevertheless by $\min(1, 2^d C |F|)$), say $2^{-\bar n_1}$. Then set
$$\rr P_{\bar n_1}:=\lbrace  P \in \rr P_{Stock} : \frac{1}{| I_P|} \int_{\rr R} \one_F \cdot \ci_{I_P}^Mdx  \sim 2^{-\bar n_1}  \rbrace.$$

The family $\ii I _1^{\bar n_1}$ will consist of maximal dyadic intervals $I_1$ so that there exists $P \in \rr P_{\bar n_1}$ with $I_P \subseteq I_1$, and moreover, we require that 
\[
2^{-\bar n_1-1} \leq \frac{1}{| I_1|} \int_{\rr R} \one_F \cdot \ci_{I_1}^M dx \leq 2^{-\bar n_1}.
\] 
Clearly, $\ii I_{\bar n_1} \neq \emptyset$ unless $\rr P_{\bar n_1} =\emptyset$, and all the intervals $I_1 \in \ii I_{1}^{\bar n_1}$ are mutually disjoint.

Then we set $\rr P_{I_1}:=\lbrace P \in \rr P_{Stock}: I_P \subseteq I_1   \rbrace$, and we note that, for any subset $\rr P' \subseteq \rr P_{I_1}$,
\[
\sssize_{\rr P'} \one_F \lesssim 2^{-\bar n_1}.
\]
Before repeating the algorithm, we set $\ds \rr P_{Stock}:=\rr P_{Stock} \setminus \bigcup_{I_1 \in \ii I_1^{\bar n_1}}  \rr P_{I_1}$. As a consequence, the maximal possible $\sssize \one_F $ decreases. 

We continue the construction of $\ii I_1^{\bar n_1 +1}, \ii I_1^{\bar n_1 +2}, \ldots$  (which could be empty), until $\rr P_{Stock}=\emptyset$. Two properties are especially important:
\begin{itemize}
\item[(1)] for every $n_1$ so that $\ii I^{n_1}_1 \neq \emptyset$, we have $\ds \sum_{I \in \ii I^{n_1}_1}|I| \lesssim 2^{n_1} |F|$.
\item[(2)] $\ds \sssize_{\rr P'_{I_1}} \one_F \lesssim \sssize_{\rr P_{I_1}} \one_F \lesssim 2^{-n_1} \lesssim \min(1, 2^d |F|/|H|)$, whenever $\rr P'_{I_1}$ is a subcollection of $\rr P_{I_1}$.
\end{itemize}
The stopping times for $\sssize \one_G$ and $\sssize \one_{H'}$ are very similar, with the exception that $\sssize_{\rr P _{I_3}} \one_{H'} \lesssim 2^{-n_3} \lesssim 2^{-Md}$ . We end by recalling how to deduce \eqref{eq:restr-weak-type-1} from \eqref{eq:bht-local-fns-restr}:
\begin{align*}
&\big| \sum_k \Lambda_{\rr P}(f_k, g_k, h_k)  \big| \lesssim \sum_{n_1, n_2, n_3} \sum_{\substack{I_0=I_1\cap I_2 \cap I_3 \\ I_j \in \ii I_j^{n_j}}}  \sum_k \big| \Lambda_{\rr P_{I_0}}(f_k, g_k, h_k)  \big| \\
&\lesssim \sum_{n_1, n_2, n_3} \sum_{\substack{I_0=I_1\cap I_2 \cap I_3 \\ I_j \in \ii I_j^{n_j}}} \left(\sssize_{I_0} \one_F  \right)^{\frac{1+\theta_1}{2}-\epsilon} \cdot \left(\sssize_{I_0} \one_G \right)^{\frac{1+\theta_2}{2} -\epsilon} \cdot \left(\sssize_{I_0} \one_{H'} \right)^{\frac{1+\theta_3}{2}-\epsilon} \cdot \vert I_0 \vert \\
&\lesssim \sum_{n_1, n_2, n_3} 2^{-n_1 \left( \frac{1+\theta_1}{2} -\epsilon- \gamma_1  \right)} \cdot 2^{-n_2 \left( \frac{1+\theta_2}{2} -\epsilon- \gamma_2  \right)} \cdot 2^{-n_3 \left( \frac{1+\theta_3}{2} -\epsilon- \gamma_3 \right)}\cdot |F|^{\gamma_1} \cdot  |G|^{\gamma_2} \cdot |H|^{\gamma_3},
\end{align*}
where $0 \leq \gamma_1, \gamma_2, \gamma_3 \leq 1$, with $\gamma_1+\gamma_2+\gamma_3=1$. Since 
\[
2^{-n_1} \lesssim \min(1, 2^d |F|/|H|), \quad 2^{-n_2} \lesssim \min(1, 2^d |G|/|H|), \quad 2^{-n_3} \lesssim 2^{-Md},
\]
we obtain \eqref{eq:restr-weak-type-1} for $\ds 0 < \alpha_1 \leq  \frac{1+\theta_1}{2} -\epsilon,  0 < \alpha_2 \leq  \frac{1+\theta_2}{2} -\epsilon, 0< \alpha_3 \leq \frac{1+\theta_3}{2}-\epsilon$, and $\alpha_1+\alpha_2 +\alpha_3=1$.

Lastly, we present a geometric property of the collection of intervals obtained through the stopping times above.

\begin{proposition}
The collection $\ii I_1:=\bigcup_{n_1} \ii I_1^{n_1} $ of dyadic intervals satisfies a $\kappa$-Carleson condition, i.e. there exists $\kappa \geq 1$ so that for every $I \in \ii I_1$, we have
\begin{equation}
\label{eq:Carleson}
\sum_{\substack{I' \in \ii I_1 \\ I' \subseteq I }} |I'| \leq \kappa |I|.
\end{equation}
\begin{proof}
If $I \in \ii{I}_1^n$ and $I' \in \ii{I}_1^m$ are so that $I' \subset I$, then $m >n$ and moreover, $I'$ was selected first. This is due to the maximality of the intervals and because, in this situation, 
$$ \frac{1}{| I|} \int_{\rr R} \one_F \cdot \ci_{I}^M dx \sim 2^{n}, \qquad  \frac{1}{| I'|} \int_{\rr R} \one_F \cdot \ci_{I'}^M dx \sim 2^{m}.$$

The above condition can be reinterpreted, so that in fact,
\[
C \, 2^m \cdot \one_{I'}(x)\leq \mathcal{M} \left( \one_F \cdot \ci_I  \right)(x) \cdot \one_{I'}(x).
\]
Then, since all the intervals in $\ii{I}_1^m$ are disjoint,
\begin{equation}
\sum_{I' \in \ii{I}_1^m} \vert I' \vert \leq \vert \lbrace x : \mathcal{M}(\one_F \cdot \ci_I)(x) > C \cdot 2^m  \rbrace \vert \leq \tilde{C} \, C^{-1} \, 2^{-m} \| \one_F \cdot \ci_I \|_1 \leq \tilde{C} \, 2^{n+1-m} \vert I \vert.
\end{equation}
Summing in $m>n$, we get the desired conclusion:
\[ 
\sum_{\substack{I' \subset I \\ I' \in \ii I_1}} \vert I' \vert=\sum_{m >n} \sum_{I' \in \ii{I}_1^m} \vert I' \vert \leq \sum_{m >n} \tilde{C} \, 2^{n+1-m} \vert I \vert \leq \kappa |I|.
\]

Note that the constant $\kappa$ depends only on the norm of the maximal operator.
\end{proof}
\end{proposition}

\subsection{Conclusions}~\\
\label{sec:conclusions-comp}
As mentioned earlier, Carleson and sparse collections represent the same concept: following \cite{Nazarov-Lerner-DyadicCalculus}, a $\kappa$-Carleson collection of intervals is in fact $\frac{1}{\kappa}$-sparse, and vice versa.  

The stopping time used for proving vector-valued estimates yields three Carleson collections, one for every function. The algorithm starts with the largest possible $\sssize_{\rr P_{Stock}} \one_F$, selects intervals $I_1$ with
\[
\sssize_{\rr P_{Stock}} \one_F \sim \frac{1}{|I_1|} \int_{\rr R} \one_F \cdot \ci_{I_1}^M dx,
\]
and all the tiles in $\rr P_{Stock} \cap \rr P(I_1)$, with spatial interval contained in $I_1$. Then the procedure resumes, the maximal size decreases, while the intervals in $\ii I_1$ become larger and larger.

On the other hand, the stopping time for the sparse domination produces one sparse collection $\ic S$ of dyadic intervals. For each $Q \in \ic S$, we also have a subcollection $\rr P_Q $ of tritiles, which we want to satisfy simultaneously
{\fontsize{10}{10}\[
\sssize_{\rr P_Q} \one_F \leq C \frac{1}{|Q|} \int_{\rr R} \one_F \cdot \ci_{Q}^M dx, \quad \sssize_{\rr P_Q} \one_G \leq C \frac{1}{|Q|} \int_{\rr R} \one_G \cdot \ci_{Q}^M dx,\quad \sssize_{\rr P_Q} \one_{H'} \leq C \frac{1}{|Q|} \int_{\rr R} \one_{H'} \cdot \ci_{Q}^M dx.
\]}

We start with the largest possible spatial intervals, and in the selection process we make sure that the above conditions hold. In this case, the spatial intervals are becoming smaller and smaller, and at the same time, the size (which should be regarded as a maximal average) is increasing.

\section{A rank $k$ collection of multi-tiles}
\label{sec:rank-k}
Now we prove sparse and multi-vectorial estimates for the operator $T_k$ of Theorem \ref{thm:main-thm}, the $n$-linear multiplier whose symbol is singular along a $k$-dimensional subspace, with $\ds k<\frac{n+1}{2}$.

\subsection{A few definitions}
Here we consider $\Gamma$ to be the $n$-dimensional vector space 
\[
\Gamma:=\lbrace \xi \in \rr R^{n+1}:  \xi_1+\ldots +\xi_{n+1} =0 \rbrace,
\]
and $\Gamma' \subset \Gamma$ a non-degenerate subspace of $\Gamma$ of dimension $0 \leq k <\dfrac{n+1}{2}$.

The $(n+1)$-linear form of an $n$-linear operator which is singular along $\Gamma'$ is given by 
\begin{equation}
\label{def:op-sing-k}
\Lambda(f_1, \ldots, f_{n+1})=\int_{\rr R^{n+1}} \delta(\xi_1+\ldots +\xi_{n+1}) m(\xi_1, \ldots , \xi_{n+1}) \hat{f}_1(\xi_1) \cdot \ldots \cdot \hat{f}_{n+1}(\xi_{n+1}) d \xi_1 \ldots d \xi_{n+1},
\end{equation}
where $m$ is a multiplier satisfying
\[
\vert \partial_\xi^\alpha m(\xi) \vert \lesssim \dist(\xi, \Gamma')^{-\vert \alpha \vert}
\]
for all partial derivatives $\partial_\xi^\alpha$ on $\Gamma$ up to some finite order. 

The model for the multilinear form is given by
\begin{equation}
\label{def-model-op-k}
\Lambda_{\rr P}(f_1, \ldots, f_{n+1}):=\sum_{P \in \rr P} \vert I_P\vert^{-\frac{n-1}{2}} \langle f_1, \phi_P^1  \rangle \cdot \ldots \cdot \langle f_{n+1}, \phi_P^{n+1} \rangle, 
\end{equation}
where $\rr P$ is a rank $k$ family of $(n+1)$-tiles. This notion will be specified shortly, but it essentially means that there are $k$ independent parameters in frequency.

We recall the order relation on tiles:
\begin{definition}
If $P, P'$ are tiles, then $P'< P$ if $I_{P'}\subsetneq I_P$ and $\omega_P \subseteq 3 \omega_{P'}$. Also, $P' \lesssim P$ if $I_{P'} \subseteq I_P$ and $\omega_P \subseteq 100 \omega_{P'}$. We write $P' \leq P$ if $P'<P$ or $P'=P$, and $P' \lesssim' P $ if $P' \lesssim P$ but $P' \nleq P$.
\end{definition}

Our notion of a rank $k$ collection of multi-tiles is somewhat different from the one in \cite{multilinearMTT}: in order to simplify the presentation, we include in the definition certain properties that were deduced in \cite{multilinearMTT}.

\begin{definition}
\label{def:rank}
A collection $\rr P$ of multi-tiles is said to have \emph{rank $k$} if for any $P, P' \in \rr P$ the following conditions are satisfied:
\begin{itemize}
\item \emph{any $k$ components determine the remaining ones:} if $1 \leq i_1< \ldots <i_k \leq n+1$ and if $\omega_{P_{i_s}}=\omega_{P'_{i_s}}$ for all $1 \leq s \leq k$, then $\omega_{P_i} = \omega_{ P'_i}$ for all $1 \leq i \leq n$.
\item \emph{any two multi-tiles which are overlapping in $k$ components have (frequency only) dilates that are overlapping in the remaining components:} if $1 \leq i_1< \ldots <i_k \leq n+1$ and if $P'_{i_s} \leq P_{i_s}$ for all $1 \leq s \leq k$, then $P'_i \lesssim P_i$ for all $1 \leq i \leq n$.
\item \emph{if the two multi-tiles correspond to different scales and they are overlapping in $k$ components, there will be at least two components which are not overlapping, though their frequency dilates are:} if we further assume that $\vert I_{P'} \vert \ll \vert I_P\vert$, then we have $P'_i \lesssim ' P_i$ for at least two choices of $i$.
\end{itemize}
\end{definition}

It is not difficult to prove that the discretization of the multilinear form from $\eqref{def-model-op-k}$ admits a rank $k$ model form as in \eqref{def:op-sing-k}. This is detailed in \cite{multilinearMTT}.  The most well-known example corresponds to $n=2$ and $k=1$: the bilinear Hilbert transform $BHT$ is a bilinear operator given by the multiplier $ \text{sgn} (\xi_1-\xi_2)$, which is singular along a line.

The handling of the more general case of a rank $k$ model operator is similar: the multi-tiles are grouped in subcollections called \emph{trees} according to their \emph{size}. In order to make this statement precise, we need to introduce a few definitions.

\begin{definition}
\label{def:tree}
For any $1 \leq j \leq n+1$, a \emph{$j$-tree} with \emph{top} $P_T=(I_T \times \omega_{T_1}, \ldots, I_T \times \omega_{T_{n+1}})$ is a subcollection $T$ of $\rr P$ so that $P_j \lesssim P_{T, j}$ for all $P \in T$.

A tree is called \emph{$j$-overlapping} if $P_j \leq P_{T, j}$ for all $P \in T$, and \emph{$j$-lacunary} if $P_j \lesssim ' P_{T, j}$ for all $P \in T$.

Finally, if $1 \leq i_1< \ldots <i_k \leq n+1$,  a subcollection $T$ of $\rr P$ is called an $(i_1, \ldots, i_k)$-tree if $T$ is an $i_s$-tree for all $1 \leq s \leq k$. 
\end{definition}

In the case of the bilinear Hilbert transform operator, which corresponds to a rank-1 family of tri-tiles, the convention used in \cite{multilinear_harmonic} or \cite{biest} is that a $j$-tree is $j$-overlapping and lacunary in the other two directions. For the general case, things are slightly more complicated since there are $k$ degrees of freedom. However, the notion of \emph{rank $k$} from Definition \ref{def:rank} ensures that once $k$ directions are fixed, one obtains a $j$-tree for all $1 \leq j \leq n+1$, and moreover, this tree will have at least two lacunary directions. In some sense, for a rank $k$ collection of multi-tiles, the $(i_1, \ldots, i_k)$-trees represent the ``fundamental constituents".

\begin{definition}
\label{def:size-m-k}
Given $\rr P$ a rank-$k$ collection of multi-tiles and $1 \leq j \leq n+1$, we define the \emph{size} of the sequence $\ds \big( \langle f, \phi_P^j \rangle \big)_{P \in  \rr P}$ by
\begin{equation}
\label{eq:def-size}
\ssize_{\rr P} \big(\langle f, \phi_P^j \rangle_{P \in \rr P} \big) :=\sup_{T \subseteq \rr P} \big( \frac{1}{\vert I_T \vert} \sum_{P \in T} \vert \langle f, \phi_P^j \rangle \vert^2  \big)^{\frac{1}{2}}
\end{equation}
where $T$ ranges over all $j$-lacunary trees in $\rr P$ (including one multi-tile trees).
\end{definition}

It was proved in \cite{biest} that 
\[
\ssize_{\rr P} \big(\langle f, \phi_P^j \rangle_{P \in \rr P} \big) \lesssim_M \sup_{P \in \rr P} \frac{1}{\vert I_P \vert} \int_{\rr R} \vert f(x) \vert \cdot \ci_{I_P}^M(x) dx,
\]
and for this reason we use the notation from Definition \ref{def:size}:
\begin{equation}
\label{eq:def-ssize}
\sssize_{\rr P}(f):=\sup_{P \in \rr P} \frac{1}{\vert I_P \vert} \int_{\rr R} \vert f(x) \vert \cdot \ci_{I_P}^M(x) dx.
\end{equation}


\begin{lemma}
\label{lemma-tree-lemma}
If $T$ is an $(i_1, \ldots, i_k)$-tree for some $1 \leq i_1<\ldots  <i_k \leq n+1$, then 
\[
\vert \sum_{P \in T} \vert I_P\vert^{-\frac{n-1}{2}} \langle f_1, \phi_P^1  \rangle \cdot \ldots \cdot \langle f_{n+1}, \phi_P^{n+1} \rangle \vert \lesssim \prod_{j=1}^{n+1} \ssize_{T} \big(\langle f_j, \phi_P^j \rangle \big) \cdot \vert I_T \vert.
\]
\end{lemma}

Because of the existence of at least two lacunary directions in a $(i_1, \ldots, i_k)$-tree, we have a way of estimating how the functions act on such ``elementary blocks". Further, an important step in the proof is the \emph{decomposition lemma} below, which allows us to organize a collection of tiles into trees of similar size:

\begin{lemma}
\label{lemma:dec-lemma}
If $\rr P_j$ is a collection of $j$-tiles with $\ssize_{\rr P_j}\big(\langle f_j, \phi_P^j \rangle \big) \leq \lambda$, then there exists a decomposition 
$\rr P_j =\rr P'_j \cup \rr P''_j$ so that $\ssize_{\rr P'_j}\big(\langle f_j, \phi_P^j \rangle \big) \leq \frac{\lambda}{2}$ and $\rr P''_j$ is a union $\ds \rr T=\bigcup_{T \in \rr T} T$ of disjoint trees so that 
\begin{equation}
\label{eq:trees-j-projection}
\sum_{T \in \rr T} \vert I_T \vert \lesssim \lambda^{-2} \|  f_j \|_2^2.
\end{equation}
Furthermore, if all the tiles in $\rr P_j$ have their spatial support contained inside a fixed dyadic interval $I_0$, then \eqref{eq:trees-j-projection} can be improved to 
\begin{equation}
\label{eq:improved-local-sum-tree-top}
\sum_{T \in \rr T} \vert I_T \vert \lesssim \lambda^{-2} \|  f_j \cdot \ci_{I_0} \|_2^2.
\end{equation}
\begin{proof}
The decomposition of the collection $\rr P_j$ into trees according to their size is a classical result in time-frequency analysis and it corresponds to a Calder\'on-Zygmund decomposition at the level of the tiles. The selection of the trees in $\rr T$ is the conventional one: we start by choosing the maximal $j$-lacunary ones which have size greater than $\lambda/2$, and in order to insure orthogonality, some $j$-overlapping trees are removed as well. Here, we will only elaborate on the localized version, i.e. the estimate \eqref{eq:improved-local-sum-tree-top}, which follows by proving
\[
\sum_{T \in \rr T} \sum_{P \in T} \vert \langle f, \phi_P  \rangle \vert^2 \lesssim \lambda \big( \sum_{T \in \rr T} |I_T|  \big)^\frac{1}{2} \, \| f \cdot \ci_{I_0} \|_2,
\]
or equivalently, by a $TT^*$ argument,
\[
\big\|   \big(  \sum_{T \in \rr T} \sum_{P \in T} \langle f, \phi_P  \rangle \phi_P \,   \big) \ci_{I_0}^{-\frac{N}{2}}  \big\|_2   \lesssim \lambda \big( \sum_{T \in \rr T} |I_T|  \big)^\frac{1}{2}.
\]
We perform a decomposition of $\rr R$ into dyadic shells around $I_0$, which reduces the above estimate to 
\[
\big\|   \big(  \sum_{T \in \rr T} \sum_{P \in T} \langle f, \phi_P  \rangle \phi_P \,   \big) \cdot \one_{10 I_0}  \big\|_2   \lesssim \lambda \,\big( \sum_{T \in \rr T} |I_T|  \big)^\frac{1}{2}, \quad \big\|   \big(  \sum_{T \in \rr T} \sum_{P \in T} \langle f, \phi_P  \rangle \phi_P \,   \big) \ci_{I_0, \kappa} \big\|_2   \lesssim 2^{-\frac{\kappa \, N}{2}} \lambda \, \big( \sum_{T \in \rr T} |I_T|  \big)^\frac{1}{2},
\]
where for every $\kappa \geq 2$, $\ci_{I_0, \kappa}$ is a smooth bump function adapted to the region $2^{\kappa+1}I_0 \setminus 2^\kappa I_0$, supported outside $2^{\kappa-1}I_0$ for $\kappa$ large enough.
 
In the first term, we ignore the truncation and instead deal with
\[
\big\|  \sum_{T \in \rr T} \sum_{P \in T} \langle f, \phi_P  \rangle \phi_P \,  \big\|_2   \lesssim \lambda \,\big( \sum_{T \in \rr T} |I_T|  \big)^\frac{1}{2}.
\] 
This is a standard estimate and it heavily relies on the orthogonality of the selected trees. It remains to prove, for any $\kappa \geq 2$, and any interval $I \subseteq I_0$:
\begin{equation}
\label{eq:trees-energy-equal-length}
\big\|   \big(  \sum_{T \in \rr T} \sum_{\substack{P \in T \\ I_P =I}} \langle f, \phi_P  \rangle \phi_P \,   \big) \ci_{I_0, \kappa} \big\|_2   \lesssim 2^{-\frac{\kappa \, N}{2}} \frac{|I|^\frac{3}{2}}{|I_0|^\frac{3}{2}} \lambda\, \big( \sum_{T \in \rr T} |I_T|  \big)^\frac{1}{2},
\end{equation}  
which implies the desired result upon summation in $\kappa$ and in $I \subseteq I_0$, details which are left to the interested reader. On the right hand side, any power strictly greater than 2 would be sufficient: it allows us to sum with respect to the scale $|I|/ |I_0|$ and the number of intervals of fixed length. 

Relying again on a $TT^*$ argument, the left hand side of \eqref{eq:trees-energy-equal-length} squared can be written as
\[
\sum_{T \in \rr T} \sum_{\substack{P \in T\\ I_P=I}} \sum_{T' \in \rr T} \sum_{\substack{ P' \in T' \\ I_{P'}=I}} \langle f, \phi_P \rangle \, \overline{\langle f, \phi_{P'} \rangle} \int_{\rr R} \phi_P(x) \, \overline{\phi_{P'}(x)} \ci^2_{I_0, \kappa}(x) dx.
\]

Since $\vert \langle f, \phi_P \rangle \, \overline{\langle f, \phi_{P'} \rangle}  \vert \lesssim |\langle f, \phi_P \rangle |^2 +\vert \langle f, \phi_{P'} \rangle \vert ^2$, it will be enough to show, for any fixed tree $T \in \rr T$ and any fixed multi-tile $P \in T$ with $I_P =I$, that
\begin{equation}
\label{eq:tree-reduced-form-local}
\sum_{T' \in \rr T} \sum_{P' \in T'}\big \vert \int_{\rr R} \phi_P(x) \, \overline{\phi_{P'}(x)} \ci^2_{I_0, \kappa}(x) dx \big \vert \lesssim 2^{-\kappa N} \, \frac{|I|^3}{|I_0|^3}.
\end{equation}
This is because the size of every tree is controlled by $\lambda$ and in consequence
\[
\sum_{T \in \rr T} \sum_{P \in T} |\langle f, \phi_P \rangle |^2 \lesssim \lambda^2 \sum_{T \in \rr T} |I_T|.
\]

To deal with \eqref{eq:tree-reduced-form-local}, we note that since $I_P=I=I_P'$, every $P'$ must come from a different tree, and that all the frequency intervals $\omega_{P'}$ (actually $\omega_{P'_j}$, but we ignore the $j$ index here) are mutually disjoint as $P'$ varies in $\ds \cup_{T' \in \rr T} T'$. Moreover, they all have the same length and they are equally spaced. For simplicity, we assume $\omega_P$ is centered at $0$, so $\omega_P +\omega_{P'}$, which contains the Fourier support of $\phi_P \, \overline{\phi_{P'}}$, is a subset of $[\text{dist\,}(\omega_{P}, \omega_{P'}), \text{dist\,}(\omega_{P}, \omega_{P'}) +2 |\omega_P|]$. Hence, if $P' \neq P$, and $\Phi^{[M]}(x)$ is a function so that 
\[
\frac{d^M}{d \, x^M} \Phi^{[M]}(x) = \phi_P(x) \, \overline{\phi_{P'}}(x),
\]
then we have 
\[
\vert \Phi^{[M]}(x) \vert \lesssim |I|^{-1} \, \text{dist }(\omega_P, \omega_{P'})^{-M}.
\]
Integrating by parts $M$ times, we obtain
\begin{align*}
& \big \vert \int_{\rr R} \phi_P(x) \, \overline{\phi_{P'}(x)} \ci^2_{I_0, \kappa}(x) dx \big \vert \lesssim \big \vert \int_{\rr R} \Phi^{[M]}(x) \frac{d ^M}{d\, x^M} \ci^2_{I_0, \kappa}(x) dx \big \vert \\\
&\lesssim  \big(\frac{ \text{dist }(\omega_P, \omega_{P'}) }{|\omega_P|} \big)^{-M} \, 2^{- \kappa M} \Big( \frac{|I|}{|I_0|} \Big)^M.
\end{align*}
Summing over tiles $P' \neq P$ with $I_{P'}=I=I_P$ is now possible and we get \eqref{eq:tree-reduced-form-local}. To deal with the case $P'=P$, we use the fast decay of the $L^2$ normalized bump function $\phi_P$, which is adapted to $I_P=I$, and the fact that $\ci_{I_0, \kappa}$ is supported outside $2^{\kappa-1}I_0$.
\end{proof}
\end{lemma}

For a rank-1 family of multi-tiles, the decomposition above can be performed directly on the collection $\rr P$ of multi-tiles (this is Lemma 7.7 of \cite{multilinearMTT}): if $\ssize_{\rr P}\big(\langle f_j, \phi_P^j \rangle \big) \leq \lambda$, then there exists a decomposition $\rr P =\rr P' \cup \rr P''$ so that $\ssize_{\rr P'}\big(\langle f_j, \phi_P^j \rangle \big) \leq \frac{\lambda}{2}$ and $\rr P''$ is a union $\rr T$ of disjoint $j$-trees ($j$-lacunary or $j$-overlapping) so that 
\[
\sum_{T \in \rr T} \vert I_T \vert \lesssim \lambda^{-2} \|  f_j \|_2^2.
\]

As explained in Section \ref{sec:local->sparse}, the following localization result plays a key role in the proof of the sparse and vector-valued results of Theorem \ref{thm:main-thm}.

\begin{theorem}
\label{thm:localization-rank-k}
Let $I_0$ be a fixed dyadic interval and $\rr P$ a rank $k$ collection of multi-tiles, $E_1, \ldots, E_{n+1}$ sets of finite measure, and $f_1, \ldots, f_{n+1}$ functions with the property that $\vert  f_j(x) \vert \leq \one_{E_j}(x), \forall 1 \leq j \leq n+1$. Then we have 
\begin{equation}
\label{eq:local-est-k-thm}
\vert \Lambda_{\rr P \left(I_0 \right)}(f_1, \ldots, f_{n+1}) \vert \lesssim \prod_{j=1}^{n+1} \big( \sssize_{I_0} \one_{E_j}    \big)^{1-\alpha_j} \cdot \vert I_0  \vert,
\end{equation}
where the exponents $\alpha_j \in (0,1/2)$. Moreover, they are defined by
\begin{equation}
\label{def:exp-alpha_j}
\alpha_j:=\sum_{\substack{1 \leq i_1<\ldots < i_k \leq n+1 \\ i_s =j \text{  for some   } 1\leq  s \leq k}} \theta_{i_1, \ldots, i_k}
\end{equation}
where $ 0 \leq \theta_{i_1, \ldots, i_k} \leq 1$ are any positive numbers indexed by ordered $k$-tuples such that 
\[
\sum_{1 \leq i_1 < \ldots < i_k \leq n+1} \theta_{i_1, \ldots, i_k}=1.
\]
\end{theorem}

\subsection{Proof of Theorem \ref{thm:localization-rank-k}: the case $k=1$}~\\
\label{sec:proof-bht-case}
We first present the proof in the case $k=1, n=2$, which corresponds to the bilinear Hilbert transform operator. Although the localization result for $BHT$ has already appeared in \cite{vv_BHT}, we include a short revision, anticipating that the general rank $k$ case will be built upon it.

Thus we have a rank $1$ family $\rr P$ of tri-tiles, $I_0$ is a fixed interval, and $f_1, f_2, f_3$ are restricted-type functions:
\[
\vert f_1(x) \vert \leq \one_{E_1}(x), \quad \vert f_2(x) \vert \leq \one_{E_2}(x) \quad \text{and} \quad \vert f_3(x) \vert \leq \one_{E_3}(x).
\] 
We want to prove that 
\begin{equation}
\label{eq:BHT-local-proof-rank-k}
\vert \Lambda_{\rr P \left(I_0 \right)}(f_1, f_2, f_3)  \vert \lesssim \big( \sssize_{I_0} \one_{E_1}  \big)^{1-\theta_1} \cdot  \big( \sssize_{I_0} \one_{E_2}  \big)^{1-\theta_2} \cdot  \big( \sssize_{I_0} \one_{E_3}  \big)^{1-\theta_3} \cdot \vert I_0  \vert,
\end{equation}
where $0 < \theta_i \leq \frac{1}{2}$ and $\theta_1+\theta_2+\theta_3=1$. 

Using the decomposition result in Lemma \ref{lemma:dec-lemma}, we obtain the families: $\ic T_{n_1}, \ic T_{n_2}$ and $\ic T_{n_3}$, where each $\ic T_{n_j}$ is a union of disjoint $j$-trees satisfying 
\[
\sum_{T \in \ic T_{n_j}} |I_T| \lesssim 2^{2\, n_j} \|f_j \cdot \ci_{I_0} \|_2^2, \qquad \ssize_{T}(f_j) \leq 2^{-n_j} \leq \sssize_{I_0} \one_{E_j} \quad \forall \, T \in \ic T_j.
\]
If we denote by $\ic T_{n_1, n_2, n_3}$ the collection of trees $T:=T_1 \cap T_2 \cap T_3$, where $T_j \in \ic T_{n_j}$, we have that 
\begin{align*}
\vert \Lambda_{\rr P \left(I_0 \right)}(f_1, f_2, f_3) \vert & \lesssim \sum_{n_1, n_2, n_3} \sum_{T \in \ic T_{n_1, n_2, n_3}} \sum_{\substack{P \in T \\I_P \subseteq I_0}} \big\vert \vert I_P \vert^{-\frac{1}{2}} \langle f_1, \phi_P^1  \rangle  \langle f_2, \phi_P^2  \rangle  \langle f_3, \phi_P^3 \rangle  \big\vert \\
&\lesssim \sum_{n_1, n_2, n_3} 2^{-n_1} 2^{-n_2} 2^{-n_3} \sum_{T \in \ic T_{n_1, n_2, n_3}}\vert I_T \vert.
\end{align*}

Then Lemma \ref{lemma:dec-lemma} implies that 
\[
\sum_{T \in \ic T_{n_1, n_2, n_3}} \vert I_T\vert \lesssim 2^{2n_1} \|  f_1 \cdot \ci_{I_0} \|_2^2 \lesssim 2^{2 n_1} \cdot \big( \int_{\rr R}  \one_{E_1} \cdot \ci_{I_0}^M dx  \big) \lesssim 2^{2 n_1} \big( \sssize_{I_0} \one_{E_1}  \big) \cdot \vert I_0 \vert.
\]

Similarly, we have 
\[
\sum_{T \in \ic T_{n_1, n_2, n_3}} \vert I_T\vert \lesssim 2^{2 n_2} \big( \sssize_{I_0} \one_{E_2}  \big) \cdot \vert I_0 \vert \qquad \text{and } \qquad \sum_{T \in \ic T_{n_1, n_2, n_3}} \vert I_T\vert \lesssim 2^{2 n_3} \big( \sssize_{I_0} \one_{E_3}  \big) \cdot \vert I_0 \vert,
\]
and interpolating the three inequalities above we obtain that 
\[
\sum_{T \in \ic T_{n_1, n_2, n_3}} \vert I_T\vert \lesssim 2^{2 \theta_1 n_1} \big( \sssize_{I_0} \one_{E_1}  \big)^{\theta_1} \cdot 2^{2 \theta_2 n_2} \big( \sssize_{I_0} \one_{E_2}  \big)^{\theta_2}  \cdot 2^{2 \theta_3 n_3} \big( \sssize_{I_0} \one_{E_3}  \big)^{\theta_3}  \cdot \vert I_0 \vert.
\]

This allows us to estimate the trilinear form by 
{\fontsize{10}{11}\begin{align*}
\vert \Lambda_{\rr P \left(I_0 \right)}(f_1, f_2, f_3) \vert & \lesssim \sum_{n_1, n_2, n_3} 2^{-n_1\left( 1- 2 \theta_1\right)} 2^{-n_2\left( 1- 2 \theta_2\right)} 2^{-n_3\left( 1- 2 \theta_3\right)} \cdot \big( \sssize_{I_0} \one_{E_1}  \big)^{\theta_1} \cdot \big( \sssize_{I_0} \one_{E_2}  \big)^{\theta_2} \cdot \big( \sssize_{I_0} \one_{E_3}  \big)^{\theta_3} \cdot \vert I_0 \vert.
\end{align*}}

Since $2^{-n_1} \leq \sssize_{I_0} \one_{E_1}, 2^{-n_2} \leq \sssize_{I_0} \one_{E_2}$ and $2^{-n_3} \leq \sssize_{I_0} \one_{E_3}$ from the decomposition, we obtain, provided that $0 \leq \theta_1, \theta_2, \theta_3 <\dfrac{1}{2}$ and $\theta_1+\theta_2+\theta_3=1$
\[
\vert \Lambda_{\rr P \left(I_0 \right)}(f_1, f_2, f_3) \vert  \lesssim \big( \sssize_{I_0} \one_{E_1}  \big)^{1-\theta_1} \cdot \big( \sssize_{I_0} \one_{E_2}  \big)^{1-\theta_2} \cdot \big( \sssize_{I_0} \one_{E_3}  \big)^{1-\theta_3} \cdot \vert I_0 \vert.
\]

This is \eqref{eq:BHT-local-proof-rank-k}, which is another way of writing \eqref{eq:localBHT}, but the $\theta_j$ in the two equations are different.

\subsection{Proof of Theorem \ref{thm:localization-rank-k}: the case $k>1$}
Applying Lemma \ref{lemma:dec-lemma} in every component, for every $1 \leq j \leq n+1$ we can write $\rr P(I_0)$ as
\[
\rr P (I_0) := \bigcup_{l_j} \bigcup_{T_j \in \ic T_{l_j}} T_j,
\]
where the collection $T_j$ is a $j$-tree (the order relation needs to hold only in the $j^{\text{th}}$ component) satisfying 
\[
2^{-l_j-1} \leq \ssize_{T_j} \big(\langle f_j, \phi_P^j \rangle\big) \leq 2^{-l_j} \lesssim \sssize_{I_0} \one_{E_j}.
\]

Hence we can write 
\[
\rr P (I_0) :=\bigcup_{l_1, \ldots, l_{n+1}}\bigcup_{T_1 \in \ic T_{l_1}} \ldots  \bigcup_{T_{n+1} \in \ic T_{l_{n+1}}} \big( T_1 \cap \ldots \cap T_{n+1} \big).
\]
Equivalently, if we denote $\ic T_{l_1, \ldots, l_{n+1}}$ the collection of multi-tiles $\ds T=T_1 \cap \ldots \cap T_{n+1}$, where $T_j \in \ic T_{l_j}$, we have $\ds \rr P(I_0)=\bigcup_{l_1, \ldots, l_{n+1}} \bigcup_{T \in \ic T_{l_1, \ldots, l_{n+1}}} T$, and every $T \in \ic T_{l_1, \ldots, l_{n+1}}$ satisfies, for any $1 \leq j \leq n+1$, 
\[
 \ssize_{T} \big(\langle f_j, \phi_P^j \rangle\big) \leq 2^{-l_j} \lesssim \sssize_{I_0} \one_{E_j}.
\]

Using Lemma \ref{lemma-tree-lemma}, we have that
\[
\vert \Lambda_{\rr P \left(I_0 \right)}(f_1, \ldots, f_{n+1}) \vert \lesssim \sum_{l_1, \ldots, l_{n+1}} 2^{-l_1} \cdot \ldots \cdot 2^{-l_{n+1}} \sum_{T \in \ic T_{l_1, \ldots, l_{n+1}}} \vert I_T \vert.
\] 

Every subset $T \in \ic T_{l_1, \ldots, l_{n+1}}$ is part of an $(i_1, \ldots, i_k)$-tree (or a ``fundamental constituent"), and there are $n+1 \choose k$ such trees. The reader should recall that in Section \ref{sec:proof-bht-case}, any $T \in \ic T_{n_1, n_2, n_3}$ is of the form $T=T_1 \cap T_2 \cap T_3$, and hence can be viewed as a subset of a $1$-tree, $2$-tree or $3$-tree, a fact which was used for estimating $\ds \sum_{T \in \ic T_{n_1, n_2, n_3}}|I_T|$.

Similarly, the sum $\ds \sum_{T \in \ic T_{l_1, \ldots, l_{n+1}}} \vert I_T \vert$ can be estimated in $n+1 \choose k$ ways, and by their geometric mean:
\[
\sum_{T \in \ic T_{l_1, \ldots, l_{n+1}}} \vert I_T \vert \lesssim \prod_{1\leq i_1< \ldots < i_k \leq n+1} \big(  \sum_{\substack{T \in \ic T_{l_1, \ldots, l_{n+1}} \\ T \text{ is a } (i_1, \ldots, i_k)-\text{tree}}}\vert I_T \vert \big)^{\theta_{i_1, \ldots, i_k}}
\]
for $0 \leq \theta_{i_1, \ldots, i_k} \leq 1$ with $\ds \sum_{1\leq i_1< \ldots < i_k \leq n+1} \theta_{i_1, \ldots, i_k}=1$

For each multi-index $(i_1, \ldots, i_k)$ we have
{\fontsize{10}{10}\begin{equation}
\label{eq:est-tops-trees}
\sum_{\substack{T \in \ic T_{l_1, \ldots, l_{n+1}} \\ T \subseteq T_{i_1} \cap \ldots \cap T_{i_k} \text{ is a } (i_1, \ldots, i_k)-\text{tree}}}\vert I_T \vert \lesssim    \sum_{T_{i_1} \in \ic T_{l_{i_1}}} \ldots \sum_{T_{i_k} \in \ic T_{l_{i_k}}} \vert I_{T_{i_1}} \cap \ldots \cap I_{T_{i_{k}}}  \vert \lesssim  \prod_{s=1}^k  2^{2l_{i_s}} \cdot \sssize_{I_0}\, \one_{E_{i_s}} \cdot \vert I_0  \vert . 
\end{equation}}
If $k=1$, the inequality is a consequence of the decomposition into trees, and was presented in Section \ref{sec:proof-bht-case}. If $k\geq 2$, $\ds \sum_{T \in \ic T_{l_1, \ldots, l_{n+1}}}|I_T|$ represents a $k$-dimensional volume, rather than a ``length"; the trees are $k$-dimensional and $k$ functions are required for computing this volume.

In order to prove \eqref{eq:est-tops-trees}, we can assume for simplicity that we sum upon collections $T_{i_s} \in \ic T_{l_{i_s}}$ satisfying $I_{T_{i_1}}\supseteq \ldots I_{T_{i_k}}$ (the general case reduces to this particular one, since the collection of dyadic intervals is well-ordered). Then we have to estimate
\[
\sum_{\substack{ T_{i_1} \in \ic T_{l_{i_1}}, \ldots, T_{i_k} \in \ic T_{l_{i_k}} \\ I_{T_{i_1}}\supseteq \ldots \supseteq I_{T_{i_k}} }} \vert I_{T_{i_1}} \cap \ldots \cap I_{T_{i_{k}}}  \vert = \sum_{\substack{ T_{i_1} \in \ic T_{l_{i_1}}, \ldots, T_{i_{k-1}} \in \ic T_{l_{i_{k-1}}} \\ I_{T_{i_1}}\supseteq \ldots \supseteq I_{T_{i_{k-1}}} }} \sum_{\substack{ T_{i_k} \in \ic T_{l_{i_k}} \\ I_{T_{i_{k}}} \subseteq I_{T_{i_{k-1}}} }} |I_{T_{i_k}}|.
\]
For the last term, we have, due to Lemma \ref{lemma:dec-lemma}:
\[
\sum_{\substack{ T_{i_k} \in \ic T_{l_{i_k}} \\ I_{T_{i_{k}}} \subseteq I_{T_{i_{k-1}}} }} |I_{T_{i_k}}| \lesssim 2^{2 l_{i_k}} \| f_{i_k} \cdot \ci_{I_{T_{i_{k-1}}}} \|_2^2 \lesssim 2^{2 l_{i_k}} \, \big( \sssize_{I_0} \one_{E_{i_k}} \big) \cdot |I_{T_{i_{k-1}}}|.
\]

We repeat the procedure, summing now over the trees $T_{i_{k-1}} \in \ic T_{l_{i_{k-1}}}$, all of which have spatial supports inside $I_{T_{i_{k-2}}}$. Eventually, we obtain \eqref{eq:est-tops-trees} since all the tiles in $\rr P(I_0)$, and in consequence all the trees, are localized on $I_0$.

Hence, we get 
\begin{align*}
\vert \Lambda_{\rr P \left(I_0 \right)}(f_1, \ldots, f_{n+1}) \vert & \lesssim \sum_{l_1, \ldots, l_{n+1}} 2^{-l_1} \cdot \ldots \cdot 2^{-l_{n+1}} \prod_{1\leq i_1< \ldots < i_k \leq n+1} \big(  \prod_{s=1}^k  2^{2l_{i_s}} \cdot \sssize_{I_0} \one_{E_{i_s}} \cdot \vert I_0  \vert \big)^{\theta_{i_1, \ldots, i_k}}  \\
&\lesssim \sum_{l_1, \ldots, l_{n+1}} 2^{-l_1\left( 1-2\alpha_1 \right)} \cdot \ldots \cdot 2^{-l_{n+1}\left(1-2\alpha_{n+1} \right)} \prod_{j=1}^{n+1}  \big( \sssize_{I_0} \one_{E_j} \big)^{\alpha_j} \cdot \vert I_0 \vert,
\end{align*}
where the $\alpha_j$ are defined by formula \eqref{def:exp-alpha_j}. Recalling that $2^{-l_j} \lesssim \sssize_{I_0} \one_{E_{l_j}}$, we deduce \eqref{eq:local-est-k-thm}. 

\begin{remark}
The condition that $k < \dfrac{n+1}{2}$, which is indispensable in the proof, becomes evident in the estimate above: if we let $l_1=\ldots =l_{n+1}=l$, then $\ds \sum_{l_1, \ldots, l_{n+1}} 2^{-l_1} \cdot \ldots \cdot 2^{-l_{n+1}} \sum_{T \in \ic T_{l_1, \ldots, l_{n+1}}} \vert I_T \vert$ becomes 
\[
\sum_{l} 2^{-l} \cdot \ldots \cdot 2^{-l} \sum_{T \in \ic T_{l_1, \ldots, l_{n+1}}} \vert I_T \vert \lesssim \sum_l 2^{-\left( n+1\right)l} \cdot 2^{2kl} \prod_{j=1}^{n+1}  \big( \sssize_{I_0} \one_{E_j} \big)^{\alpha_j} \cdot \vert I_0 \vert.
\]
This expression can be finite only if $k<(n+1)/2$.
\end{remark}

\section{Proof of Theorem \ref{thm:main-thm}}
\label{sec:proof_MainThm}

We want to prove vector-valued and sparse vector-valued estimates for the operator $T_k$. Moreover, we want a sparse domination for $\ds\big\|  \| T_k(\vec f_1, \ldots, \vec f_n) \|_{L^{R'_{n+1}}(\ii W, \mu)} \|_{L^q(w)}$, where $w \geq 0$ is a locally integrable function. In the analysis, there is a natural division between the Banach and the quasi-Banach case: that is, the case when $1 \leq \big(r^l_{n+1}  \big)' <\infty \, \forall 1 \leq l \leq m$ and that when $\big(r^{l_0}_{n+1}  \big)' <1$ for some $1 \leq l_0 \leq m$. In the former situation, the space $L^{R'_{n+1}}(\ii W, \mu)$ is Banach, it has a dual, and the operator can be studied through the associated $(n+1)$-linear form.

In the quasi-Banach case, our approach relies on methods developed in \cite{quasiBanachHelicoid}: we cannot use the multilinear form directly, but ``dualization" (in a restricted-type sense) through some $L^{r^{j_0}}$ space is still possible. Further, we need to analyze separately the situation when $\ds q \leq r^{j_0}:=\min (1, \min_{1 \leq l \leq m} \big(r^l_{n+1}  \big)' )$ and that when $q>r^{j_0}$. The latter will be deduced from the first one in Proposition \ref{prop:no-subadd}. 

We point out that for the sparse estimate \eqref{eq:sparse-T_k-all-tau} of Theorem \ref{thm:main-thm}, the case $q=1$ is equivalent, in the Banach case, to finding a sparse domination of the multilinear form, and it will be detailed in the next section.   

Also, in order to facilitate the presentation, in this section we simply denote the operator $T_k$ by $T$. Many of the results that we present here generalize to other operators, once the corresponding local estimates are known.

\subsection{The Banach case: $1 \leq R'_{n+1}<\infty$}
\label{sec:Banach-case-T1}
Now we present the proof of Theorem \ref{thm:main-thm}, in the Banach case, when $q=1$. This will be done inductively, and in the end it is the \emph{scalar} local estimate \eqref{eq:local-est-k-thm} that implies both the vector-valued and the (vector-valued) sparse estimate. In this situation, it suffices to study the multilinear form associated to the operator $T_k$.

We consider the tuple $(\alpha_1, \ldots, \alpha_{n+1})$ as in \eqref{def:exp-alpha_j} to be fixed. For any $1 \leq j \leq n+1$, let $R_j=(r_j^1, \ldots, r_j^{m+1})$ be $(m+1)$-tuples satisfying \eqref{eq:cond-Leb-exp-vv-T}. We note that $R_j=(r_j^1, \tilde R_j)$, where $\tilde R_j$ is an $m$-tuple (this will be useful in the inductive proof), and hence, the depth $m+1$-space $L^{R_j}(\ii W, \mu)$ will be denoted $X_j$, while $\tilde X_j$ will stand for the depth-$m$ space $L^{\tilde R_j}( \tilde {\ii W}, \tilde \mu)$ (and $m=0$ corresponds to the scalar case).

Assuming the depth-$m$ localization result
\begin{equation}
\label{eq:local-sizes-m} \tag{loc $m$}
\vert \Lambda_{\rr P \left(I \right)}^m (\vec f_1, \ldots, \vec f_{n+1})   \vert \lesssim \prod_{j=1}^{n+1} \big( \sssize_{I} \one_{E_j}   \big)^{1-\alpha_j-\epsilon} \cdot \vert I \vert,
\end{equation}
that holds for any given interval $I$ and any vector-valued functions $\vec f_1, \ldots \vec f_{n+1}$ satisfying $\| \vec f_j(x, \cdot)  \|_{\tilde X_j} \leq \one_{E_j}$ for all $1 \leq j \leq n+1$ (here $E_1, \ldots, E_{n+1}$ are sets of finite measure), we will prove the following inequalities:
\begin{itemize}
\item[(i)] for any dyadic interval $I_0$ and any vector-valued function $\vec g_1, \ldots \vec g_{n+1}$ satisfying $\| \vec g_j(x, \cdot)  \|_{ X_j} \leq \one_{G_j}$ for all $1 \leq j \leq n+1$, we have a depth-$(m+1)$ localization result:
\begin{equation}
\label{eq:local-sizes-m+1}\tag{loc $m+1$}
\vert \Lambda_{\rr P \left(I_0 \right)}^{m+1} (\vec g_1, \ldots, \vec g_{n+1})   \vert \lesssim \prod_{j=1}^{n+1} \big( \sssize_{I_0} \one_{G_j}   \big)^{1-\alpha_j-\epsilon} \cdot \vert I_0 \vert
\end{equation}
\begin{remark}
The $\epsilon$ in \eqref{eq:local-sizes-m+1} might differ from the $\epsilon$ in \eqref{eq:local-sizes-m}, but in essence it represents a very small loss.
\end{remark}

\item[(ii)] the operator $T$ defined through the $(n+1)$ linear form in \eqref{eq:def-op-T_m} satisfies the depth $m$ vector valued inequality:
\begin{equation} \tag{VV $m$}
\label{eq:VV-depth-m}
\big\| T(\vec f_1, \ldots, \vec f_{n})  \big\|_{L^{p'_{n+1}}( \rr R; \tilde X'_{n+1})} \lesssim \prod_{j=1}^{n} \big\|\vec f_j \big\|_{L^{p_j}( \rr R; \tilde X_{j})}
\end{equation}
where the Lebesgue exponents satisfy \eqref{eq:Holder-tuples-for-multilinear-op}, \eqref{eq:Holder-again} and \eqref{eq:cond-Leb-exp-vv-T}.
\item[(iii)] there exists a sparse collection $\ic S$ depending on $\vec f_1, \ldots, \vec f_{n+1}$ and on $s_1, \ldots, s_{n+1}$, so that
\begin{equation}
\label{eq:SParse-depth-m}\tag{VV sparse $m$}
\big| \Lambda_{\rr P}(\vec f_1, \ldots, \vec f_{n+1}) \big|\lesssim \sum_{Q \in \ic S} \prod_{j=1}^{n+1}\big( \frac{1}{|Q|} \int_{\rr R} \| \vec f_j(x, \cdot) \|_{\tilde X_j}^{s_j} \cdot \ci_{Q}^M dx \big)^{1/{s_j}}  \cdot |Q|,
\end{equation}
given that the Lebesgue exponents satisfy $\frac{1}{s_j}<1-\alpha_j \quad \forall 1 \leq j \leq n+1$.
\end{itemize}

The easiest inequality to prove is \eqref{eq:SParse-depth-m}; using Proposition \ref{prop:averages-restr-gen}, we obtain the equivalent of \eqref{eq:local-sizes-m} for general vector-valued functions, and then we just need to apply the stopping time described in Section \ref{sec:local->sparse}. This inequality also implies \eqref{eq:VV-depth-m} for $1 < p_1, \ldots, p_{n+1} \leq\infty$, just by making use of disjointness of the sets $\ds \lbrace E_Q \rbrace_{Q \in \ic S}$: assuming \eqref{eq:SParse-depth-m}, we have that
\begin{align*}
\big| \Lambda_{\rr P}(\vec f_1, \ldots, \vec f_{n+1}) \big| &\lesssim \sum_{Q \in \ic S} \prod_{j=1}^{n+1} \inf_{y \in Q} \ic M_{s_j}(\| \vec f_j(x, \cdot) \|_{\tilde X_j}) \cdot \vert E_Q \vert \\
&\lesssim \int_{\rr R} \prod_{j=1}^{n+1} \ic M_{s_j}(\| \vec f_j(x, \cdot) \|_{\tilde X_j}) dx \lesssim \prod_{j=1}^{n+1} \big\|\vec f_j \big\|_{L^{p_j}( \rr R; \tilde X_{j})},
\end{align*}
which is true as long as $s_j < p_j$.

We are left with proving $\eqref{eq:local-sizes-m} \Rightarrow \eqref{eq:local-sizes-m+1}$. In the Banach case (that is, when all Lebesgue exponents are contained inside $(1, \infty)$), we present an approach that makes use of a local version of the estimate \eqref{eq:SParse-depth-m}. In \cite{vv_BHT} however, we used a different stopping time, as presented in Section \ref{sec:stopping-time-helicoidal-method}; the implication $\eqref{eq:local-sizes-m} \Rightarrow \eqref{eq:local-sizes-m+1}$ can be proved in a way that does not use the sparseness property explicitly, but here we want to emphasize the connection between \eqref{eq:VV-depth-m} and \eqref{eq:SParse-depth-m}.

Let $G_1, \ldots, G_{n+1}$ be sets of finite measure, $I_0$ a fixed dyadic interval, and $\vec g_1, \ldots, \vec g_{n+1}$ vector-valued functions satisfying $\|  \vec g_j(x, \cdot ) \|_{L^{R_j}(\ii W, \mu)} \leq \one_{G_j}$.

Since we argue by induction, we will denote by $\vec g_{j, w_1}$ the vector-valued function with the $w_1$ component fixed. We note that these functions also satisfy $\vec g_{j, w_1}(x, \tilde w) =\one_{G_j}(x) \cdot \vec g_{j, w_1}(x, \tilde w)$, where $\tilde w= (w_2, \ldots, w_{n+1})$. Then we use a local version of \eqref{eq:SParse-depth-m}, which can easily be verified because our operator is local. Given $\vec g_1, \ldots, \vec g_{n+1}$, $G_1, \ldots, G_{n+1}$ as above, and $I$ a dyadic interval, there exists a collection $\ic S(I)$ of sparse intervals contained inside $I$ so that
{\fontsize{10}{10}\begin{equation}
\label{eq:Sparse-depth-m-local}\tag{sparse loc $m$}
\big| \Lambda_{\rr P(I)}(\vec g_{1, w_1} \cdot \one_{G_1} , \ldots, \vec g_{n+1, w_1} \cdot \one_{G_{n+1}}) \big|\lesssim \prod_{j=1}^{n+1} \big( \sssize_{I} \one_{G_j}  \big)^{\frac{1}{s_j}-\frac{1}{\tau_j}}  \sum_{Q \in \ic S(I)} \prod_{j=1}^{n+1}\big( \frac{1}{|Q|} \int_{\rr R} \| \vec g_{j, w_1}(x, \cdot) \|_{\tilde X_j}^{\tau_j} \cdot \ci_{Q}^M dx \big)^{1/{\tau_j}}  \cdot |Q|,
\end{equation}}
where $\ds \frac{1}{\tau_j} < \frac{1}{s_j}=1-\alpha_j -\epsilon$ for all $1 \leq j \leq n+1$. Here the $\tau_j$, upon which the sparse domination depends (together with the functions $\vec g_{j, w_1}$), are so that $\ds \frac{1}{\tau_j}<\min \big( \frac{1}{s_j}, \frac{1}{r_j^1} \big) < 1-\alpha_j$.

As before, by making use of the sparse property of the collection, we can prove that
\[
\big| \Lambda_{\rr P(I)}(\vec g_{1, w_1} \cdot \one_{G_1}, \ldots, \vec g_{n+1, w_1} \cdot \one_{G_{n+1}}) \big|\lesssim \prod_{j=1}^{n+1} \big( \sssize_{I} \one_{G_j}  \big)^{\frac{1}{s_j}-\frac{1}{\tau_j}} \cdot \prod_{j=1}^{n+1} \| \| \vec g_{j, w_1}(x, \cdot) \|_{\tilde X_j} \cdot \ci_{I}^M \|_{L^{r_j^1}}.
\]

Since we make appear the maximal operator $\ic M_{\tau_j}$, these exponents have to be slightly smaller than $r_j^1$; so in fact the operatorial norm can be written as 
\[
\prod_{j=1}^{n+1} \big( \sssize_{I_0} \one_{G_j}  \big)^{\frac{1}{s_j}-\frac{1}{r_j^1}-\epsilon},
\]
similarly to what we would get by applying directly the helicoidal method.

Now we return to the $(n+1)$-linear form:
\begin{align*}
&\vert \Lambda_{\rr P \left(I_0 \right)}(\vec g_1, \ldots, \vec g_{n+1}) \vert =\vert \int_{\ii W_1} \Lambda_{\rr P \left(I_0 \right)}(\vec g_{1, w_1}\cdot \one_{G_1}, \ldots, \vec g_{n+1, w_1}\cdot \one_{G_{n+1}}   ) d w_1\vert\\
& \lesssim \int_{\ii W_1} \prod_{j=1}^{n+1} \big( \sssize_{I_0} \one_{G_j}  \big)^{\frac{1}{s_j}-\frac{1}{r_j^1}-\epsilon} \cdot \prod_{j=1}^{n+1} \big\|  \| g_{j, w_1} (x, \cdot)  \|_{\tilde X_j} \cdot \ci_{I_0}^M  \big\|_{L^{r^1_j}_x} d w_1.
\end{align*}

By using H\"older's inequality and Fubini, we get that
\begin{align*}
\vert \Lambda_{\rr P \left(I_0 \right)}(\vec g_1, \ldots, \vec g_{n+1}) \vert &\lesssim \prod_{j=1}^{n+1} \big( \sssize_{I_0} \one_{G_j}  \big)^{\frac{1}{s_j}-\frac{1}{r_j^1}-\epsilon} \cdot \prod_{j=1}^{n+1} \big\|  \| g_{j} (x, \cdot)  \|_{ X_j} \cdot \ci_{I_0}^M  \big\|_{L^{r^1_j}_x} \\
& \lesssim \prod_{j=1}^{n+1} \big( \sssize_{I_0} \one_{G_j}  \big)^{\frac{1}{s_j}-\epsilon} \cdot |I_0|,
\end{align*}
which is similar to \eqref{eq:local-sizes-m+1} since $\ds \frac{1}{s_j}=1-\alpha_j -\epsilon$.

\begin{remark}
If $q<1$, the sparse domination estimate \eqref{eq:sparse-T_k-all-tau}, in the Banach case, will be discussed in the following section. We recall that the case $q=1$ corresponds to the study of the $(n+1)$-linear form.
\end{remark}

\subsection{The quasi-Banach case}
\label{sec:quasi-Banach}
We consider separately the case when there exists $1 \leq l_0 \leq m$ so that $r_{n+1}^{l_0} <1$ (in this situation, $r^{j_0}<1$). The difference now is that we cannot obtain the full range of exponents just by considering the $(n+1)$-linear form. Instead, we will use certain inequalities similar to those in \cite{quasiBanachHelicoid}.


We look at the scalar case first, corresponding to $m=0$. The estimates obtained for the multilinear form imply immediately that
\[
\|  T_{I_0}(f_1, \ldots, f_n) \cdot \one_{\tilde E_{n+1}} \|_1 \lesssim \prod_{j=1}^n \big( \sssize_{I_0} \one_{E_j}   \big)^{1-\alpha_j-\epsilon} \cdot \big( \sssize_{I_0} \one_{\tilde E_{n+1}}  \big)^{1-\alpha_{n+1}-\epsilon} \cdot \vert I_0  \vert,
\]
for functions $f_1, \ldots, f_{n+1}$ satisfying $\vert f_j(x) \vert \leq \one_{E_j}(x)$ for all $1 \leq j \leq n$. Here $E_1, \ldots, E_n, \tilde E_{n+1}$ are sets of finite measure.

Under the same assumptions, for any $\tau <1$, we get a ``\emph{better}" estimate:
\[
\|  T_{I_0}(f_1, \ldots, f_n) \cdot \one_{\tilde E_{n+1}} \|_\tau \lesssim \prod_{j=1}^n \big( \sssize_{I_0} \one_{E_j}   \big)^{1-\alpha_j-\epsilon} \cdot \big( \sssize_{I_0} \one_{\tilde E_{n+1}}  \big)^{1-\alpha_{n+1}-\epsilon-\frac{1}{\tau'}} \cdot \vert I_0  \vert^{\frac{1}{\tau}}.
\]

Because of the exponent of $\sssize_{I_0} \, \one_{\tilde E_{n+1}}:$ $1-\alpha_{n+1}-\epsilon-\frac{1}{\tau'} > 1-\alpha_{n+1}-\epsilon$ (here $\tau' <0$), this estimate is an improvement of what we could deduce by using the multilinear form exclusively.

The inductive multiple vector-valued inequalities that we obtain in this situation are the following:
{\fontsize{9}{10}\begin{equation}
\label{eq:vv-local-qB}
\big \|  \big\|  T_{I_0}(\vec f_1 \cdot \one_{E_1}, \ldots, \vec f_n \cdot \one_{E_n}) \|_{L^{R'_{n+1}}}\cdot \one_{\tilde E_{n+1}}  \big\|_{p'_{n+1}} \lesssim \prod_{j=1}^n \big(  \sssize_{I_0} \one_{E_j}  \big)^{1-\alpha_j -\epsilon -\frac{1}{p_j}} \cdot \big(  \sssize_{I_0} \one_{\tilde E_{n+1}}  \big)^{1-\alpha_{n+1} -\epsilon -\frac{1}{p_{n+1}}} \cdot \prod_{j=1}^n \big\|  \|  \vec f_j  \|_{\tilde X_j} \cdot \ci_{I_0}^M \big\|_{p_j}.
\end{equation}}
This applies to the case when $p_1, \ldots, p_{n+1}$ satisfy \eqref{eq:cond-Leb-exp-vv-T}; note that here $p_{n+1}' <1$ and hence its harmonic conjugate $p_{n+1}$ is negative.

When the functions satisfy $\| \vec f_j(x, \cdot)  \|_{\tilde X_j} \leq \one_{E_j}$, this becomes
\begin{equation}
\label{eq:vv-local-qB-restric} \tag{qB : $m$}
\big \|  \big\|  T_{I_0}(\vec f_1, \ldots, \vec f_n ) \|_{L^{R'_{n+1}}}\cdot \one_{\tilde E_{n+1}}  \big\|_{p'_{n+1}} \lesssim \prod_{j=1}^n \big(  \sssize_{I_0} \one_{E_j}  \big)^{1-\alpha_j -\epsilon} \cdot \big(  \sssize_{I_0} \one_{\tilde E_{n+1}}  \big)^{1-\alpha_{n+1} -\epsilon -\frac{1}{p_{n+1}}} \cdot \vert I_0 \vert^{\frac{1}{p'_{n+1}}}.
\end{equation}

Using a stopping time as in section \ref{sec:stopping-time-helicoidal-method} or as in \cite{quasiBanachHelicoid}, we obtain 
\begin{equation}
\label{eq:vv-qB-m}\tag{qB VV $m$}
\big\| T(\vec f_1, \ldots, \vec f_{n})  \big\|_{L^{p_{n+1}'}\left( \rr R; L^{R_{n+1}'} \right)} \lesssim \prod_{j=1}^{n} \big\| \vec f_j \big\|_{L^{p_j}\left( \rr R; L^{R_j} \right)},
\end{equation}
where the Lebesgue exponents $p_1, \ldots, p_{n+1}$ and the $m$-tuples $R_1, \ldots , R_{n+1}$ satisfy \eqref{eq:cond-Leb-exp-vv-T}.

We are left with proving the sparse domination result \eqref{eq:sparse-T_k-all-tau}. The inequality \eqref{eq:vv-local-qB-restric} can be extended to 
\begin{equation}
\label{eq:vv-local-qB-restric-tau}
\big \|  \big\|  T_{I_0}(\vec f_1, \ldots, \vec f_n ) \|_{L^{R'_{n+1}}}\cdot \one_{\tilde E_{n+1}}  \big\|_{\tau} \lesssim \prod_{j=1}^n \big(  \sssize_{I_0} \one_{E_j}  \big)^{1-\alpha_j -\epsilon} \cdot \big(  \sssize_{I_0} \one_{\tilde E_{n+1}}  \big)^{1-\alpha_{n+1} -\frac{1}{\tau'} -\epsilon} \cdot \vert I_0 \vert^{\frac{1}{\tau}},
\end{equation}
for any $\tau<1$. Note that in contrast the Lebesgue exponents $p_1, \ldots, p_{n+1}$ in \eqref{eq:vv-local-qB} have to satisfy \eqref{eq:Holder-again} and \eqref{eq:cond-Leb-exp-vv-T}.

As a consequence of Proposition \ref{prop:local->sparse->subadditive}, provided $\| \cdot \|_{L^{R'}_{n+1}}^\tau$ is \emph{subadditive}, the local estimate from \eqref{eq:vv-local-qB-restric-tau} can be put to use as in Section \ref{sec:local->sparse} (Proposition \ref{prop:averages-restr-gen} extends also to the context of multiple vector-valued Banach or quasi-Banach spaces; the multilinear form is replaced by a quasi-norm in the second situation) to obtain a sparse domination: for any vector-valued functions $\vec f_1, \ldots, \vec f_n$ so that $\big\|  \vec f_1(x, \cdot) \big\|_{ X_1}, \ldots, \big\|  \vec f_n(x, \cdot) \big\|_{ X_n}$ are locally integrable, and any locally $\tau$-integrable function $v$ , there exists a sparse collection $\ic S$ of dyadic intervals (depending on the functions $\vec f_j, v$ and on the Lebesgue exponents $s_j$ below) so that
{\fontsize{10}{10}\begin{equation}
\label{eq:sparse-qB}
\big\|  \big\|T(\vec f_1, \ldots, \vec f_n)\big\|_{L^{R'}_{n+1}} \cdot v \big\|_\tau^\tau \lesssim \sum_{Q \in \ic S} \prod_{j=1}^{n} \big( \frac{1}{\vert Q \vert} \int_{\rr R} \big\|  \vec f_j(x, \cdot) \big\|_{ X_j}^{s_j} \cdot \ci_{Q}^{M} dx  \big)^{\frac{\tau}{s_j}} \cdot \big( \frac{1}{\vert Q \vert} \int_{\rr R} |v(x)|^{s_{n+1}} \cdot \ci_{Q}^{M} dx  \big)^{\frac{\tau}{s_{n+1}}}  \cdot \vert Q \vert
\end{equation}} 
for any $s_j$ so that $\frac{1}{s_j} <1 -\alpha_j$ for all $1 \leq j \leq n$, $\frac{1}{s_{n+1}} <\frac{1}{\tau} -\alpha_{n+1}$.

Subadditivity is achieved if $\ds \tau< r^{j_0}:=\min( 1, \min_{ 1 \leq l \leq m} (r^l)')$ (here we apply directly Proposition \ref{prop:local->sparse->subadditive} since $\|  \cdot \|_{L^{R'_{n+1}}}^{r^{j_0}}$ is subadditive), but it doesn't hold in general. We can overcome this situation by using the following proposition:

\begin{proposition}
\label{prop:no-subadd}
Let $T$ a multilinear operator, and $ X_j=L^{R_j}(\ii W, \mu) \, \forall 1 \leq j \leq n$, $ X_{n+1}=L^{R'_{n+1}}(\ii W, \mu)$ multiple vector-valued spaces as above (that is, the Lebesgue exponents $r_j^l, (r^l)'$ satisfy condition \eqref{eq:cond-Leb-exp-vv-T} for a fixed tuple $(\alpha_1, \ldots, \alpha_{n+1})$). Let $\tau \geq r^{j_0}$. Given functions $\vec f_1, \ldots, \vec f_n$ so that $\big\|  \vec f_1(x, \cdot) \big\|_{ X_1}, \ldots, \big\|  \vec f_n(x, \cdot) \big\|_{ X_n}$ are locally integrable, and $v$ any locally $\tau$-integrable function, there exists a sparse domination collection $\ic S$ of dyadic intervals (depending on the functions $\vec f_j, v$ and on the Lebesgue exponents $s_j, r_j^l, (r^l)'$ below) so that
\begin{equation*}
\big\|  \big\|T(\vec f_1, \ldots, \vec f_n)\big\|_{X_{n+1}} \cdot v \big\|_\tau^\tau \lesssim \sum_{Q \in \ic S} \prod_{j=1}^{n} \big( \frac{1}{\vert Q \vert} \int_{\rr R} \big\|  \vec f_j(x, \cdot) \big\|_{ X_j}^{s_j} \cdot \ci_{Q}^{M} dx  \big)^{\frac{\tau}{s_j}} \cdot \big( \frac{1}{\vert Q \vert} \int_{\rr R} |v(x)|^{\tilde s_{n+1}} \cdot \ci_{Q}^{M} dx  \big)^{\frac{\tau}{\tilde s_{n+1}}}  \cdot \vert Q \vert.
\end{equation*} 
for any $s_j$ so that $\frac{1}{s_j} <1 -\alpha_j$ for all $1 \leq j \leq n$, $\frac{1}{\tilde s_{n+1}} <\frac{1}{\tau} -\alpha_{n+1}$.
\begin{proof}
We lack subadditivity because $\tau > r^{j_0}$; however, this means that we can dualize the $\ds L^\frac{\tau}{r^{j_0}}$ norm. We start with the observation that
\[
\big\|  \big\|T(\vec f_1, \ldots, \vec f_n)\big\|_{X_{n+1}} \cdot v \big\|_\tau^\tau= \big\|  \big\|T(\vec f_1, \ldots, \vec f_n)\big\|_{X_{n+1}}^{r^{j_0}} \cdot v^{r^{j_0}} \big\|_{\frac{\tau}{r^{j_0}}}^\frac{\tau}{r^{j_0}}=\big( \int_{\rr R} \big\|T(\vec f_1, \ldots, \vec f_n)\big\|_{X_{n+1}}^{r^{j_0}} \cdot v^{r^{j_0}} \cdot u \,dx\big)^\frac{\tau}{r^{j_0}},
\]
where $u \in  L^{(\frac{\tau}{r^{j_0}})'}$, with $\ds \| u \|_{(\frac{\tau}{r^{j_0}})'}=1$. Note that the function $u$ depends on the previous data (the function $\vec f_1, \ldots, \vec f_n, v$, the Lebesgue exponents).

Moreover, 
\[
\int_{\rr R} \big\|T(\vec f_1, \ldots, \vec f_n)\big\|_{X_{n+1}}^{r^{j_0}} \cdot v^{r^{j_0}} \cdot u\, dx =\big\| \big\|T(\vec f_1, \ldots, \vec f_n)\big\|_{X_{n+1}} \cdot v \cdot u^\frac{1}{r^{j_0}}   \big\|_{r^{j_0}}^{r^{j_0}},
\]
for which we can apply the sparse estimate for the exponent $r^{j_0}<1$. Then $\| \| \cdot\|_{X_{n+1}}  \|_{r^{j_0}}^{r^{j_0}}=\| \| \cdot\|_{L^{R'_{n+1}}}  \|_{r^{j_0}}^{r^{j_0}}$ is subadditive (by Proposition 7 of \cite{quasiBanachHelicoid}) and we can deduce the sparse domination
{\fontsize{10}{10}\begin{equation*}
\big\|  \big\|T(\vec f_1, \ldots, \vec f_n)\big\|_{X_{n+1}} \cdot v \cdot u^\frac{1}{r^{j_0}} \big\|_{r^{j_0}}^{r^{j_0}} \lesssim \sum_{Q \in \ic S} \prod_{j=1}^{n} \big( \frac{1}{\vert Q \vert} \int_{\rr R} \big\|  \vec f_j(x, \cdot) \big\|_{ X_j}^{s_j} \cdot \ci_{Q}^{M} dx  \big)^{\frac{r^{j_0}}{s_j}} \cdot \big( \frac{1}{\vert Q \vert} \int_{\rr R} |v(x) \cdot u^\frac{1}{r^{j_0}}|^{ s_{n+1}} \cdot \ci_{Q}^{M} dx  \big)^{\frac{r^{j_0}}{ s_{n+1}}}  \cdot \vert Q \vert,
\end{equation*}} 
where, in this case, $s_{n+1}$ satisfies $\ds \frac{1}{s_{n+1}}<\frac{1}{r^{j_0}}-\alpha_{n+1}$.

Now we use H\"older's inequality for the $L^{s_{n+1}}$ average, but we need to be cautious about the Lebesgue exponents: we will have 
\[
\big( \frac{1}{\vert Q \vert} \int_{\rr R} |v(x) \cdot u^\frac{1}{r^{j_0}}|^{ s_{n+1}} \cdot \ci_{Q}^{M} dx  \big)^{\frac{1}{ s_{n+1}}} \leq \big( \frac{1}{\vert Q \vert} \int_{\rr R} |v(x) |^{ p_1} \cdot \ci_{Q}^{M} dx  \big)^{\frac{1}{ p_1}} \cdot \big( \frac{1}{\vert Q \vert} \int_{\rr R} |u^\frac{1}{r^{j_0}}|^{ p_2} \cdot \ci_{Q}^{M} dx  \big)^{\frac{1}{p_2}},
\]
where $\ds \frac{1}{s_{n+1}}=\frac{1}{p_1}+\frac{1}{p_2}$ and $\frac{p_2}{r^{j_0}}<\big(\frac{\tau}{r^{j_0}}\big)'$. The latter condition (the necessity of which will be made clear soon) is equivalent to $\ds \frac{1}{p_2}>\frac{1}{r^{j_0}}-\frac{1}{\tau}$, so we can set  $\ds \frac{1}{p_2}=\frac{1}{r^{j_0}}-\frac{1}{\tau}+\epsilon$, where $\epsilon >0$ is arbitrarily small.

We also want to use H\"older's inequality for the spaces $\ds \ell^{\frac{\tau}{r^{j_0}}}$ and $\ds \ell^{\big(\frac{\tau}{r^{j_0}}\big)'}$, indexed after the sparse collection of dyadic intervals $\ic S$. In this way, we have 
\begin{align*}
& \big\|  \big\|T(\vec f_1, \ldots, \vec f_n)\big\|_{X_{n+1}} \cdot v \cdot u^\frac{1}{r^{j_0}} \big\|_{r^{j_0}}^{r^{j_0}} \\
& \lesssim \Big( \sum_{Q \in \ic S} \prod_{j=1}^{n} \big( \frac{1}{\vert Q \vert} \int_{\rr R} \big\|  \vec f_j(x, \cdot) \big\|_{ X_j}^{s_j} \cdot \ci_{Q}^{M} dx  \big)^{\frac{\tau}{s_j}} \cdot \big( \frac{1}{\vert Q \vert} \int_{\rr R} |v(x)|^{ p_1} \cdot \ci_{Q}^{M} dx  \big)^{\frac{\tau}{ p_1}} |Q| \Big)^{\frac{r^{j_0}}{\tau}} \\
&\qquad \cdot \Big( \sum_{Q \in \ic S} \big( \frac{1}{\vert Q \vert} \int_{\rr R} |u(x)|^{\frac{p_2}{r^{j_0}}} \cdot \ci_{Q}^{M} dx  \big)^{\frac{r^{j_0}}{ p_2} \cdot \big(\frac{\tau}{r^{j_0}} \big)'} |Q| \Big)^{\frac{1}{\big(\frac{\tau}{r^{j_0}}\big)'}}.
\end{align*}

Using the sparseness properties of the collection $\ic S$, the term on the last line can be bounded by
\[
\Big\|  \ic M_{\frac{p_2}{r^{j_0}}} (u) \Big\|_{\big(\frac{\tau}{r^{j_0}}\big)'} \lesssim \| u \|_{\big(\frac{\tau}{r^{j_0}}\big)'}=1,
\]
since $p_2$ was chosen so that $\frac{p_2}{r^{j_0}}<\big(\frac{\tau}{r^{j_0}}\big)'$. We need to observe that our choice of $p_2$ implies that
\[
\frac{1}{p_1}=\frac{1}{s_{n+1}}-\frac{1}{p_2}=\frac{1}{s_{n+1}}-\frac{1}{r^{j_0}}+\frac{1}{\tau}-\epsilon.
\]
Hence, if we set $\ds \frac{1}{\tilde s_{n+1}}=\frac{1}{p_1}=\frac{1}{s_{n+1}}-\frac{1}{r^{j_0}}+\frac{1}{\tau}$, then we have $\ds \frac{1}{\tilde s_{n+1}}<\frac{1}{\tau}-\alpha_{n+1}$, which is what we wanted.

As a matter of fact, we have obtained that 
{\fontsize{10}{10}\[
\big\|  \big\|T(\vec f_1, \ldots, \vec f_n)\big\|_{X_{n+1}} \cdot v \big\|_\tau^\tau \lesssim \sum_{Q \in \ic S} \prod_{j=1}^{n} \big( \frac{1}{\vert Q \vert} \int_{\rr R} \big\|  \vec f_j(x, \cdot) \big\|_{ X_j}^{s_j} \cdot \ci_{Q}^{M} dx  \big)^{\frac{\tau}{s_j}} \cdot \big( \frac{1}{\vert Q \vert} \int_{\rr R} |v(x)|^{\tilde s_{n+1}} \cdot \ci_{Q}^{M} dx  \big)^{\frac{\tau}{ \tilde s_{n+1}}} |Q|. 
\]}
\end{proof}
\end{proposition}

\begin{remark}
We note that a consequence of Proposition \ref{prop:no-subadd} is that the $L^{s_{n+1}}$ average of $v$ in the sparse domination of $\ds \| \|T (\vec f_1, \ldots, \vec f_n)\|_{L^{R'_{n+1}}}\cdot v \|_q^q$, for $q \leq r^{j_0}$, can be replaced by an $L^{\tilde s_{n+1}}$ average in the sparse domination of $\ds \| \|T (\vec f_1, \ldots, \vec f_n)\|_{L^{R'_{n+1}}}\cdot v \|_\tau^\tau$ if $\tau >q$, where $\frac{1}{\tilde s_{n+1}}=\frac{1}{s_{n+1}}-\frac{1}{q}+\frac{1}{\tau}$.
\end{remark}

\subsection{Weighted estimates}
\label{sec:weighted-est}

Using the Fefferman-Stein vector-valued inequality from Corollary \ref{cor:fefferman-stein-general}, we can deduce weighted estimates for the multiple vector-valued extensions of $T_k$. Alternatively, weighted estimates can be deduced directly from the sparse form, which would probably imply better quantitative estimates. Instead, we use the Fefferman-Stein inequality, which is itself deduced from the sparse domination.

We recall the multi(sub)-linear maximal operator $\vec {\ic M}_{s_1, \ldots, s_n}$, defined in \eqref{eq:def-func-max-multilin} by
\[
\vec {\ic M}_{s_1, \ldots, s_n}(f_1, \ldots, f_n)(x):=\sup_{Q \ni x} \prod_{j=1}^n  \big(  \frac{1}{|Q|}  \int_{Q}  |f_j(y)|^{s_j} dy  \big)^\frac{1}{s_j}.
\]
We use the inequality
\[
\big\|  \| T_k( \vec f_1, \ldots, \vec f_n)  \|_{L^{R'_{n+1}}} \big\|_{L^q(w^q)} \lesssim  \big\|  \vec {\ic M}_{s_1, \ldots, s_n} ( \| \vec f_1 (x,\cdot)  \|_{L^{R_1}}, \ldots, \| \vec f_n (x,\cdot)  \|_{L^{R_n}} )  \big\|_{L^q(w^q)},
\]
which holds true for $s_1, \ldots, s_{n+1}, q$ as in Corollary \ref{cor:fefferman-stein-general}, provided $w^q \in RH_{\frac{s_{n+1}}{q}}$.

Weighted estimates for $\vec {\ic M}_{s_1, \ldots, s_n}$ can be obtained, following closely the procedure in \cite{multiple-weights-multilinear-op}. We have the following results:

\begin{proposition}
\label{propo:weighted-est-weak-type-multi-max}
Let $\frac{1}{n}<q \leq \infty$, $1\leq s_j<q_j \leq \infty$, $1 \leq j \leq n$, and $\frac{1}{q}=\frac{1}{q_1}+\ldots + \frac{1}{q_n}$. Let $\nu$ and $w_j$ be weights (here we set $w_j \equiv 1$ if $p_j=\infty$). Then the inequality 
\[
\big \| \vec {\ic M}_{s_1, \ldots, s_n}(f_1, \ldots, f_n) \big\|_{L^{q, \infty}(\nu)} \leq C \prod_{j=1}^n \| f_j \|_{L^{q_j}(w_j)}
\]  
holds for any functions $f_1, \ldots, f_n$ if and only if 
\begin{equation}
\label{eq:cond-weak-weight-multi-max}
\sup_{Q} \big( \aver{Q} \nu \big)^\frac{1}{q} \, \prod_{j=1}^n \Big(  \aver{Q} w_j^{- \frac{s_j}{q_j-s_j}} \Big)^{\frac{1}{s_j}-\frac{1}{q_j}}<+\infty.
\end{equation}
\begin{proof}
For the necessity, we note that
\[
Q \subseteq \lbrace x: \vec {\ic M}_{s_1, \ldots, s_n}(f_1, \ldots, f_n)(x)> \prod_{j=1}^n  \big( \aver{Q} |f_j|^{s_j}\big)^\frac{1}{s_j}    \rbrace,
\]
and hence
\[
\big( \aver{Q} \nu \big)^\frac{1}{q} \, \prod_{j=1}^n \big( \aver{Q} |f_j|^{s_j} dx \big)^\frac{1}{s_j} \leq C \, \big( \aver{Q} |f_j|^{q_j} \, w_j dx \,\big)^\frac{1}{q_j}.
\]
Picking $f_j=w_j^\frac{1}{s_j-q_j}$ so that $|f_j|^{s_j}=|f_j|^{q_j} w_j$, we obtain \eqref{eq:cond-weak-weight-multi-max}, since $Q$ was arbitrary.

We adapt the proof from \cite{multiple-weights-multilinear-op} in order to show sufficiency. A straightforward application of H\"older's inequality yields that 
\[
\vec {\ic M}_{s_1, \ldots, s_n}(f_1, \ldots, f_n)(x) \leq \sup_{Q} \big( \aver{Q} \nu \big)^\frac{1}{q} \, \prod_{j=1}^n \Big(  \aver{Q} w_j^{- \frac{s_j}{q_j-s_j}} \Big)^{\frac{1}{s_j}-\frac{1}{q_j}}  \cdot \prod_{j=1}^n \ic M_{\nu} \Big( |f_j|^{q_j} \frac{w_j}{\nu}  \Big)^\frac{1}{q_j}. 
\]

H\"older's inequality for weak $L^p$ spaces and the boundedness of the weighted maximal operator $\ic M_\nu$ imply the $L^{q_1}(w_1) \times \ldots \times L^{q_n}(w_n) \to L^{q, \infty}(w)$ boundedness of $\vec {\ic M}_{s_1, \ldots, s_n}$. Whenever $q_j=\infty$, the $L^\infty$ norm of $f_j$ is set on the side, and the problem reduces to the study of an $(n-1)$-linear maximal operator.
\end{proof}
\end{proposition}

\begin{remark}
Since our study of the $T_k$ operators is focusing on $L^{p_1}\times \ldots \times L^{p_n} \mapsto L^{p'_{n+1}}$ for $1<p_1, \ldots, p_n \leq \infty$, we only consider the boundedness of $\vec {\ic M}_{s_1, \ldots, s_n}$ operators on $L^{q_1} \times \ldots \times L^{q_n}$ for $q_j>s_j$: for simplicity, we leave out the case $q_j=s_j$.
\end{remark}

However, the operators $T_k$ and their vector-valued extensions are controlled by $\vec {\ic M}_{s_1, \ldots, s_n}$ in $L^q$ norms and we need a strong-type version of the above result. We continue on following \cite{multiple-weights-multilinear-op} and adapt the definitions therein:
\begin{definition}
Let $1 \leq s_j < q_j \leq \infty$. Given $\vec w =(w_1, \ldots, w_n)$, we set $\ds \nu_{\vec w}:=\prod_{j=1}^n w_j^\frac{q}{q_j}$. We say that $\vec w$ satisfies the $A_{\vec q, \vec s}$ condition if  
\begin{equation}
\label{eq:def-vec-Aps-condition}
\sup_{Q} \big( \aver{Q} \nu_{\vec w} \big)^\frac{1}{q} \, \prod_{j=1}^n \Big(  \aver{Q} w_j^{- \frac{s_j}{q_j-s_j}} \Big)^{\frac{1}{s_j}-\frac{1}{q_j}}<+\infty.
\end{equation}
We denote the quantity in \eqref{eq:def-vec-Aps-condition} $[w_1, \ldots, w_n]_{A_{\vec q, \vec s}}$.
\end{definition} 

\begin{notation}
Above and everywhere else, whenever $q_j=\infty$, it should be understood that $w_j \equiv 1$.
\end{notation}

Then we have the strong-type boundedness of $ \vec {\ic M}_{s_1, \ldots, s_n}$:
\begin{proposition}
\label{prop-strong-est-weighted-multi-max}
Let $1\leq s_j<q_j \leq \infty$, $1 \leq j \leq n$, and $\frac{1}{q}=\frac{1}{q_1}+\ldots + \frac{1}{q_n}$. Let $\vec w=(w_1, \ldots, w_n)$ be a vector weight. Then the inequality 
\[
\big \| \vec {\ic M}_{s_1, \ldots, s_n}(f_1, \ldots, f_n) \big\|_{L^{q}(\nu_{\vec w})} \leq C \prod_{j=1}^n \| f_j \|_{L^{q_j}(w_j)}
\]  
holds for any functions $f_1, \ldots, f_n$ if and only if $\vec w$ satisfies the $A_{\vec q, \vec s}$ condition.
\end{proposition}

The proof of the above Proposition (the details of which are left to the interested reader) is based on a \emph{reverse H\"older} property of each of the weights making up the vector weight (see Theorem 3.6 in \cite{multiple-weights-multilinear-op} and Lemma 3.2 in \cite{martell-kangwei-mulilinear-weights-extrapolation}). This would yield, for some $\alpha<1$
\[
\vec {\ic M}_{s_1, \ldots, s_n}(f_1, \ldots, f_n)(x) \leq  \, C  \cdot \prod_{j=1}^n \ic M_{\nu_{\vec w}} \Big( \big(|f_j|^{q_j} \frac{w_j}{\nu_{\vec w}}\big)^{\alpha}  \Big)^\frac{1}{q_j \alpha},
\]
and such an inequality, together with the strong boundedness of the weighted (sub-)linear maximal function $\ic M_{\nu_{\vec w}}$ implies the result in Proposition \ref{prop-strong-est-weighted-multi-max}.

In fact, for proving the boundedness of $\vec {\ic M}_{s_1, \ldots, s_n}$ we need a simpler version than Lemma \ref{lemma:RH-vector-weights} below. However, we will see that the vector weight condition \eqref{eq:def-vec-Aps-condition} pairs with a certain reverse H\"older condition to yield the assumption \eqref{eq:joint-weight-condition} of Corollary \ref{cor:main-weights-vector-weights}. Hence the required result is similar to Lemma 3.2 from \cite{martell-kangwei-mulilinear-weights-extrapolation}.

\begin{lemma}
\label{lemma:RH-vector-weights}
 Let $\vec w=(w_1, \ldots, w_n)$ be a vector weight as above. Then 
\begin{equation}
\label{eq:vec-weight-alpha}
\sup_{Q} \big( \aver{Q} \nu_{\vec w}^\frac{\beta}{q} \big)^\frac{1}{\beta} \, \prod_{j=1}^n \big(  \aver{Q} w_j^{-  \frac{\beta_j}{q_j}} \big)^{\frac{1}{\beta_j}} <\infty
\end{equation}
if and only if $\nu_{\vec w}^{\frac{\beta}{q}} \in A_{t}$ with $\ds t=1+\beta \big( \sum_{i=1}^n \frac{1}{\beta_i} \big) $ and for every $1 \leq j \leq n$, $w_j^{-  \frac{ \beta_j}{q_j} } \in A_{t_j}$ where $\ds t_j=1+\beta_j \big(\frac{1}{\beta} +\sum_{i \neq j} \frac{1}{\beta_i} \big)$. Moreover, if we denote by $C$ the expression in \eqref{eq:vec-weight-alpha}, we have $[\nu_{\vec w}]_{A_t} \leq C^\beta$ and $\big[w_j^{-  \frac{ \beta_j}{q_j} }\big]_{A_{t_j}} \leq C^{\beta_j}$.
\begin{proof}
We only prove the direct implication, the reverse being an easy consequence of H\"older's inequality. Fix $1 \leq j \leq n$. We want to prove that 
\[
\aver{Q} w_j^{- \frac{\beta_j}{q_j}} \cdot \big( \aver{Q} w_j^{ \frac{ \beta_j}{q_j} ( t_j'-1 )} \big)^{\left( t_j-1 \right)} \leq C^{\beta_j}.
\]

First, note that $\ds w_j^{  \frac{ \beta_j}{q_j} (t_j'-1)}=\prod_{i=1}^n  w_i^{ \frac{ \beta_i}{q_i}\, (t_i'-1)}  \cdot \prod_{\substack{i=1\\ i \neq j}}^n w_i^{- \frac{ \beta_i}{q_i}\, (t_i'-1)}$, for numbers $t_i'$ that are to be determined if $i \neq j$. We will also need certain Lebesgue indices $r_1, \ldots, r_n$ satisfying the H\"older condition $\frac{1}{r_1}+\ldots+\frac{1}{r_n}=1$. With these, we have
\[
\aver{Q} w_j^{ \frac{ \beta_j}{q_j} (t_j'-1)} dx \leq \Big( \aver{Q} \prod_{i=1}^n w_i^{ \frac{ \beta_i}{q_i}\, r_j \, (t_i'-1)} \Big)^{\frac{1}{r_j}} \cdot \prod_{\substack{i=1 \\ i \neq j}}^n \Big( \aver{Q} w_i^{-  \frac{ \beta_i}{q_i}\, r_i \, (t_i'-1)}  \Big)^\frac{1}{r_i}.
\]
Now we want to choose the Lebesgue exponents so that 
\[
\frac{ \beta_i}{q_i}\, r_j \, (t_i'-1) =\frac{q}{q_i} \cdot \frac{\beta}{q} \quad \text{for all }\quad 1\leq i \leq n, \qquad \frac{ \beta_i}{q_i}\, r_i \, (t_i'-1)=\frac{ \beta_i}{q_i} \quad \text{for all }\quad 1\leq i \leq n, i \neq j.
\]
Such a choice is possible provided $\ds r_i= r_j \frac{\beta_i}{\beta}$ for $i \neq j$ and $ r_j =\frac{\beta \,(t_j-1)}{\beta_j}$. Moreover, the condition $\ds \sum_{i=1}^n \frac{1}{r_i}=1$ determines $t_j$ uniquely: it is given by $\ds t_j=1+\beta_j \big(\frac{1}{\beta} +\sum_{i \neq j} \frac{1}{\beta_i} \big)$. 
In order to prove the desired estimate  $\big[w_j^{-  \frac{ \beta_j}{q_j} }\big]_{A_{t_j}} \leq C^{\beta_j}$, we only need to check that $\dfrac{t_j-1}{r_j}=\dfrac{\beta_j}{\beta}$ and $\dfrac{t_j-1}{r_i}=\dfrac{\beta_j}{\beta_i}$, which are immediate consequences of the above choices. The estimate for $\nu_{\vec w}^{\frac{\beta}{q}}$ is very similar and we skip the details.
\end{proof}
\end{lemma} 

Now we are ready to deduce the weighted estimates for $T_k$.

\begin{proof}[Proof of Corollary \ref{cor:main-weights-vector-weights}]
The exponents $s_1, \ldots, s_n, s_{n+1}$ satisfy, for $\alpha_1, \ldots, \alpha_{n+1}$ given by \eqref{def:exp-alpha_j}
\[
\frac{1}{s_j}<1-\alpha_j \quad \text{for all }\quad 1 \leq j \leq n, \qquad \frac{1}{s_{n+1}}<\frac{1}{q}-\alpha_{n+1}.
\]
The Fefferman-Stein inequality from Corollary \ref{cor:fefferman-stein-general} implies, under the assumption that $w^q \in RH_{\frac{s_{n+1}}{q}}$, that 
\[
\big\|  \| T_k( \vec f_1, \ldots, \vec f_n)  \|_{L^{R'_{n+1}}} \big\|_{L^q(w^q)} \lesssim  \big\|  \vec {\ic M}_{s_1, \ldots, s_n} ( \| \vec f_1 (x,\cdot)  \|_{L^{R_1}}, \ldots, \| \vec f_n (x,\cdot)  \|_{L^{R_n}} )  \big\|_{L^q(w^q)},
\]
while Proposition \ref{prop-strong-est-weighted-multi-max} implies that 
\[
\big \| \vec {\ic M}_{s_1, \ldots, s_n}(f_1, \ldots, f_n) \big\|_{L^{q}(\nu_{\vec w})} \leq C \prod_{j=1}^n \| f_j \|_{L^{q_j}(w_j^{q_j})},
\]  
if the $A_{\vec q, \vec s}$ condition is satisfied for the vector weight $\vec w=(w_1^{q_1}, \ldots, w_n^{q_n})$.

Hence we need to check that \eqref{eq:joint-weight-condition} implies that
\begin{equation}
\label{eq:two-cond-implied-by-joint}
w^q \in RH_{\frac{s_{n+1}}{q}} \quad \text{and} \quad \vec w =(w_1^{q_1}, \ldots, w_n^{q_n}) \in A_{\vec q, \vec s}.
\end{equation}

We note that in this case $\ds \nu_{\vec w}=(w_1 \cdot \ldots \cdot w_n)^q=w^q$. Our main tool is Lemma \ref{lemma:RH-vector-weights}: we apply it with $\beta=s_{n+1}$ and $\beta_j=\frac{q_j \, s_j}{q_j - s_j}$ to the joint vector weight condition \eqref{eq:joint-weight-condition}
\begin{equation*}
\sup_{Q} \big( \aver{Q} w^{s_{n+1}} \big)^\frac{1}{s_{n+1}} \, \prod_{j=1}^n \Big(  \aver{Q} w_j^{- \frac{1}{\frac{1}{s_j}-\frac{1}{q_j}}} \Big)^{\frac{1}{s_j}-\frac{1}{q_j}} <\infty.
\end{equation*}

We obtain that 
\begin{equation}
\label{eq:implication-vector-weight-RH}
(w^q)^\frac{s_{n+1}}{q} \in A_{1+\frac{s_{n+1}}{q} (  \sum\limits_{i=1}^n \frac{q}{\beta_i} ) }, \quad w_j^{-\beta_j} \in A_{t_j}=A_{1+\beta_j ( \frac{1}{s_{n+1}}+\sum\limits_{i \neq j} \frac{1}{\beta_i} )}.
\end{equation}

But the condition $\ds (w^q)^\frac{s_{n+1}}{q} \in A_{1+\frac{s_{n+1}}{q} (  \sum\limits_{i=1}^n \frac{q}{\beta_i} ) }$ is equivalent to $w^q \in RH_{\frac{s_{n+1}}{q}} \cap A_{1+q (\sum\limits_{i=1}^n \frac{1}{\beta_i})}$ (here we use that $v \in A_p \cap RH_s \Leftrightarrow v^s \in A_{s(p-1)+1}$). Hence we obtain 
\begin{equation}
\label{eq:RH-cond-for-w^q}
w^q \in RH_{\frac{s_{n+1}}{q}} \quad \text{and} \quad w^q \in A_{1+q (\sum\limits_{i=1}^n \frac{1}{\beta_i})}.
\end{equation}

We are left with checking that $\vec w \in A_{\vec q, \vec s}$, and for that we use again Lemma \ref{lemma:RH-vector-weights}, but in the reverse direction. In this case $\tilde \beta=q$, $\tilde \beta_j=\beta_j=\frac{q_j \, s_j}{q_j - s_j}$, and
\[
\tilde t= 1+ \tilde \beta (\sum\limits_{i=1}^n \frac{1}{\tilde \beta_i} )=1+ q (\sum\limits_{i=1}^n \frac{1}{ \beta_i} ), \qquad \tilde t_j= 1+ \tilde \beta_j (\frac{1}{\tilde \beta}+ \sum_{i \neq j} \frac{1}{\tilde \beta_i})= 1+ \beta_j (\frac{1}{q}+ \sum_{i \neq j} \frac{1}{\beta_i}).
\]

Lemma \ref{lemma:RH-vector-weights} implies that $\vec w \in A_{\vec q, \vec s}$ provided 
\[
\nu_{\vec w}=w^q \in A_{\tilde t}=A_{1+ q (\sum\limits_{i=1}^n \frac{1}{ \beta_i} )}, \qquad w^{-\beta_j} \in A_{\tilde t_j}.
\]
The first condition was deduced in \eqref{eq:RH-cond-for-w^q}, while the second one follows from \eqref{eq:implication-vector-weight-RH} since $\ds \tilde t_j >t_j$ (due to the fact that $\frac{1}{s_{n+1}}< \frac{1}{q}-\alpha_{n+1} <\frac{1}{q}$) and $\ds A_{t_j} \subset A_{\tilde t_j}$. Therefore the multi(sub-)linear operator does satisfy the weighted estimates we wanted.

\end{proof}

We remark that weighted estimates for vector weights satisfying \eqref{eq:joint-weight-condition} can also be deduced directly from the sparse domination.

\subsection{The case $k=0$} \label{sec:k=0} Of particular importance is the case $k=0$. In this situation, the operator $T$ of Theorem \ref{thm:main-thm} is a multilinear Fourier multiplier, as introduced by Coifman and Meyer \cite{CoifMeyer-ondelettes}. Weighted estimates for multilinear Calder\'on-Zygmund operators were achieved in \cite{Nazarov-Lerner-DyadicCalculus}, by proving a pointwise sparse domination, which allowed the authors to treat directly the quasi-Banach case as well.

We note that a \emph{rank-$0$} collection of tiles is one which has \emph{$0$ degrees of freedom} in frequency; that means, the tiles are completely characterized by their spatial interval. For that reason, we can index the tiles after a collection $\ii I$ of dyadic intervals.

In the scalar case, that is, when $m=0, q=1$, we recover the sparse estimates of \cite{Nazarov-Lerner-DyadicCalculus} by examining the $(n+1)$-linear form. In the Banach case, the localization result can be obtained directly for general functions, with $L^1$ sizes (here we have $\alpha_j=0$):
\begin{equation}
\label{eq:local-sizes-m-k=0} \tag{loc $m$}
\vert \Lambda_{\rr P \left(I_0 \right)} ( f_1, \ldots, f_{n+1})   \vert \lesssim \prod_{j=1}^{n+1} \big( \sssize_{I_0} |f_j | \big) \cdot \vert I_0 \vert.
\end{equation}

To see this, it is enough to observe in the proof of Theorem \ref{thm:localization-rank-k} that Lemma \ref{lemma:dec-lemma} can be replaced by the following decomposition algorithm:
\begin{lemma}
\label{lemma:dec-lemma-k=0}
If $\rr P$ is a rank-$0$ collection of tiles with $\ssize_{\rr P}\big(\langle f, \phi_P \rangle \big) \leq \lambda$, then there exists a decomposition 
$\rr P =\rr P' \cup \rr P''$ so that $\ssize_{\rr P'}\big(\langle f, \phi_P \rangle \big) \leq \frac{\lambda}{2}$ and $\rr P''$ is a union $\ds \rr T=\bigcup_{T \in \rr T} T$ of disjoint trees so that 
\[
\sum_{T \in \rr T} \vert I_T \vert \lesssim \lambda^{-1} \|  f \|_1.
\]
\begin{proof}
The proof is basically contained in Lemma 2.10 and Lemma 2.14 of \cite{multilinear_harmonic}.
\end{proof}
\end{lemma}

For the quasi-Banach case, when $\tau<1$, the localization estimate is 
\begin{equation}
\label{eq:quasi-Banach-local}
\| T(f_1, \ldots, f_n) \cdot v\|_\tau \lesssim \prod_{j=1}^n \big( \sssize_{I_0} f_j  \big) \cdot \big( \sssize_{I_0}^\tau v  \big) \cdot |I_0|^\frac{1}{\tau}.
\end{equation}
If $v$ is a characteristic function, the result has already appeared in \cite{quasiBanachHelicoid}, and the general case follows as in Proposition \ref{prop:local-C-var} of the later Section \ref{sec:$L^1$ sizes for the multiple vector-valued sparse domination}, by writing $v=v_1 \cdot v_2$, where $v_1 \in L^1_{\loc}$ and $v_2 \in L^{\tau_0}_{loc}$, with $\ds 1<\frac{1}{\tau}=1+\frac{1}{\tau_0}$.

This further implies the sparse estimate (which depends on the locally integrable functions $f_1, \ldots, f_n, v^\tau$):
{\fontsize{10}{10}\begin{equation*}
\big\|   T( f_1, \ldots,f_n) \cdot v \big\|_\tau^\tau \lesssim \sum_{Q \in \ic S} \prod_{j=1}^{n} \big( \frac{1}{\vert Q \vert} \int_{\rr R} \big|  f_j\big| \cdot \ci_{Q}^{M} dx  \big)^{\tau} \big( \frac{1}{\vert Q \vert} \int_{\rr R} \big| v(x) \big|^\tau \cdot \ci_{Q}^{M} dx  \big)  \cdot  \vert Q \vert,
\end{equation*}}
which is very similar to the quasi-Banach case presentation in \cite{Nazarov-Lerner-DyadicCalculus}.

Reasoning as in Proposition \ref{prop:no-subadd}, we can prove, for $q>1$ and any $\epsilon>0$, the sparse domination
{\fontsize{10}{10}\begin{equation*}
\big\|   T( f_1, \ldots,f_n) \cdot v \big\|_q^q \lesssim \sum_{Q \in \ic S} \prod_{j=1}^{n} \big( \frac{1}{\vert Q \vert} \int_{\rr R} \big|  f_j\big| \cdot \ci_{Q}^{M} dx  \big)^{q} \big( \frac{1}{\vert Q \vert} \int_{\rr R} \big| v(x) \big|^{q+\epsilon} \cdot \ci_{Q}^{M} dx  \big)^\frac{q}{q+\epsilon}  \cdot  \vert Q \vert,
\end{equation*}}where the collection $\ic S$ of dyadic intervals depends on the functions $f_1, \ldots, f_n, v$ and on the Lebesgue exponent $q$.

In the general depth-$m$ case, we are short of obtaining a similar result. An argument similar to that presented in the proof of Theorem \ref{thm:main-thm} yields a sparse domination of the multilinear form involving $L^{1+\epsilon}$ averages for all functions involved. Using the more careful argument of Section \ref{sec:$L^1$ sizes for the multiple vector-valued sparse domination}, we can in fact allow ``most of the averages" to be in $L^1$. This will be made precise later.

Nevertheless, if all the Lebesgue exponents $r_j^l$ are strictly between $1$ and $\infty$ (so we are in the case of reflexive Banach spaces), we can obtain directly a localization result with $L^1$ vector-valued averages which entails, in the case $q=1$, a sparse domination result with $L^1$ averages.
\begin{proposition}
\label{prop:paraprod-UMD}
If $X_j=L^{R_j}(\ii W, \mu)$ are multiple vector-valued spaces with $1< r_j^l <\infty$ for all $1 \leq l \leq m, 1 \leq j \leq n+1$, we have, for the $(n+1)$-linear form of a vector-valued multilinear Fourier multiplier: 
\begin{equation}
\label{eq:local-sizes-m-k=0-vv} 
\vert \Lambda_{\rr P \left(I_0 \right)} ( \vec f_1, \ldots, \vec f_{n+1})   \vert \lesssim \prod_{j=1}^{n+1} \big( \sssize_{I_0} \|\vec f_j(x, \cdot) \|_{X_j} \big) \cdot \vert I_0 \vert.
\end{equation}
\begin{proof}
We present the proof in the particular case when $n=2$, $m=1$, but the generalization is routine. We note that we can write the trilinear form as
\[
\Lambda_{\rr P \left(I_0 \right)}(\vec f, \vec g, \vec h)=\sum_k \sum\limits_{\substack{I \in \ii I \\ I \subseteq I_0}} \frac{1}{|I|^{\frac{1}{2}}} \langle f_k, \varphi_I  \rangle  \langle g_k, \psi_I  \rangle  \langle h_k, \psi_I  \rangle,
\] 
where $\lbrace \psi_I \rbrace_I$ are lacunary and $\lbrace \varphi_I\rbrace_I$ are ``overlapping" families of $L^2$-adapted wave packets associated to the collection $\ii I$. In the general case of a $(n+1)$-linear form, we have at least two lacunary families. The lacunarity condition implies that 
\[
\ic S_{I_0}(g)(x):= \big(\sum\limits_{\substack{I \in \ii I \\ I \subseteq I_0}} \frac{| \langle g, \psi_I  \rangle  |^2}{|I|} \cdot \one_I(x) \big)^\frac{1}{2}
\]
defines a discretized square function  as in Chapter 2.3 of \cite{multilinear_harmonic}. Similarly, 
\[
f \mapsto \sup\limits_{\substack{I \in \ii I \\ I \subseteq I_0}} \frac{| \langle f, \varphi_I  \rangle |}{|I|^\frac{1}{2}}
\]
plays the role of a maximal operator.

If the vector spaces involved are $\ell^{r_1}, \ell^{r_2}, \ell^{r_3}$, with $\ds \frac{1}{r_1}+\frac{1}{r_2}+\frac{1}{r_3}=1$, we have the estimate
\begin{align*}
&\Big |\sum_k \sum\limits_{\substack{I \in \ii I \\ I \subseteq I_0}} \frac{1}{|I|^{\frac{1}{2}}} \langle f_k, \varphi_I  \rangle  \langle g_k, \psi_I  \rangle  \langle h_k, \psi_I  \rangle  \Big |\\
& =\Big| \int_{\rr R} \sum_k \sum\limits_{\substack{I \in \ii I \\ I \subseteq I_0}} \frac{ \langle f_k, \varphi_I  \rangle}{|I|^{\frac{1}{2}}}\cdot \one_I(x) \frac{ \langle g_k, \psi_I  \rangle}{|I|^{\frac{1}{2}}} \cdot \one_I(x) \frac{ \langle h_k, \psi_I  \rangle}{|I|^{\frac{1}{2}}} \cdot \one_{I}(x) dx \Big|\\ 
 &\lesssim  \int_{\rr R} \Big( \sum_{k} \Big| \sup\limits_{\substack{I \in \ii I \\ I \subseteq I_0}} \frac{ |\langle f_k, \varphi_I  \rangle|}{|I|^{\frac{1}{2}}}\cdot \one_{I_0}(x) \Big|^{r_1} \Big)^{\frac{1}{r_1}} \\
 & \cdot \Big( \sum_{k} \big(\sum\limits_{\substack{I \in \ii I \\ I \subseteq I_0}} \frac{| \langle g_k, \psi_I  \rangle  |^2}{|I|} \cdot \one_I(x) \big)^\frac{r_2}{2} \Big)^{\frac{1}{r_2}} \cdot  \Big( \sum_{k} \big(\sum\limits_{\substack{I \in \ii I \\ I \subseteq I_0}} \frac{| \langle h_k, \psi_I  \rangle  |^2}{|I|} \cdot \one_I(x) \big)^\frac{r_3}{2} \Big)^{\frac{1}{r_3}} dx
\end{align*}

The first term, corresponding to the maximal operator, can be bounded pointwise by 
\[
\Big( \sum_{k}  \Big|  \inf_{y \in I_0} \ic M (f_k \cdot \ci_{I_0}^M)(y) \Big|^{r_1} \Big)^{\frac{1}{r_1}} \lesssim \frac{1}{|I_0|} \Big\|  \Big( \sum_{k} \Big|  \ic M (f_k \cdot \ci_{I_0}^M) \Big|^{r_1} \Big)^{\frac{1}{r_1}} \Big\|_{1, \infty}.
\]

Invoking the weak-type $(1, 1)$ Fefferman-Stein inequality, we can estimate this by 
\[
\frac{1}{|I_0|} \big\|  \big( \sum_{k} |  f_k  |^{r_1} \big)^{\frac{1}{r_1}} \cdot \ci_{I_0}^M \big\|_{1} \lesssim \sssize_{I_0} \| \vec f \|_{\ell^{r_1}}.
\]

Hence
\begin{align*}
\big| \Lambda_{\rr P \left(I_0 \right)}(\vec f, \vec g, \vec h)  \big| &\lesssim \sssize_{I_0} \| \vec f \|_{\ell^{r_1}} \cdot \frac{1}{|I_0|^\frac{1}{2}} \Big\|  \Big( \sum_{k} \big(\sum\limits_{\substack{I \in \ii I \\ I \subseteq I_0}} \frac{| \langle g_k, \psi_I  \rangle  |^2}{|I|} \cdot \one_I(x) \big)^\frac{r_2}{2} \Big)^{\frac{1}{r_2}}  \Big\|_2 \\
&\cdot  \frac{1}{|I_0|^\frac{1}{2}} \Big\|  \Big( \sum_{k} \big(\sum\limits_{\substack{I \in \ii I \\ I \subseteq I_0}} \frac{| \langle h_k, \psi_I  \rangle  |^2}{|I|} \cdot \one_I(x) \big)^\frac{r_3}{2} \Big)^{\frac{1}{r_3}}  \Big\|_2  \cdot |I_0|.
\end{align*}

So we have, for each of the functions $\vec g$ and $\vec h$, an $L^2$ average of the vector-valued square function. Using a vector-valued John-Nirenberg  inequality (a straightforward generalization of Theorem 2.7 of \cite{multilinear_harmonic}), we can obtain a ``weak $L^{1}$ average":
\[
\sup_{I' \subseteq I_0} \frac{1}{|I'|^\frac{1}{2}} \Big\|  \Big( \sum_{k} \big(\sum\limits_{\substack{I \in \ii I \\ I \subseteq I'}} \frac{| \langle g_k, \psi_I  \rangle  |^2}{|I|} \cdot \one_I(x) \big)^\frac{r_2}{2} \Big)^{\frac{1}{r_2}}  \Big\|_2 \sim  \sup_{I' \subseteq I_0 }\frac{1}{|I'|} \Big\|  \Big( \sum_{k} \big(\sum\limits_{\substack{I \in \ii I \\ I \subseteq I'}} \frac{| \langle g_k, \psi_I  \rangle  |^2}{|I|} \cdot \one_I(x) \big)^\frac{r_2}{2} \Big)^{\frac{1}{r_2}}  \Big\|_{1, \infty}.
\]

If $1<r_2<\infty$, the $\ell^{r_2}$-valued square function is a bounded operator from $L^1$ into $L^{1, \infty}$, so in the end we get \eqref{eq:local-sizes-m-k=0-vv}.
\end{proof}
\end{proposition}

From the proof above, it is clear that, as long as they are not associated to the lacunary directions (i.e. to one of the two square functions), we can allow the Lebesgue exponents $r_j^l$ to be equal to $\infty$. But in the general depth-$m$ case, when we let $1 <r_j^l \leq \infty$, we have, as announced in Theorem \ref{thm:k=0case}, the following: for any $\epsilon >0$ and any vector-valued functions $\vec f_1, \ldots, \vec f_{n+1}$ so that $\| \vec f_j(x, \cdot) \|_{X_j}$ are locally integrable, there exists a sparse collection $\ic S$ depending on the functions $\vec f_j$ for which
\begin{equation*}
\big| \Lambda_{\rr P}(\vec f_1, \ldots, \vec f_{n+1}) \big|\lesssim \sum_{Q \in \ic S} \prod_{j=1}^{n+1}\big( \frac{1}{|Q|} \int_{\rr R} \| \vec f_j(x, \cdot) \|_{X_j}^{1+\delta_j} \cdot \ci_{Q}^M dx \big)^\frac{1}{1+\delta_j}  \cdot |Q|,
\end{equation*}
where $\delta_j \in \lbrace 0, \epsilon \rbrace$, and we can arrange that at most $m$ of the $\delta_j \neq 0$. However, if $\delta_j\neq 0$, then $j$ is an index so that in the multiple vector space $X_j=L^{R_j}(\ii W, \mu)$ there are no $L^\infty$ involved: $\ds r_j^l <\infty$ for all $1 \leq l \leq m$.

As mentioned before, this follows from an analysis similar to that of Section \ref{sec:$L^1$ sizes for the multiple vector-valued sparse domination}. In the scalar case, when $m=0$, we can obtain directly a sparse domination with only $L^1$ averages. For $m=1$, the result with at most one $L^{1+\epsilon}$ average follows from the estimate
\begin{align*}
\big| \Lambda_{\rr P \left(I_0 \right)}(f, f^1_1\cdot f^2_1, \ldots, f^1_n \cdot f^2_n) \big|\lesssim \sssize_{I_0}^{r+\epsilon}(\one_F) \cdot |F|^\frac{1}{r'} \cdot \prod_{k=1}^n \big( \sssize_{I_0}^{r_k'} f_k^1 \big) \cdot \|f_k^2\|_{r_k}, 
\end{align*}
where $|f(x)|\leq \one_F,$ $\ds \frac{1}{r_1}+\ldots  +\frac{1}{r_2}=\frac{1}{r}$, and $1<r <\infty$, $1< r_k <\infty$. Then we deduce the vector-valued local result, as in Proposition \ref{prop:local-ell^s-aver-L^1}, leaving intact the sizes corresponding to $\ell^\infty$, if there are any. In would be interesting to understand whether a multiple vector-valued sparse domination with $L^1$ averages can be obtained in the general case, when $L^\infty$ spaces are allowed. However, it was noticed in \cite{Nazarov-Lerner-DyadicCalculus} that weighted estimates involving $L^\infty$ spaces can still be deduced from a sparse domination with exponents strictly between $1$ and $\infty$, by a careful passage to the limit.

In the quasi-Banach case, if $n=2$, we can prove using \eqref{eq:quasi-Banach-local} that the operator 
\[
f \mapsto T_{I_0}(f \cdot \one_F, g_1\cdot g_2)
\]
maps $L^{r_1,1}$ into $L^r(w)$, with an operatorial norm equal to
\[
\big( \sssize_{I_0} \one_F  \big)^{\frac{1}{r_1'}-\epsilon} \cdot \big( \sssize_{I_0}^{r_2'} g_1 \big) \cdot \big( \sssize_{I_0} w \big)^{\frac{1}{r}} \|g_2 \cdot \ci_{I_0}^M\|_{r_2},
\]
where $w \geq 0$ is a fixed locally integrable function and $1<r_1, r_2 \leq \infty$, $\ds \frac{1}{r_1}+\frac{1}{r_2}=\frac{1}{r}>1$. Since $r<1$, we cannot have $r_1 =\infty$ or $r_2=\infty$. For $n \geq 3$, $L^\infty$ spaces could appear, and in that situation the analysis is similar to the Banach case described earlier.

Using interpolation (for example,  Theorem 1.4.19 of \cite{grafakos-book}), with $r_2$ fixed, we get that 
\[
\|T_{I_0}(f \cdot \one_F, g_1 \cdot g_2) \cdot w^{\frac{1}{r}}\|_r^r \lesssim \big[ \big( \sssize_{I_0} \one_F  \big)^{\frac{1}{r_1'}-\epsilon} \cdot \big( \sssize_{I_0}^{r_2'} g_1 \big) \cdot \big( \sssize_{I_0} w \big)^{\frac{1}{r}} \big]^r \|f \cdot \ci_{I_0}^M\|_{r_1}^r \cdot \|g_2 \cdot \ci_{I_0}^M\|_{r_2}^r.
\]

As in Proposition \ref{prop:local-ell^s-aver-L^1}, this can be used for proving, for vector-valued functions $\vec f =\lbrace  f_k \rbrace, \vec g =\lbrace  g_k \rbrace$ so that $\|\vec f(x)\|_{\ell^{r_1}}, \|\vec g(x)\|_{\ell^{r_2}}$ are locally integrable, that
\[
\big\|  \big(  \sum_k  |T_{I_0}(f_k, g_k)|^r \big)^{\frac{1}{r}} \cdot v  \big\|_{\tau}^{\tau} \lesssim \big( \sssize_{I_0}^{1+\epsilon} \| \vec f\|_{\ell^{r_1}} \big)^\tau \cdot \big( \sssize_{I_0} \|\vec g\|_{\ell^{r_2}} \big)^\tau \cdot \big( \sssize_{I_0} |v|^\tau \big) \cdot |I_0|.
\]

If $\ds \|  \cdot \|_{L^{R'_{n+1}}}^\tau$ is subadditive (a sufficient condition is that $\tau<r^{j_0}$), this generalizes to $n$-linear Calder\'on-Zygmund operators, yielding eventually the sparse domination
{\fontsize{9}{10}\begin{equation}
\label{eq:sparse-qB-k=0}
\big\|  \big\|T(\vec f_1, \ldots, \vec f_n)\big\|_{L^{R'_{n+1}}} \cdot v \big\|_\tau^\tau \lesssim \sum_{Q \in \ic S} \prod_{j=1}^{n} \big( \frac{1}{\vert Q \vert} \int_{\rr R} \big\|  \vec f_j(x, \cdot) \big\|_{\tilde X_j}^{(1+\delta_j)} \cdot \ci_{Q}^{M} dx  \big)^\frac{\tau}{1+\delta_j} \big( \frac{1}{\vert Q \vert} \int_{\rr R} \big| v(x) \big|^\tau \cdot \ci_{Q}^{M} dx  \big)  \cdot  \vert Q \vert,
\end{equation}}where $\delta_j\in \lbrace 0, \epsilon  \rbrace$, with at most $m$ of the $\delta_j \neq 0$ (in this respect, the situation is similar to the Banach case: the $L^{1+\epsilon}$ averages cannot correspond to multiple vector-valued spaces involving $L^\infty$).

In the usual manner, the sparse domination result is to be read as: there exists a sparse family $\ic S$ of dyadic intervals, which depends on the functions $\vec f_1, \ldots, \vec f_{n}$ and $v$, so that \eqref{eq:sparse-qB-k=0} holds.  
\begin{remark}
A careful inspection of the proof of Proposition \ref{prop:local-est-var-C-L^q-size} reveals that it is essential to have at least one restricted-type function in the estimation of the multilinear form. Also, in order to sum up the averages, we need to lose a bit of information by performing a stopping time in $\one_F$ with respect to $L^{r_1-\epsilon}$.
\end{remark}

For $q>r^{j_0}$, we have similarly, the sparse domination
{\fontsize{9}{10}\begin{equation*}
\big\|  \big\|T(\vec f_1, \ldots, \vec f_n)\big\|_{L^{R'_{n+1}}} \cdot v \big\|_q^q \lesssim \sum_{Q \in \ic S} \prod_{j=1}^{n} \big( \frac{1}{\vert Q \vert} \int_{\rr R} \big\|  \vec f_j(x, \cdot) \big\|_{\tilde X_j}^{(1+\delta_j)} \cdot \ci_{Q}^{M} dx  \big)^\frac{q}{1+\delta_j} \big( \frac{1}{\vert Q \vert} \int_{\rr R} \big| v(x) \big|^{q+\epsilon_q} \cdot \ci_{Q}^{M} dx  \big)^\frac{q}{q+\epsilon_q}  \cdot  \vert Q \vert.
\end{equation*}}

Lastly, we present a proof of Corollary \ref{cor:Fefferman-Stein-k=0case}, the Fefferman-Stein inequality for $T_0$, which follows immediately from the sparse domination above. Let $0<p< \infty$, and $w \in A_{\infty}$. The $A_\infty$ assumption appears twice in the argument: first for invoking a \emph{reverse H\"older} property, and second for deducing that $\ds E_Q \subseteq Q, |E_Q | \geq \eta |Q| \Rightarrow w(E_Q) \geq \tilde \eta \,w(Q)$.

Let $(R_1, \ldots, R_n, R'_{n+1})$ be $m$-tuples so that $1<r_j^l \leq \infty$ for all $1 \leq l \leq m, 1 \leq j \leq n$, $\frac{1}{2} <  (r^l_{n+1})' <\infty$. Then if $p<r^{j_0}$, we have the sparse domination, which depends on the functions and the Lebesgue exponents:
\[
\int_{\rr R} \big\|T(\vec f_1, \ldots, \vec f_n)\big\|_{L^{R'_{n+1}}}^p w(x) dx \lesssim \sum_{Q \in \ic S} \prod_{j=1}^{n} \big( \frac{1}{\vert Q \vert} \int_{Q} \big\|  \vec f_j(x, \cdot) \big\|_{\tilde X_j}^{(1+\delta_j)}  dx  \big)^\frac{p}{1+\delta_j} \big( \frac{1}{\vert Q \vert} \int_{Q} w(x) dx  \big)  \cdot  \vert Q \vert.
\]

The sparseness property of the intervals implies the existence of mutually disjoints subsets $E_Q \subseteq Q$ with $ |E_Q | \geq \eta |Q| $ and hence 
$w(E_Q) \geq \tilde \eta \, w(Q)$, which yields
{\fontsize{10}{10}\begin{align*}
\int_{\rr R} \big\|T(\vec f_1, \ldots, \vec f_n)\big\|_{L^{R'_{n+1}}}^p w(x) dx &\lesssim  \sum_{Q \in \ic S} \inf_{x \in E_Q} \vec {\ic M}_{1+\delta_1, \ldots, 1+\delta_n}(\big\|  \vec f_1(y, \cdot) \big\|_{L^{R_1}}, \ldots, \big\|  \vec f_n(y, \cdot) \big\|_{L^{R_n}}))^p(y) \cdot w(E_Q) \\
&\lesssim \int_{\rr R} \big| \vec {\ic M}_{1+\delta_1, \ldots, 1+\delta_n}(\big\|  \vec f_1(x, \cdot) \big\|_{L^{R_1}}, \ldots, \big\|  \vec f_n(x, \cdot) \big\|_{L^{R_n}}))(x)\big|^p w(x) dx.
\end{align*}}

If $p>r^{j_0}$, the $L^1$ average of $w$ is replaced by an $L^{1+\epsilon_p}$ average, where $\epsilon_p$ can be as small as we wish. Since $w \in A_\infty$, it satisfies a reverse H\"older inequality: there exists $\epsilon >0$ so that
\[
(\aver{Q} w^{1+\epsilon})^\frac{1}{1+\epsilon} \lesssim \aver{Q} w, \quad \text{for all intervals } Q.
\]
Hence we obtain again the same estimate as above.

The Fefferman-Stein inequalities for operators $T_k$ of Theorem \ref{thm:main-thm} or for Carleson and variational Carleson operators can be proved similarly.

\section{Another study case: Carleson and variational Carleson operators}
\label{sec:Carleson-op}
\subsection{The Carleson operator}
\label{sec:carleson_op-only}
Before studying the variational Carleson operator $\ic C^{var,r}$, we briefly describe the procedure for obtaining vector-valued or sparse estimates for the more classical Carleson operator, which is defined by
\begin{equation}
\label{def:Carleson}
\ic C f(x):=\sup_{N} \big\vert \int_{\xi < N} \hat{f}(\xi) e^{2 \pi i x \xi} d \xi  \big\vert.
\end{equation}

Its boundedness on $L^p$, together with a transference principle, implies a.e. convergence of Fourier series in $L^p(\rr T)$.

The operator $\ic C$ can be linearized by introducing the function $N(x)$ which attains the sup in \eqref{def:Carleson}, of which nothing is known. Furthermore, the condition $\xi < N(x)$ is replaced by
\[
\text{there exists a dyadic interval $\omega$ so that $\xi \in \omega_{\text{left}}$ and $N(x) \in \omega_{\text{right}}$.}
\]

As a consequence, the study of $ \ic C$ is replaced by that of its model operator
\[
\ic C_{\rr P}(x):=\sum_{P \in \rr P} \langle f, \phi_P \rangle \phi_P(x) \one_{\left\lbrace x: N(x) \in \omega_{P_2} \right\rbrace}:=\sum_{P \in \rr P} \langle f, \phi_P \rangle \tilde \phi_P(x),
\]
where $\rr P$ is a collection of bitiles $P=(I_P \times \omega_{P_1}, I_P \times \omega_{P_2})$ (here $\omega_{P_1}$ and $\omega_{P_2}$ are the left half and right half of the dyadic interval $\omega_P$).

We need to introduce a new type of size, which captures the behavior of the functions $\tilde \phi_P(x)$. Here we follow the presentation in \cite{multilinear_harmonic}.

\begin{definition}
If $P$ and $P'$ are distinct bitiles, we say that $P \leq P'$ if $I_P \subseteq I_{P'}$ and $\omega_{P'} \subseteq \omega_P$.

Also, we denote by $\bar {\rr P}$ the collection of all possible dyadic bitiles in the plane. If the collection $\rr P(I_0)$ is localized in space onto a certain dyadic interval $I_0$, then $\bar {\rr P}(I_0)$ denotes the collection of all possible dyadic bitiles $P$ in the plane with $I_P \subseteq 3 I_0$.
\end{definition}

\begin{definition}
If $\rr P$ is a finite collection of bitiles, then
\[
\ssize_{\rr P} (\langle g, \tilde \phi_P  \rangle):=\sup_{P \in \rr P} \sup_{\substack{P' \in \bar{\rr P} \\ P' \geq P}} \frac{1}{\vert  I_P' \vert} \int_{\rr R} \vert g(x) \vert \ci_{I_{P'}}^M(x) \cdot \one_{\left \lbrace x: N(x) \in \omega_{P'}  \right\rbrace} dx.
\]
\end{definition}

Here we use the classical proof of the boundedness of $\ic C$ as a black box, but it will be enough to consider a simpler $\ssize$, which is larger than the one we introduced above:
\[
\ssize_{\rr P} (\langle g, \tilde \phi_P  \rangle) \lesssim \sssize_{\rr P}^* (g):=\sup_{P \in \rr P} \sup_{\substack{P' \in \bar{\rr P} \\ P' \geq P}} \frac{1}{\vert  I_P' \vert} \int_{\rr R} \vert g(x) \vert \cdot \ci_{I_{P'}}^M(x) dx.
\]

The proof of \cite{multilinear_harmonic} relies on a quantity dual to the size, called \emph{energy}. We can avoid to overtly use the `energy', but we refer the interested reader to Chapter 7 of \cite{multilinear_harmonic}. Instead, we use directly the estimate
\begin{proposition}[Proposition 7.7 from \cite{multilinear_harmonic}]
If $\rr P$ is a finite collection of bitiles, and $f$ and $g$ are measurable functions, then 
\[
\big\| \int_{\rr R}\ic C_{\rr P}(f)(x) g(x) dx \big\| \lesssim \big( \sssize_{\rr P} f   \big) ^{\theta_1} \cdot \big( \sssize_{\rr P}^*(g)  \big)^{\theta_2} \cdot \|  f \|^{1-\theta_1}_2 \cdot \| g \|^{1-\theta_2}_1
\]
for any $0 \leq \theta_1<1, 0< \theta_2 \leq \frac{1}{2}$ with $\theta_1 +2 \theta_2=1$.
\end{proposition}

In order to obtain vector-valued or sparse estimates for $\ic C_{\rr P}$, we need to work with the localized bilinear form. Here we use the fact that the energy corresponding to $g$ is an $L^1$ quantity, so for $g$ we don't need to use restricted-type functions.

\begin{lemma}
\label{lemma:local-Carleson}
If $\rr P$ is a finite collection of bitiles, $I_0$ is a fixed dyadic interval, $f$ and $g$ are measurable functions so that $\vert f(x)\vert \leq \one_{F}(x)$, then we have
\[
\vert \Lambda_{\ic C; \rr P (I_0)}(f,g)  \vert \lesssim \big( \sssize_{I_0} \one_F \big)^{\frac{1+\theta_1}{2}} \cdot \big( \sssize_{I_0}^* g  \big)^1 \cdot \vert  I_0 \vert,
\]
for any $0 \leq \theta_1 <1$.
\end{lemma}

Using the techniques presented in the previous sections, we get the following sparse domination for $\ic C_{\rr P}$:
\begin{proposition}
If $\rr P$ is a finite collection of bitiles, and $f$ and $g$ are locally integrable functions, then for any $\epsilon>0$ there exists a sparse collection $\ic S$ of dyadic intervals so that 
\[
\vert \Lambda_{\ic C _{\rr P}}(f, g)  \vert \lesssim \sum_{Q \in \ic S} \big( \frac{1}{\vert Q \vert} \int_{\rr R} \vert f(x) \vert^{1+\epsilon} \cdot \ci_{Q}^M(x) dx \big)^{\frac{1}{1+\epsilon}} \cdot \big( \frac{1}{\vert Q \vert} \int_{\rr R} \vert g(x) \vert \cdot \ci_{Q}^M(x) dx \big) \cdot \vert Q \vert.
\]
\end{proposition}

By applying Proposition 6.4. of \cite{BernicotFreyPetermichl}, we recover the known weighted estimates for the Carleson operator: for any $\epsilon >0$ and any $1<p <\infty$,
$$\ic C :L^p(w) \to L^p(w)  \quad \text{for all     } w \in A_{p/{1+\epsilon}},$$
with an operatorial norm 
\[
\|  \ic  C \|_{ L^p(w) \to L^p(w)} \lesssim [w]_{A_\frac{p}{1+\epsilon}}^{\max \big( \frac{1}{p-1-\epsilon}, 1 \big)}.
\]
We recall that weighted estimates for the Carleson operator were proved in \cite{weighted-Carleson-DiPlinioLerner}. Sparse vector-valued estimates were also proved in \cite{DavidBeltran-FeffermanSteinCarleson}.

The multiple vector-valued and sparse vector-valued estimates that we can get for $\ic C$ are the following:
\begin{theorem}
\label{thm:multi-Carleson}
If $R=(r^1, \ldots, r^m)$ is an $m$-tuple with $1<r^j<\infty$, then we have
\[
\ic C: L^p(\rr R; L^R(\ii W, \mu)) \to L^p(\rr R; L^R(\ii W, \mu))
\]
for all $1<p<\infty$. Moreover, for any $q>0$, $\epsilon, \epsilon_q>0$, any $\vec f$ multiple vector-valued functions with $\| f(x, \cdot) \|_{L^R(\ii W, \mu)}$locally integrable and any locally $q$-integrable function $v$, there exists a sparse collection $\ic S$ so that
{\fontsize{10}{10}\[
\big\| \| \ic C _{\rr P}(\vec f)\|_{L^R(\ii W, \mu)} \cdot v \big\|_q^q  \lesssim \sum_{Q \in \ic S} \big( \frac{1}{\vert Q \vert} \int_{\rr R} \| f(x, \cdot) \|_{L^R(\ii W, \mu)}^{1+\epsilon} \cdot \ci_{Q}^M(x) dx \big)^{\frac{q}{1+\epsilon}} \cdot \big( \frac{1}{\vert Q \vert} \int_{\rr R} |v(x)|^{q+\epsilon_q} \cdot \ci_{Q}^M(x) dx \big)^{\frac{q}{q+\epsilon_q}} \cdot \vert Q \vert.
\]}
If $q \leq 1$, we can take $\epsilon_q=0$.

\end{theorem}
\begin{remark}
The proof of the vector-valued or sparse estimates follows the same ideas from Sections \ref{sec:local->sparse} and \ref{sec:rank-k}, and we leave the details to the interested reader. 

For the operator of Theorem \ref{thm:main-thm}, conditions \eqref{eq:cond-Leb-exp-vv-T} point to an open range of Lebesgue exponents. Instead, for
Carleson and variational Carleson operators, there exists some sort of endpoint estimate, which remains noticeable also in the sparse domination. Since the adjoint operators $\ic C^*$ and $( \ic C^{var, r})^*$ map $L^1$ into $L^{1, \infty}$, we should be able to obtain, in the sparse domination of the bilinear form (hence, in the case $q=1$), an $L^1$ average! 

These optimal results, corresponding to $q=1$ and $\epsilon_q=0$, are far more involved, but they will be described in detail in the following section, where the variational Carleson operator is examined.
\end{remark}

Also, we have the following upper-bound for the Carleson operator:
\begin{corollary}
For any $0<p<\infty$, $\epsilon>0$, and any $m$-tuple $R=(r^1, \ldots, r^m)$ with $1< r^j<\infty$, we have 
\[
\big\| \big\| \ic C f(x, \cdot) \big\|_{ L^R(\ii W, \mu)}  \big\|_p < \big\|  \ic M_{1+\epsilon} \big( \| f(x, \cdot)\|_{L^R(\ii W, \mu)} \big)\big\|_p.
\]
Under the assumption that $w \in A_\infty$, we have, for any $0<p<\infty$
\[
\big\| \| \ic C _{\rr P}(\vec f)\|_{L^R(\ii W, \mu)}\big\|_{L^p(w)} \lesssim \big\| \ic M_{1+\epsilon}\| (\vec f)\|_{L^R(\ii W, \mu)}\big\|_{L^p(w)}.
\]
The implicit constant in the last inequality might depend on $w$ or on $[w]_{A_\infty}$.
\end{corollary}

\subsection{Variational Carleson operator}~\\
\label{sec:var-Carleson}
Alternatively, one might be interested in finding the rate of convergence of the Fourier series in $L^p(\rr T)$. In that case, the necessary tool is the \emph{variational Carleson operator}, defined by means of the $r$-variation norm as
\begin{equation}
\label{def-varCarleson}
\ic C^{var, r} (f)(x):= \sup_{K} \sup_{n_0< \ldots < n_k} \big( \sum_{\ell=1}^K \vert  \int_{a_{n_{\ell-1}}}^{a_{n_\ell}} \hat{f}(\xi) e^{2 \pi i x \xi}  d \xi  \vert^r \big)^{\frac{1}{r}}.
\end{equation}

We recall the following result from \cite{variational_Carleson} (both conditions $r>2$ and $p>r'$ are necessary):
\begin{theorem}[Oberlin, Seeger, Tao, Thiele, Wright]
\label{thm:OSTTW}
Suppose that $r>2$ and $r'<p < \infty$. Then $\ic C^{var, r} :L^p \to L^p$.
\end{theorem}

Our intention here is to present a proof of Theorem \ref{thm:OSTTW} which is based on the same localization principle that was portrayed in Sections \ref{sec:local->sparse}, \ref{sec:review-hel-mthd} and \ref{sec:rank-k}. Not only will this allow us to obtain a simplified proof (we avoid in this way $BMO$ estimates of $\sum_{T \in \bbf T} \one_T$), but it will immediately imply sparse estimates and, after applying the helicoidal method, also multiple vector-valued and sparse multiple vector-valued estimates. 

More exactly, we will show that the variational Carleson operator has a \emph{local} character as well, and upon obtaining the local inequality (corresponding to the ``level 0" estimate in the helicoidal method), we conclude the multi-vectorial sparse estimates described in Theorem \ref{thm:multi-var-Carleson}.

\subsection{Study of the model operator}
\label{sec:model-op-Carleson-var}
By standard techniques, $\ic C^{var, r}$ can also be linearized: first we fix $K \in \rr Z^+$, and for every $x \in \rr R$, there exist measurable real-valued functions  
$\xi_0(x) < \xi_1(x)< \ldots <\xi_K(x)$ and measurable complex valued functions $a_1(x), \ldots, a_K(x)$ satisfying
\[
\big( \sum_{k=1}^K \vert a_k(x)\vert^{r'}  \big)^\frac{1}{r'}=1
\]
and so that 
\[
\ic C^{var, r}(f)(x)=\sum_{k=1}^K a_k(x) \cdot \int_{\rr R} \one_{\left[\xi_{k-1}\left(x\right), \xi_k\left(x \right) \right]}(x) \hat{f}(\xi) e^{2 \pi i x \xi} d \xi.
\]

Further, assuming that the functions $\xi_k(\cdot)$ take only a finite number of values, the operator can be approximated by a model operator:
\begin{equation}
\label{def-modOpVarCarl}
\ic C^{var, r}_{\rr P} (f)(x):= \sum_{P \in \rr P} \langle f, \phi_P  \rangle \phi_P(x) a_P(x).
\end{equation}
The multi-tiles in $\rr P$ are of the form $P=I_P \times \omega_P$, where each $\omega_P$ is a union of three intervals $\omega_l, \omega_u, \omega_h$. The wave packet $\phi_P$ contains the information in $I_P \times \omega_u$, while in $\omega_l$ (and $\omega_h$) is captured the lowest (or highest) possible frequency information: if $ x \in I_P$, then for some $1 \leq k \leq K$, $\xi_{k-1}(x) \in \omega_l$ (and $\xi_{k}(x) \in \omega_h$). In that case, $a_P(x)=a_k(x)$, and $a_P(x)=0$ if no such a $k$ exists.

In addition, since the frequency information represents a Whitney decomposition of the interval $[a_{k -1(x)}, a_{k(x)}]$, we can assume that there exist constants $1 \leq C_3<C_2<C_1$ so that, for every $P \in \rr P$,
\[
\text{supp\,}(\hat{\phi}_P) \subset C_3 \omega_u, \quad C_2 \omega_u \cap C_2 \omega_l =\emptyset, \quad C_2 \omega_u \cap C_2 \omega_h =\emptyset, \quad C_2 \omega_l \subset C_1 \omega_u, \quad C_2 \omega_u \subset C_1 \omega_l.
\]
On account of the frequency data being contained in three intervals, we denote by $\omega_m$ the convex hull of $C_2 \omega_u \cup C_2 \omega_l$:
\[
\omega_m:= \text{conv} (C_2 \omega_u \cup C_2 \omega_l).
\]

Now we are ready to define the trees, which, as usual, play a very important role in the analysis. As in \cite{variational_Carleson}, we assume there exists $\Xi^{\text{top}}$, a finite set of admissible tree top frequencies.
\begin{definition}
A tree $T=(T, I_T, \xi_T)$ with top datum $(I_T, \xi_T)$, where $I_T$ is a dyadic interval and $\xi_T \in \Xi^{\text{top}}$, is a subcollection of multi-tiles $P \in \rr P$ with the property that
\begin{equation}
\label{eq-cond-tree}
I_P \subseteq I_T \quad \text{and}\quad \omega_{T}:=[\xi_T-\frac{C_2-1}{4 |I_T|}, \xi_T+\frac{C_2-1}{4 |I_T|} )\subseteq \omega_m.
\end{equation}
A tree $(T, I_T, \xi_T)$ is called \emph{$l$-overlapping} if $\xi_T \in C_2 \omega_l$ for every $P \in T$, and \emph{$l$-lacunary} if $\xi_T \notin C_2 \omega_l$ for all $P \in T$.
\end{definition}

In order to obtain the vector-valued estimate in Theorem \ref{thm:multi-var-Carleson}, we need to prove, for $p_0>r'$ 
\begin{equation}
\label{eq:ineq-var-C}
\vert \Lambda_{\ic C^{var, r}; \rr P(I_0)}^{F, G}(f,g)  \vert \lesssim \big\| \Lambda_{\ic C^{var, r}; \rr P(I_0)}\big\| \|f \cdot \ci_{I_0}\|_{p_0} \cdot \|g \cdot \ci_{I_0}\|_{p'_0},
\end{equation}
where $\big\| \Lambda_{\ic C^{var, r}; \rr P(I_0)}\big\|$ is the operatorial norm given by
\begin{equation}
\label{def:op-norm-var-Carleson}
\big\| \Lambda_{\ic C^{var, r}; \rr P(I_0)}\big\|:= \big( \sssize_{\rr P \left( I_0 \right)}  \one_F \big)^{\frac{1}{r'}- \frac{1}{p_0}-\epsilon} \cdot \big( \sssize_{\rr P \left( I_0 \right)}  \one_G \big)^{\frac{1}{p_0}-\epsilon}.
\end{equation}

As usual, this will be showed through restricted type-interpolation. If $E_1, E_2$ are sets of finite measure, it is enough to find $E_2' \subseteq E_2$ a major subset so that, for any restricted-type functions $f$ and $g$ satisfying $|f(x)| \leq \one_{E_1}(x), |g(x)| \leq \one_{E'_2}(x)$, we have
 \[
 \vert \Lambda_{\ic C^{var, r}; \rr P(I_0)}^{F, G}(f,g)  \vert \lesssim \big\| \Lambda_{\ic C^{var, r}; \rr P(I_0)}\big\| \cdot |E_1|^\frac{1}{p} |E_2|^\frac{1}{p'},
 \]
for all $p >r'$ in a neighborhood of $p_0$.
 
Interpolation theory then implies that  
\[ 
\vert \Lambda_{\ic C^{var, r}; \rr P(I_0)}^{F, G}(f,g)  \vert \lesssim \big\| \Lambda_{\ic C^{var, r}; \rr P(I_0)}\big\| \|f \|_{p_0} \cdot \|g \|_{p'_0},
\]
but the estimate \eqref{eq:ineq-var-C} can be easily deduced by using the decay of the sizes (so here we might loose an $\epsilon$ in the exponents of the two sizes).

The proof of Theorem \ref{thm:OSTTW} from \cite{variational_Carleson} is based on a \emph{tree estimate} and two decomposition lemmas: one for the \emph{energy} and one for the \emph{density}. In order to obtain the full range from Theorem \ref{thm:OSTTW}, the authors use a $BMO$ estimates for the tops of the trees (that is, $ \| \sum\limits_{T \in \bbf T} \one_{I_T}  \|_{BMO}$). This is necessary (via intepolation and the classical $ \| \sum\limits_{T \in \bbf T} \one_{I_T}  \|_1$ estimate) for defining the exceptional set, which is somehow complicated. By using localizations and two additional stopping times, we obtain a somewhat simpler demonstration. The ``energy" and ``density" of \cite{variational_Carleson} will be replaced by sizes, an approach which is more similar to the exposition in \cite{multilinear_harmonic}. 

\begin{definition}
\[
\ssize_{\rr P}(f):= \sup_{\substack{T \subset \rr P \\ \text{$l$-overlapping tree}}} \big( \frac{1}{\vert I_T \vert} \sum_{P \in T} \vert \langle f, \phi_P \rangle  \vert^2   \big)^{\frac{1}{2}}
\]

For $g$, the size will be similar to the one appearing in the case of the Carleson operator:
\begin{equation}
\label{def:size-g-C-var}
\ssize_{\rr P} ( g):=  \sup_{P \in \rr P} \sup_{\substack{P' \in \bar{\rr P} \\ P' \geq P}} \big(  \frac{1}{\vert  I_P' \vert} \int_{\rr R} \vert g(x) \vert^{r'} \ci_{I_{P'}}^M(x) \cdot \sum_{k=1}^K \vert a_k(x)  \vert^{r'} \cdot \one_{\omega_{P'}}(\xi_{k-1}(x)) dx \big)^{\frac{1}{r'}}.
\end{equation}
\end{definition}

\begin{remark}
In the definition of $\ssize_{\rr P}(g)$, we consider the supremum over all admissible tiles $P'$ so that $I_P \subseteq I_{P'}$ and $\omega_{P'} \subset \omega_P$, where 
\[
\omega_{P'}:=[\xi_{P'}-\frac{C_2-1}{4 |I_{P'}|}, \xi_{P'}+\frac{C_2-1}{4 |I_{P'}|} )
\]
for some $\xi_{P'} \in \Xi^{\text{top}}$.

\end{remark}

\begin{proposition}[Similar to Prop. 5.1 of \cite{variational_Carleson}] \label{prop:tree-est-Car-var} Let $T \subset \rr P$ be a tree. Then 
\begin{equation}
\label{eq:tree-est-L1}
\vert  \int_{\rr R} \sum_{P \in T} \langle f, \phi_P  \rangle \phi_P(x) a_P(x) g(x) dx \vert \lesssim \ssize_{T}(f) \cdot \ssize_{ T}(g) \cdot \vert I_T \vert
\end{equation}
and also,  
\begin{equation}
\label{eq:tree-est-Lr'}
\|\sum_{P \in T} \langle f, \phi_P  \rangle \phi_P a_P g \|_{L^{r'}} \lesssim \ssize_{T}(f) \cdot \ssize_{ T}(g) \cdot \vert I_T \vert^{\frac{1}{r'}}.
\end{equation}
Furthermore, for $\ell \geq 0$, we have 
\[
\|\sum_{P \in T} \langle f, \phi_P  \rangle \phi_P a_P g \|_{L^1(\rr R \setminus 2^{\ell} I_T)} \lesssim 2^{-\ell (N-10)} \ssize_{T}(f) \cdot \ssize_{ T}(g) \cdot \vert I_T \vert
\]
and 
\[
\|\sum_{P \in T} \langle f, \phi_P  \rangle \phi_P a_P g \|_{{L^{r'}(\rr R \setminus 2^{\ell} I_T)}} \lesssim 2^{-\ell (N-10)} \ssize_{T}(f) \cdot \ssize_{ T}(g) \cdot \vert I_T \vert^{\frac{1}{r'}}.
\]
\begin{proof}The proof follows the same reasoning as that of Proposition 5.1 from \cite{variational_Carleson}, which corresponds to the special case when $g(x)$ is a characteristic function. We have to consider $l$-lacunary trees and $l$-overlapping trees and in the second case, make use of L\'epingle's inequality.

The quantity $\ssize_{\rr P}(g)$ appears naturally in the \emph{tree estimate}. If $T$ is a tree with top $I_T$, we perform a decomposition of $\rr R$ into a collection $\ii J$ made of maximal dyadic intervals $J$ with the property that there is no $P \in T$ so that $I_P \subseteq 3J$.

The tree estimate from \eqref{eq:tree-est-Lr'} follows by estimating first $\|\sum_{P \in \rr P} \langle f, \phi_P  \rangle \phi_P a_P g \|_{L^{r'}(J)}$ for every $J \in \ii J$. There are two types of intervals $J \in \ii J$:
\begin{itemize}
\item[1.] relatively large intervals $J$ that are supported away from $I_P$, for $P \in T$. In this case, $\ssize_{\rr P}(g)$ becomes
\[
\sup_{\substack{ P \in T\\ |I_P| \leq C'' |J| }} \Big( \frac{1}{|I_P|}  \int_{\rr R} |g(x)|^{r'}  \ci_{I_P}^M(x) \cdot \sum_{k=1}^K |a_k(x)|^{r'} \cdot \one_{\omega_P}(\xi_{k-1}(x))  dx \Big)^\frac{1}{r'}
\]
\item[2.] small intervals $J$, that are smaller than certain $I_P$, for $P \in T$. However, in this case, by the maximality of $J$, there exists $P(J) \in T$ so that $I_{P(J)} \subseteq  3 \tilde J$, where $\tilde J$ is the dyadic parent of $J$. Hence there exists $J'$ of length comparable to $J$ and $I_P$ so that $I_{P(J)} \subseteq J'$. The $\ssize_{\rr P}(g)$ that emerges in this situation is 
\[
\big(  \frac{1}{|J'|} \int_{\rr R} |g(x)|^{r'} \ci_{J'}^M(x) \cdot \sum_{k=1}^K |a_k(x)|^{r'} \cdot \one_{\omega_{J'}}(\xi_{k-1}(x))  dx   \big)^\frac{1}{r'}
\]
\end{itemize}

In either case, we can see how $\ssize_{\rr P}(g)$ can be replaced by a similar quantity, more local in nature, defined by
\begin{equation}
\label{eq:size-g-localized}
 \sup_{P \in \rr P} \sup_{\substack{P' \in \bar{\rr P} \\ P' \geq P, I_{P'} \subseteq 9 I_P}} \big(  \frac{1}{\vert  I_P' \vert} \int_{\rr R} \vert g(x) \vert^{r'} \ci_{I_{P'}}^M(x) \cdot \sum_{k=1}^K \vert a_k(x)  \vert^{r'} \cdot \one_{\omega_{P'}}(\xi_{k-1}(x)) dx \big)^{\frac{1}{r'}}.
\end{equation}
This will be useful later on in Section \ref{sec-local-est-Var-C}, for the local estimates.
\end{proof}
\end{proposition}


\begin{remark}
From now on, we will use the definition in \eqref{eq:size-g-localized} for $\ssize_{\rr P}(g)$. We can also see that this quantity is bounded above by an $L^{r'}$ size as defined in \eqref{eq:def-L^p-size}:
\[
\ssize_{\rr P}(g)\lesssim \sssize^{r'}_{\rr P}(g).
\]
\end{remark}

\subsection{Localization results and Proof of Theorem \ref{thm:multi-var-Carleson}}
\label{sec-local-est-Var-C}
We adapt the previous notions to a local setting: let $I_0$ be a fixed dyadic interval. Then $\rr P(I_0)$ denotes a certain subcollection of multi-tiles:
\[
\rr P(I_0):=\lbrace  P \in \rr P: I_P \subseteq I_0 \rbrace.
\]

The strategy is to obtain precise estimates for $\Lambda^{F, G}_{\ic C^{var, r}; \rr P(I_0)}$ and afterwards use them to derive the vector-valued and sparse estimates.

\begin{proposition}
\label{prop:dec--var-C-f-better}
Let $I_0$ be a fixed dyadic interval and $f$ a function so that $\vert f(x) \vert \leq \one_F(x)$ for a.e. $x$ and assume that 
\[
1+\frac{\dist(\text{supp } f, I_0)}{\vert I_0  \vert} \sim 2^\kappa, \quad \text{for some   } \kappa \geq 0.
\]
 Let $\rr P$ be a collection of multi-tiles so that $I_P \subseteq I_0$ for all $P \in \rr P$ and so that $\ssize_{\rr P}(f) \leq \ic E$. Then there exists a decomposition $\rr P =\rr P' \cup \rr P''$ so that $\ssize_{\rr P'} (f) \leq \ic E/2$ and $\rr P''$ can be written as a union of trees $\ds \rr P''=\bigcup_{T \in \bbf T} T$ with the property that 
\begin{equation}
\label{eq:top-trees-f-decay}
\sum_{T \in \bbf T} \vert I_T \vert \lesssim 2^{-M \kappa} \ic E^{-2} \vert F \vert.
\end{equation}

\end{proposition}

\begin{remark}
The proposition above corresponds to a classical ``energy decomposition" algorithm, that appeared already in Lemma \ref{lemma:dec-lemma}.
\end{remark}

Similarly, for the ``density", we have a decomposition algorithm:
\begin{proposition}
\label{prop:dec-top-trees-g-better}
Let $I_0$ be a fixed dyadic interval, $g$ a function in $L^{r'}$ and assume that 
\[
1+\frac{\dist(\text{supp }\, g, I_0)}{\vert I_0  \vert} \sim 2^\kappa, \quad \text{for some   } \kappa \geq 0.
\]
Let $\rr P$ a collection of multi-tiles so that $\ssize_{\rr P}(g) \leq \lambda$. Then there exists a decomposition $\rr P =\rr P' \cup \rr P''$ so that $\ssize_{\rr P'}(g) \leq \lambda/2$ and $\rr P ''$ can be written as a union of trees $\ds \rr P''= \bigcup_{T \in \rr T} T$ with the property that 
\begin{equation}
\label{eq:top-trees-g-decay}
\sum_{T \in \rr T} \vert  I_T \vert \lesssim 2^{- M \kappa} \lambda^{-r'} \|g\|_{r'}^{r'}.
\end{equation}

\end{proposition}

Using the propositions above, we are able to prove a very precise local estimate, involving an $L^1$ size of $g$.

\begin{proposition}
\label{prop:local-C-var}
Let $I_0$ be a fixed dyadic interval and $\rr P$ a collection of multi-tiles; $F$ is a set of finite measure, $f$ a function so that $|f|\leq\one_F$, and $g$ is a locally integrable function. Then 
\begin{equation}
\label{eq:ineq-lin-form}
\big| \Lambda_{\ic C^{var, r}; \rr P \left( I_0 \right)}(f, g)\big| \lesssim \big( \sssize_{\rr P\left( I_0 \right)} \one_F \big)^{\frac{1}{r'}- \epsilon} \cdot \big( \sssize_{\rr P\left( I_0 \right)} g \big) \cdot |I_0|.
\end{equation}
\begin{proof}

Obtaining an $L^1$ size of $g$ is important for the sparse estimate of Theorem \ref{thm:multi-var-Carleson} (scalar, as well as multiple vector-valued), because the weighted estimates that follow are not optimal if the $L^1$ size is replaced by an $L^{1+\epsilon}$ average. For vector-valued estimates however, an $L^{1+\epsilon}$ size of a characteristic function would suffice.

Still, $\ssize_{\rr P}(g)$ is inherently an $L^{r'}$ quantity: even though in the \emph{tree estimate} we can replace it by an $L^1$ size, for the decomposition result in Proposition \ref{prop:dec-top-trees-g-better} we need to regard it as an $L^{r'}$ quantity in order to estimate $\ds \sum_{T \in \bbf T} |I_T|$.
 
As a consequence, we write $g \in L^1$ as
\[
g = g_1 \cdot g_2, \qquad \text{where} \quad g_1 \in L^{r} \quad \text{and} \quad g_2 \in L^{r'}.
\]

This can easily be achieved by setting $g_1(x)=|g(x)|^\frac{1}{r}$ and $g_2(x)=g(x)/g_2(x)$ if $g(x) \neq 0$, $g_2(x)=0$ if $g(x)=0$.

Now we show how to attain the estimate from \eqref{eq:ineq-lin-form}. First, we decompose $f$ and $g$ as
\[
f=\sum_{\kappa_1 \geq 0} f \cdot \one_{ \lbrace\dist(x, I_0) \sim \left(  2^{\kappa_1}-1 \right) |I_0| \rbrace}:= \sum_{\kappa_1 \geq 0} f_{\kappa_1}
\]
and for $1 \leq j \leq 2$
\[
g_j=\sum_{\kappa_2 \geq 0} g_j \cdot \one_{ \lbrace\dist(x, I_0) \sim \left(  2^{\kappa_2}-1 \right) |I_0| \rbrace}:= \sum_{\kappa_2 \geq 0} g_{j,\kappa_2}.
\]
Similarly, $ \ds F_{\kappa_1}:= F \cap { \lbrace\dist(x, I_0) \sim \left(  2^{\kappa_1}-1 \right) |I_0| \rbrace}$. The functions $g_1$ and $g_2$ will be localized onto the same set, which corresponds to localizing the function $g=g_1 \cdot g_2$.

Then, for $\kappa_1$ and $\kappa_2$ fixed, we perform the decomposition of the collection $\rr P(I_0)$. Applying Proposition \ref{prop:dec--var-C-f-better} iteratively to $f_{\kappa_1}$, we obtain $\ds \rr P(I_0) =\bigcup_{n_1} \bigcup_{T \in \bbf T_{n_1, \kappa_1}} T$, where the size of each tree $\ds T \in \bbf T_{n_1, \kappa_1} \sim 2^{-n_1}$ and moreover
\begin{equation}
\label{eq:weak-L1-local}
 \sum_{T \in \bbf T_{n_1, \kappa_1}}|I_T| \lesssim 2^{-M \kappa_1} 2^{2 n_1} |F_{\kappa_1}|.
\end{equation}


Proposition \ref{prop:dec-top-trees-g-better} applied to $g_2$ yields an analogous decomposition: $\ds \rr P(I_0) =\bigcup_{n_2} \bigcup_{T \in \bbf T_{n_2, \kappa_2}} T$, where the size of each tree $\ds T \in \bbf T_{n_2, \kappa_2} \sim 2^{-n_2}$ and
\[
 \sum_{T \in \bbf T_{n_2, \kappa_2}}|I_T| \lesssim 2^{-M \kappa_2 } 2^{r' n_2}\| g_{2, \kappa_2}  \|_{r'}^{r'}.
\]

First, we estimate the bilinear form associated to $\ic C^{var, r}_{\rr P(I_0)}$: 
\begin{align*}
&\big| \int_{\rr R}\sum_{P \in \rr P \left( I_0 \right)} \langle f, \phi_P \rangle \phi_P\, a_P\, g \,dx \big| =\big| \int_{\rr R} \sum_{P \in \rr P \left( I_0 \right)} \langle f, \phi_P \rangle \phi_P a_P \, g_1 \cdot g_2\, dx \big|
\\& \lesssim \sum_{\kappa_1, \kappa_2} \sum_{n_1, n_2} \sum_{T \in \bbf T_{n_1, \kappa_1} \cap \bbf T_{n_2, \kappa_2}} \big| \int_{\rr R} \big( \sum_{P \in T} \langle f_{\kappa_1}, \phi_P \rangle \phi_P\, a_P\, g_{1,\kappa_2} \cdot g_{2,\kappa_2} \big) \cdot \one_{I_T} \,dx  \big| \\
&+ \sum_{\ell \geq 0} \sum_{\kappa_1, \kappa_2} \sum_{n_1, n_2} \sum_{T \in \bbf T_{n_1, \kappa_1} \cap \bbf T_{n_2, \kappa_2}} \big| \int_{\rr R}  \big( \sum_{P \in T} \langle f_{\kappa_1}, \phi_P \rangle \phi_P\, a_P\, g_{1,\kappa_2} \cdot g_{2,\kappa_2}\big) \cdot \one_{2^{\ell +1}I_T \setminus 2^{\ell}I_T} \,dx\big|  := (I)+(II).
\end{align*}

We will only consider the first term; for the second one, the analysis is similar, and it only employs the decay in the tree estimate outside $2^\ell I_T$, which was mentioned in Proposition \ref{prop:tree-est-Car-var}.

For now, consider $\kappa_1$ and $\kappa_2$ fixed; we apply H\"older to get
\begin{align*}
& \sum_{T \in \bbf T_{n_1, \kappa_1} \cap \bbf T_{n_2, \kappa_2}} \big|  \int_{\rr R} \big( \sum_{P \in T} \langle f_{\kappa_1}, \phi_P \rangle \phi_P \, a_P \, g_{1,\kappa_2} \cdot g_{2,\kappa_2} \big) \cdot \one_{I_T} \, dx  \big|\\
& \lesssim  \sum_{T \in \bbf T_{n_1, \kappa_1} \cap \bbf T_{n_2, \kappa_2}} \| g_{1,\kappa_2} \cdot \one_{I_T}   \|_r \cdot 
 \big\|  \big( \sum_{P \in T} \langle f_{\kappa_1}, \phi_P \rangle \phi_P  \, a_P\, g_{2,\kappa_2} \big) \cdot \one_{I_T}   \big\|_{r'}.
\end{align*}

We want to transform $\ds  \| g_{1,\kappa_2} \cdot \one_{I_T}   \|_{r} $ into some type of average or ``size" of $g_{1,\kappa_2}$ onto $I_T$. In fact, if $k_2 \neq 0$, this term is equal to $0$, but for $(II)$, we consider dilates of $I_T$ and we will have to face a similar situation.

Hence, we can see that 
\[
\frac{1}{|I_T|} \| g_{1,\kappa_2} \cdot \one_{I_T}\|_r^r \leq \frac{1}{|I_T|} \| g \cdot \one_{I_T}\|_1 \leq \sssize^1_{\rr P(I_0)}(g).
\]
Similarly, for $(II)$ we have
\[
\frac{2^{-\ell M}}{|I_T|} \| g_{1,\kappa_2} \cdot \one_{2^{\ell+1} I_T}\|_r^r \leq \frac{2^{-\ell M}}{|I_T|} \| g \cdot \one_{2^{\ell+1}I_T}\|_1 \leq \frac{1}{|I_T|} \| g \cdot \ci^M_{I_T}\|_1 \leq \sssize^1_{\rr P(I_0)}(g).
\]

The above reasoning implies that 
\[
\| g_{1,\kappa_2} \cdot \one_{I_T}   \|_{r}\lesssim \big( \sssize_{\rr P(I_0)}^1 g   \big)^\frac{1}{r} \cdot |I_T|^\frac{1}{r}.
\]

The tree estimate of Proposition \ref{prop:tree-est-Car-var}, and the stopping times in the construction of $\bbf T_{n_1, \kappa_1}$ and $\bbf T_{\n_2, \kappa_2}$ imply that 
\[
\big\|  \big( \sum_{P \in T} \langle f_{\kappa_1}, \phi_P \rangle \phi_P \, a_P \, g_{2,\kappa_2} \big) \cdot \one_{I_T}   \big\|_{r'} \lesssim 2^{-n_1} \cdot 2^{-n_2} \cdot |I_T|^\frac{1}{r'}.
\]

For the sum of the tree tops $|I_T|$ we have from Propositions \ref{prop:dec--var-C-f-better} and \ref{prop:dec-top-trees-g-better}:
\[
\sum_{T \in \bbf T_{n_1, \kappa_1} \cap \bbf T_{n_2, \kappa_2}} |I_T| \lesssim \min( 2^{-M \kappa_1} 2^{2 n_1} |F_{\kappa_1}|, 2^{-M \kappa_2} 2^{r' n_2} \|  g_{2,\kappa_2} \|_{r'}^{r'}).
\]

All of the above imply that 
\begin{align*}
& \sum_{T \in \bbf T_{n_1, \kappa_1} \cap \bbf T_{n_2, \kappa_2}} \big| \int_{\rr R}  \big( \sum_{P \in T} \langle f_{\kappa_1}, \phi_P \rangle \phi_P\,  a_P\, g_{1,\kappa_2} \cdot g_{\kappa_2} \big) \cdot \one_{I_T} dx  \big| \\
& \lesssim \big( \ssize_{\rr P(I_0)}^1(g)  \big)^\frac{1}{r} \,2^{-n_1 } 2^{-n_2 } \sum_{T \in \bbf T_{n_1, \kappa_1} \cap \bbf T_{n_2, \kappa_2}}  |I_T| \\
&\lesssim \big( \ssize_{\rr P(I_0)}^1(g)  \big)^\frac{1}{r} \, 2^{-n_1 } 2^{-n_2}  \big(   2^{-M \kappa_1} 2^{2 n_1} |F_{\kappa_1}| \big)^{\theta_1} \cdot  \big(   2^{-M \kappa_2} 2^{r' n_2} \| g_{2,\kappa_2} \|_{r'}^{r'} \big)^{\theta_2},
\end{align*}
where $0 \leq \theta_1, \theta_2 \leq 1$, and $\theta_1+\theta_2=1$.

Now we can sum in $\kappa_1$ and $\kappa_2$ to get
\begin{align}
\label{eq:this-decrease-exp}
(I) & \lesssim \big( \ssize_{\rr P(I_0)}^1(g)  \big)^\frac{1}{r} \, \| \one_F \cdot \ci_{I_0}^{\tilde M}   \|_1^{\theta_1} \cdot \| g_{2} \cdot \ci_{I_0}^{\tilde M}   \|_{r'}^{\theta_2\cdot r'} \sum_{n_1, n_2} 2^{-n_1 \left( 1 -2 \theta_1\right)} \cdot 2^{-n_2 \left( 1 - r' \theta_2  \right)}.
\end{align}

Note that $\|g_2 \cdot \ci_{I_0}^{\tilde M}   \|_{r'}^{r'} = \|  g \cdot \ci_{I_0}^{\tilde{ \tilde M}} \|_1$, so, if we can sum in $n_1$ and $n_2$, we would have that 
\begin{align*}
(I) & \lesssim \big( \sssize_{\rr P(I_0)} \one_F   \big)^{1 -\theta_1} \cdot \big( \sssize_{\rr P(I_0)}^1 g  \big) \cdot |I_0|,
\end{align*}
where we used the inequalities
\[
2^{-n_1 } \sim \ssize_{T} (f) \lesssim \sssize_{\rr P(I_0)}(f), \qquad 2^{-n_2} \sim \ssize_T \, (g) \lesssim \sssize_{\rr P(I_0)}^{r'} \, g_2= \big( \sssize_{\rr P(I_0)}^{1} \, g \big)^\frac{1}{r'}.
\]

The summation is possible only if 
\begin{equation*}
1- 2 \theta_1 >0, \qquad 1 - r' \theta_2 >0,
\end{equation*}
or equivalently, if 
\begin{equation}
\label{eq:cond-summation}
\frac{1}{2} > \theta_1> \frac{1}{r}.
\end{equation}

Upon treating the second term $(II)$ in a similar manner, we obtain that 
\[
\big| \int_{\rr R}\ic C^{var, r} _{\rr P(I_0)}(f)(x) g(x) dx \big| \lesssim \big( \sssize_{\rr P(I_0)} \one_F   \big)^{1- \theta_1} \cdot \big( \sssize_{\rr P(I_0)}^1 g   \big) \cdot |I_0|.
\]

Choosing $\theta_1=\frac{1}{r}+\epsilon$, we obtain exactly \eqref{eq:ineq-lin-form}.

\end{proof}
\end{proposition}

\begin{theorem}
\label{thm:local-C-var}
Let $I_0$ be a fixed dyadic interval, and $F, G$ sets of finite measure. Then for any $p_0 >r'$, 
\begin{equation*}
\vert \Lambda_{\ic C^{var, r}; \rr P(I_0)}^{F, G}(f,g)  \vert \lesssim \big( \sssize_{\rr P \left( I_0 \right)}  \one_F \big)^{\frac{1}{r'}- \frac{1}{p_0}-\epsilon} \cdot \big( \sssize_{\rr P \left( I_0 \right)}  \one_G \big)^{\frac{1}{p_0}-\epsilon} \|f \cdot \ci_{I_0}\|_{p_0} \cdot \|g \cdot \ci_{I_0}\|_{p'_0}.
\end{equation*}
\begin{proof}
We will use the estimates of Proposition \ref{prop:local-C-var} and restricted type interpolation in the form of Proposition 10 of \cite{quasiBanachHelicoid}. Let $E_1, E_2$ be sets of finite measure and we can assume without loss of generality that $|E_2|=1$. We will construct $E_2' \subseteq E_2$ major subset and show that 
\[
 \vert \Lambda_{\ic C^{var, r}; \rr P(I_0)}^{F, G}(f,g)  \vert \lesssim \big\| \Lambda_{\ic C^{var, r}; \rr P(I_0)}\big\| \cdot |E_1|^\frac{1}{p},
 \]
for $p >r'$ in a neighborhood of $p_0$, and all functions $f, g$ with $|f(x)| \leq \one_{E_1}, |g(x)| \leq \one_{E'_2}$ for a.e. $x$.

The exceptional set is defined in a standard way:
\[
\Omega:=\lbrace x: \ic M (\one_{E_1 \cap F}) > C |E_1 \cap F|   \rbrace
\]
and we set $\ds E_2':=E_2 \setminus \Omega$.

For any $d \geq 0$, we let $\ds \rr P_d(I_0):= \lbrace P \in \rr P(I_0) : 1 +\frac{\dist(I_P, \Omega^c)}{|I_P|} \sim 2^d  \rbrace$. Another decomposition will be needed: we consider $d$ fixed, and as in \cite{vv_BHT}, we will perform a double stopping time which allows us to write
\[
\rr P_d(I_0):=\bigcup_{n_1, n_1} \bigcup_{I \in \ii I_{n_1}\cap \ii I_{n_2} } \rr P_{I}.
\]

Here $\ii I_{n_1}$ represents a collection of maximal dyadic intervals contained inside
\[
\lbrace x: \ic M(\one_{F \cap E_1})(x) > 2^{-{n_1}-1}  \rbrace.
\] 
More exactly,  $I \in \ii I_{n_1}$ if there exists at least one $P \in \rr P_{Stock}$ with $I_P \subseteq I$, and $I$ is maximal with the property that 
\[
2^{-n_1-1} < \frac{1}{|I|} \int_{\rr R} \one_{F \cap E_1} \cdot \ci_I^M dx \leq 2^{-n_1}.
\]
Then we set
 $$\rr P_I :=\lbrace P \in \rr P_{Stock}: I_P \subseteq I  \rbrace.$$ 

This implies that $\ds \sssize_{\rr P_I}(f \cdot \one_{F}) \lesssim 2^{-n_1} \lesssim \sssize_{\rr P(I_0)} \, \one_{F \cap E_1} \lesssim 2^d |E_1 \cap F|$, and also $\ds \sum_{I \in \ii I_{n_1}} |I| \lesssim 2^{n_1} |E_1 \cap F|$.

Similarly for $g$, we have $\ds \sssize_{\rr P_I}(g \cdot \one_{G}) \lesssim 2^{-n_2} \lesssim 2^{-Md}$, and $\ds \sum_{I \in \ii I_{n_2}} |I| \lesssim 2^{n_2} |E_2 \cap G|$.
We recall that 
\[
\sssize_{\rr P_I}(g \cdot \one_{G})  \lesssim \sup_{P \in \rr P_I} \sup_{J \supseteq I_P} \frac{1}{|J|} \int_J \one_{E'_2 \cap G} \cdot \ci_I^{2M} dx \lesssim 2^{-Md}.
\]
The decay in $2^{-Md}$ is a consequence of integrating $\ci_I^M$ over $E_2' \subset \Omega^c$.

Then we obtain, using Proposition \ref{prop:local-C-var} with $F$ replaced by $E_1 \cap F$ and $g$ by $\one_{E'_2 \cap G}$:
\begin{align*}
\big| \Lambda_{\ic C^{var, r}; \rr P(I_0)}^{F, G}(f,g) \big|& \lesssim \sum_{d \geq 0} \sum_{n_1, n_2} \sum_{I \in \ii I_{n_1 \cap \ii I_{n_2}}} \big|   \Lambda_{\ic C^{var, r}; \rr P_I}^{F, G}(f,g)\big|\\
&\lesssim \sum_{d \geq 0} \sum_{n_1, n_2} \sum_{I \in \ii I_{n_1} \cap \ii I_{n_2}} \big( \sssize_{\rr P_I} \one_{E_1 \cap F} \big)^{\frac{1}{r'}-\epsilon} \cdot \big( \sssize_{\rr P_I} \one_{E'_2 \cap G} \big) \cdot |I|.
\end{align*}

We note that
\[
\sssize_{\rr P_I} \one_{E_1 \cap F} \leq \min\big( \sssize_{\rr P(I_0)} \one_{ F}, 2^{-n_1}  \big), \quad \sssize_{\rr P_I} \one_{E'_2 \cap G } \leq \min\big( \sssize_{\rr P(I_0)} \one_{ G}, 2^{-n_2}  \big).
\]

Hence, if $0< \gamma_1 <\dfrac{1}{r'}$ and $0 < \gamma_2< 1$, we have
\begin{align*}
\big| \Lambda_{\ic C^{var, r}; \rr P(I_0)}^{F, G}(f,g) \big|  & \lesssim \big( \sssize_{\rr P(I_0)} \one_{ F}  \big)^{\frac{1}{r'} -\gamma_1 -\epsilon} \cdot \big( \sssize_{\rr P(I_0)} \one_{ G}  \big)^{1 -\gamma_2 -\epsilon} \\ &  \cdot \sum_{d \geq 0} \sum_{n_1, n_2} 2^{-n_1  \gamma_1} \cdot  2^{-n_2  \gamma_2} \left( 2^{n_1} |E_1 \cap F|  \right)^{\theta_1}  \left( 2^{n_2} |E_2 \cap G|  \right)^{\theta_2},
 \end{align*}
where $0 \leq \theta_1 , \theta_2 \leq 1$, $\theta_1+\theta_2= 1$. On the last line we interpolate between the estimate for $\ds \sum_{I \in \ii I_{n_1}}|I|$ and the one for $\ds \sum_{I \in \ii I_{n_2}}|I|$.

We want to sum in $n_1$ and $n_2$, and this is possible provided $\gamma_1+\gamma_2 > \theta_1+\theta_2=1$. We can choose $\gamma_1=\dfrac{1}{p}$ and $\gamma_2=\dfrac{1}{p'}+\epsilon$, where $p>r'$ is arbitrarily close to $p_0$. Summation in $d$ presents no problem due to the fast decay. In the end, all of the above imply that 
\begin{align*}
& \big| \Lambda_{\ic C^{var, r}; \rr P(I_0)}^{F, G}(f,g) \big| \lesssim \big( \sssize_{\rr P(I_0)} \one_{ F}  \big)^{\frac{1}{r'} -\frac{1}{p_0} -\epsilon} \cdot \big( \sssize_{\rr P(I_0)} \one_{ G}  \big)^{1 -\frac{1}{p'_0} -\epsilon}  |E_1|^{\frac{1}{p}},
\end{align*}
which is precisely what we wanted. (The assumption that $|E_2|=1$ is only used for simplifying the computations.)
\end{proof}
\end{theorem}

\subsection{$L^1$ sizes for the multiple vector-valued sparse domination}~\\
\label{sec:$L^1$ sizes for the multiple vector-valued sparse domination}
The local estimate for the variational Carleson $\ic C^{var, r}$ from Proposition \ref{prop:local-C-var}, which is the equivalent of Lemma \ref{lemma:local-Carleson} for the Carleson operator, implies, as described in Section \ref{sec:local->sparse}, the following sparse domination result: \emph{given $\rr P$ a finite collection of multi-tiles,  $f$ and $g$ locally integrable functions and $\epsilon >0$, there exists a sparse family of dyadic intervals $\ic S$ so that 
\[
\big|  \Lambda_{\ic C^{var, r}; \rr P} (f, g) \big| \lesssim_{\epsilon} \sum_{Q \in \ic S} \big( \frac{1}{\vert Q \vert} \int_{\rr R} \vert f(x) \vert^{r'+\epsilon} \cdot \ci_{Q}^M(x) dx \big)^{\frac{1}{r'+\epsilon}} \cdot \big( \frac{1}{\vert Q \vert} \int_{\rr R} \vert g(x) \vert \cdot \ci_{Q}^M(x) dx \big) \cdot \vert Q \vert.
\]}
 
Using the helicoidal method, we obtain a similar result for multiple vector-valued functions $\vec f, \vec g$, but with the $L^1$ average of $\vec g$ replaced by an $L^{1+\epsilon}$ average: \emph{given $\epsilon>0$, $\rr P$ a finite collection of multi-tiles,  $\vec f$ and $ \vec g$ vector-valued functions so that $\| f(x, \cdot) \|_{L^R(\ii W, \mu)}$ and $\| g(x, \cdot) \|_{L^{R'}(\ii W, \mu)}$ are locally integrable functions, there exists a sparse family of dyadic intervals $\ic S$ so that
{\fontsize{9}{10}\[
\vert \Lambda_{\ic C^{var, r};{\rr P}}(\vec f, \vec g)  \vert \lesssim_{\epsilon} \sum_{Q \in \ic S} \big( \frac{1}{\vert Q \vert} \int_{\rr R} \| f(x, \cdot) \|_{L^R(\ii W, \mu)}^{r'+\epsilon} \cdot \ci_{Q}^M(x) dx \big)^{\frac{1}{r'+\epsilon}} \cdot \big( \frac{1}{\vert Q \vert} \int_{\rr R} \| g(x, \cdot) \|_{L^{R'}(\ii W, \mu)}^{1+\epsilon} \cdot \ci_{Q}^M(x) dx \big)^{\frac{1}{1+\epsilon}} \cdot \vert Q \vert.
\]
}}

However, we can improve on this result, namely by proving that we can indeed have an $L^1$ average of $\| g(x, \cdot) \|_{L^{R'}(\ii W, \mu)}$. We will only show this in the case $\ds L^{R}(\ii W, \mu) = \ell^s$, for some $s>r'$, but our proof can be adapted to the general setting.

In Proposition \ref{prop:local-C-var}, we obtained an $L^1$ size of $g$ by splitting the function as $g=g_1 \cdot g_2$, where $g_1 \in L^r$ and $g_2 \in L^{r'}$. Because we want to obtain an $L^1$ average for $g$, we cannot afford to use restricted-type functions (that is, we cannot assume that $|g(x)| \leq \one_{G}(x)$ for some finite set $G \subset \rr R$), and instead we use the above decomposition of the function. The situation will be similar for the vector-valued estimate, in the sense that we will need to find a ``good" splitting of $\vec g$.

\begin{proposition}
\label{prop:local-ell^s-aver-L^1}
Let $s>r'$, $\rr P$ a finite collection of multi-tiles, $I_0$ a fixed dyadic interval, and $F$ a set of finite measure. Let $\vec f = \lbrace f_k \rbrace_k$ and $\vec g =\lbrace g_k  \rbrace_k$ be vector-valued functions so that $\ds \| \vec f (x) \|_{\ell^s} \leq \one_F(x)$ for a.e. $x$ and $\ds \|  \vec g(x) \|_{\ell^{s'}}$ is a locally integrable function on $\rr R$. Then 
\begin{equation}
\label{eq:local-vv-est-var-C}
\big| \sum_k \Lambda_{\ic C^{var, r}; \rr P(I_0)}(f_k, g_k) \big| \lesssim \big( \sssize_{\rr P(I_0)} \one_F   \big)^{\frac{1}{r'}-\epsilon} \cdot \big(  \sssize_{\rr P(I_0)}^1 \| \vec g \|_{\ell^{s'}} \big) \cdot |I_0|.
\end{equation}
As a consequence, we obtain the following sparse domination result: given $\vec f$ and $ \vec g$ vector-valued functions so that $\| \vec f(x) \|_{\ell^s}$ and $\| \vec g(x) \|_{\ell^{s'}}$ are locally integrable functions and $\epsilon >0$, there exists a sparse family $\ic S$ of dyadic intervals  so that
\[
\vert \Lambda_{\ic C^{var, r};{\rr P}}(\vec f, \vec g)  \vert \lesssim_{\epsilon} \sum_{Q \in \ic S} \big( \frac{1}{\vert Q \vert} \int_{\rr R} \| \vec f(x) \|_{\ell^s}^{r'+\epsilon} \cdot \ci_{Q}^M(x) dx \big)^{\frac{1}{r'+\epsilon}} \cdot \big( \frac{1}{\vert Q \vert} \int_{\rr R} \| \vec g(x) \|_{\ell^{s'}} \cdot \ci_{Q}^M(x) dx \big)\cdot \vert Q \vert.
\]
\end{proposition}

The proof of estimate \eqref{eq:local-vv-est-var-C} relies on the following localization result:
\begin{proposition}
\label{prop:local-est-var-C-L^q-size}
Let $q'>r'$, let $\rr P$ be a finite collection of multi-tiles, $I_0$ a fixed dyadic interval, $F$ a set of finite measure and $f, g_1, g_2$ locally integrable functions, with the additional property that $|f(x)|\leq \one_F(x)$ for a.e. $x$. Then 
\begin{equation}
\label{eq:local-L^q-aver-var-C}
\big| \Lambda_{\ic C^{var, r}; \rr P(I_0)}(f, g_1 \cdot g_2) \big| \lesssim \big( \sssize_{\rr P(I_0)} \one_F  \big)^{\frac{1}{r'}-\frac{1}{q'}-\epsilon} \left(  \sssize_{\rr P(I_0)}^{q'} \,g_1 \right) \cdot |F|^{\frac{1}{q'}} \cdot \|g_2 \cdot \ci_{I_0}^M\|_q.
\end{equation}
\end{proposition}

We postpone the proof of Proposition \ref{prop:local-est-var-C-L^q-size} in order to show how it implies Proposition \ref{prop:local-ell^s-aver-L^1}:
\begin{proof}[Proof of Proposition \ref{prop:local-ell^s-aver-L^1}]
We know that $\ds \| \vec f (x) \|_{\ell^s} \leq \one_F(x)$ and we want to estimate
\[
\big|  \Lambda_{\ic C^{var, r}; \rr P(I_0)} (f_k, g_k)  \big|=\big|  \Lambda_{\ic C^{var, r}; \rr P(I_0)} (f_k \cdot \one_F, g_k)  \big|:=\big|  \Lambda^F_{\ic C^{var, r}; \rr P(I_0)} (f_k, g_k)  \big|.
\]
We write $\ds g_k:= \tilde g_k \cdot \mathfrak{g}$, where $\g:=\| \vec g  \|_{\ell^{s'}}$ (if $\g (x)=0$, we simply set $\tilde g_k(x)=0$). Then we apply the result of Proposition \ref{prop:local-est-var-C-L^q-size} to $\Lambda_{\ic C^{var, r}; \rr P(I_0)} (f_k \cdot \one_F, g_k)$, under the assumption that $|f_k(x)| \leq \one_{E_1}$. In fact, we have $|f_k(x)| \leq \one_{E_1 \cap F}$ (because $f_k$ is supported inside $F$), and we can ameliorate the estimate of Proposition \ref{prop:local-est-var-C-L^q-size}. Here we prove the inequality in the case of restricted-type functions, and as a consequence of interpolation theory, we will deduce the general case:
\begin{align*}
\big| \Lambda_{\ic C^{var, r}; \rr P(I_0)} (f_k \cdot \one_F, \g^\frac{1}{q'} \cdot \g^\frac{1}{q} \tilde g_k) \big| &\lesssim \big( \sssize_{\rr P(I_0)} \one_F  \big)^{\frac{1}{r'}-\frac{1}{q'}-\epsilon} \left(  \sssize_{\rr P(I_0)}^{q'} \, \g^\frac{1}{q'} \right) \cdot |E_1|^{\frac{1}{q'}} \cdot \| \g^\frac{1}{q} \tilde g_k \|_q \\
&\lesssim \big( \sssize_{\rr P(I_0)} \one_F  \big)^{\frac{1}{r'}-\frac{1}{q'}-\epsilon} \left(  \sssize_{\rr P(I_0)}^{1} \, \g \right)^\frac{1}{q'} \cdot |E_1|^{\frac{1}{q'}} \| \tilde g_k \|_{L^q(\g )}.
\end{align*}

Above, we applied Proposition \ref{prop:local-est-var-C-L^q-size}, with $F$ a fixed set of finite measure and $\g$ a fixed locally integrable function, with $\g \geq 0$, to obtain that the operator defined by
\[
f \mapsto \ic C_{\rr P}^{var, r} (f \cdot \one_F) 
\]
is bounded from $L^{q', 1}$ into $L^{q'}(\g )$ with an operatorial norm equal to $$\ds \big( \sssize_{\rr P(I_0)} \one_F  \big)^{\frac{1}{r'}-\frac{1}{q'}-\epsilon} \left(  \sssize_{\rr P(I_0)}^{1} \, \g \right)^\frac{1}{q'}.$$ 
Upon using restriction onto the sets $\lbrace x: \dist(x, I_0) \sim (2^\kappa -1) |I_0|\rbrace$ (we can perform a decomposition of $f$ as in the proof of Proposition \ref{prop:local-C-var}), which entails a decaying factor $2^{- \kappa M}$,and interpolation ($q'$ is arbitrary, and it can vary in a small neighborhood of $s>r'$), we deduce that 
\[
\big| \Lambda_{\ic C^{var, r}; \rr P(I_0)} (f_k \cdot \one_F, \g^\frac{1}{s} \cdot \g^\frac{1}{s'} \tilde g_k) \big| \lesssim \big( \ssize_{\rr P(I_0)} \one_F  \big)^{\frac{1}{r'}-\frac{1}{s}-\epsilon} \left(  \sssize_{\rr P(I_0)}^{1} \, \g \right)^\frac{1}{s} \cdot \| f_k \cdot \ci_{I_0}^M \|_{s} \cdot \| \tilde g_k \cdot \ci_{I_0}^M \|_{L^{s'}(\g )}.
\]

It is due to the log-convexity in the interpolation (for this, see Theorem 1.2.19 of \cite{grafakos-book}) that we can obtain the suitable exponent for $\sssize_{\rr P(I_0)}^{1} \, \g$.

From here on, the proof is standard: we sum in $k$ to get that 
\begin{align*}
\big| \sum_k \Lambda_{\ic C^{var, r}; \rr P(I_0)}(f_k, g_k) \big| &\lesssim  \big( \ssize_{\rr P(I_0)} \one_F  \big)^{\frac{1}{r'}-\frac{1}{s}-\epsilon} \left(  \sssize_{\rr P(I_0)}^{1} \, \g \right)^\frac{1}{s} \cdot \| \one_F \cdot \ci_{I_0}^M\|_{s} \cdot \big \| \big( \sum_k |\tilde g_k|^{s'}\big)^{\frac{1}{s'}} \cdot \ci_{I_0}^M \big\|_{L^{s'}(\g )} \\
& \lesssim  \big( \ssize_{\rr P(I_0)} \one_F  \big)^{\frac{1}{r'}-\epsilon} \left(  \sssize_{\rr P(I_0)}^{1} \, \g \right)\cdot  |I_0|,
 \end{align*}
where we used that $\ds \big( \sum_k |\tilde g_k|^{s'} \big)^\frac{1}{s'}=1$ on the set where $\g(x) \neq 0$ and hence
\[
\| \big( \sum_k |\tilde g_k|^{s'} \big)^\frac{1}{s'} \cdot \ci_{I_0}^M  \|_{L^{s'}(\g )} = \| \g \cdot \ci_{I_0}^{\tilde M} \|_1^\frac{1}{s'}.
\]

If we were allowed to use restricted-type functions, the ``weight" $\g$ would be replaced by $\one_G$, as in the previous analysis of vector-valued inequalities from \cite{vv_BHT}.
\end{proof}

We end by providing a proof for Proposition \ref{prop:local-est-var-C-L^q-size}; the case when $q'=r'$ and $g_1, g_2$ are particular functions (we recall that $\|g_1\|_r^r=\|g_2\|_{r'}^{r'}=\|g\|_1$) was already presented in Proposition \ref{prop:local-C-var}.

\begin{proof}[Proof of Proposition \ref{prop:local-est-var-C-L^q-size}]
In order to prove the estimate \eqref{eq:local-L^q-aver-var-C}, we will make use of Proposition \ref{prop:local-C-var}. As in the proof of Theorem \ref{thm:local-C-var}, we perform a stopping time for $\one_F$ with respect to $L^{q'-\epsilon}$ averages, and a stopping time for $g_2$ with respect to $L^q$ averages.

More exactly, for every $n_1$ so that  $2^{-n_1} \leq \sssize_{\rr P(I_0)}^{q'-\epsilon} \one_F \lesssim 1$, we construct a collection $\ii I^{n_1}$ of maximal intervals $I$ for which 
\[
2^{-n_1} \leq \big( \frac{1}{|I|} \int_{\rr R} \one_F \cdot \ci_I^M dx  \big)^{\frac{1}{q'-\epsilon}} \leq 2^{-n_1+1},
\]
and for every such interval we have associated a non-empty collection of multi-tiles $\rr P_I$ so that $\sssize_{\rr P_I}^{q' -\epsilon} \one_F \lesssim 2^{-n_1}$. In particular, we note that 
\[
\sum_{I \in \ii I^{n_1}} |I| \lesssim 2^{n_1 \left( q'-\epsilon\right)}|F|.
\]

Similarly, for every $n_2$ so that $ 2^{-n_2} \leq \sssize_{\rr P(I_0)}^q g_2$, we construct a collection $\ii I^{n_2}$ of maximal intervals $I$ so that 
\[
2^{-n_2} \leq \big( \frac{1}{|I|} \int_{\rr R} |g_2(x)|^q \cdot \ci_{I_0}^{Mq} \cdot \ci_I^M dx  \big)^{\frac{1}{q}} \leq 2^{-n_2+1}.
\]
For the selected intervals $I \in \ii I^{n_2}$, we assume there exists a multi-tile $P \in \rr P_{Stock}$ with $I_P \subseteq I$, and consequently a nonempty subcollection of multi-tiles $\rr P_{I} \subseteq \rr P$ so that $\sssize_{\rr P_I}^q g_2 \lesssim 2^{-n_2}$. We also obtain that
\[
\sum_{I \in \ii I^{n_2}} |I| \lesssim 2^{n_2 q} \| g_2  \cdot \ci_{I_0}^M \|_q^q.
\]

Once these arrangements are made, we are ready to prove \eqref{eq:local-L^q-aver-var-C}:
\begin{align*}
\big|  \Lambda_{\ic C^{var, r}; \rr P(I_0)}(f, g_1 \cdot g_2) \big| & \lesssim \sum_{n_1, n_2} \sum_{I \in \ii I^{n_1} \cap \ii I^{n_2}} \big|  \Lambda_{\ic C^{var, r}; \rr P_I}(f, g_1 \cdot g_2) \big|  \\
&\lesssim \sum_{n_1, n_2} \sum_{I \in \ii I^{n_1} \cap \ii I^{n_2}}  \big( \sssize_{\rr P_I} \one_F \big)^{\frac{1}{r'}- \epsilon} \cdot \big( \sssize_{\rr P_I} g_1 \cdot g_2 \big) \cdot |I| \\
&\lesssim \big( \sssize_{\rr P(I_0)} \one_F \big)^{\frac{1}{r'}- \frac{1}{q'-\epsilon}- \epsilon} \cdot \big( \sssize_{\rr P(I_0)}^{q'} g_1\big) \sum_{n_1, n_2} \sum_{I \in \ii I^{n_1} \cap \ii I^{n_2}}  2^{-n_1} 2^{-n_2} |I|,
\end{align*}
and it only remains to prove that 
\begin{equation}
\label{eq:equation-to-prove}
\sum_{n_1, n_2} \sum_{I \in \ii I^{n_1} \cap \ii I^{n_2}}  2^{-n_1} 2^{-n_2} |I| \lesssim |F|^\frac{1}{q'} \|g_2  \cdot \ci_{I_0}^M\|_q.
\end{equation}

We note that 
\begin{align*}
\sum_{n_1, n_2} \sum_{I \in \ii I^{n_1} \cap \ii I^{n_2}}  2^{-n_1} 2^{-n_2} |I| \lesssim \sum_{n_1} \sum_{n_2}  2^{-n_1} 2^{-n_2} \cdot \min \big( 2^{n_1 \left( q'-\epsilon\right)}|F|,  2^{n_2 q} \| g_2  \cdot \ci_{I_0}^M \|_q^q \big).
\end{align*}

We do not use the interpolation estimate $\min(A, B) \leq A^\theta \cdot B^{1-\theta}$; instead, we analyze the two possibilities:
\begin{itemize}
\item[(i)]
if $\ds 2^{n_1 \left( q'-\epsilon\right)}|F| \leq 2^{n_2 q} \| g_2  \cdot \ci_{I_0}^M\|_q^q$, then $\ds 2^{-n_2} \leq 2^{- \frac{n_1 \left( q'-\epsilon \right)}{q}} |F|^\frac{-1}{q}\|g_2  \cdot \ci_{I_0}^M\|_q$ and 
\begin{align*}
& \sum_{n_1} \sum_{n_2}  2^{-n_1} 2^{-n_2} \cdot \min \big( 2^{n_1 \left( q'-\epsilon\right)}|F|,  2^{n_2 q} \| g_2  \cdot \ci_{I_0}^M \|_q^q \big) \lesssim \sum_{n_1} 2^{-n_1} 2^{n_1 \left( q'-\epsilon\right)}|F| \sum_{n_2} 2^{-n_2} \\
&\lesssim \sum_{n_1} 2^{-n_1} 2^{n_1 \left( q'-\epsilon\right)}|F|  \cdot 2^{-\frac{n_1 \left( q'-\epsilon \right)}{q}} |F|^\frac{-1}{q}\|g_2  \cdot \ci_{I_0}^M\|_q \lesssim \sum_{n_1} 2^{-n_1 \left( 1-\frac{q'-\epsilon}{q'} \right)}\cdot |F|^\frac{1}{q'} \cdot \|g_2  \cdot \ci_{I_0}^M\|_q.
\end{align*}

\item[(ii)] on the other hand, if $\ds 2^{n_1 \left( q'-\epsilon\right)}|F| \geq 2^{n_2 q} \| g_2 \cdot \ci_{I_0}^M \|_q^q$, we obtain that 
\begin{align*}
& \sum_{n_1} \sum_{n_2}  2^{-n_1} 2^{-n_2} \cdot \min \big( 2^{n_1 \left( q'-\epsilon\right)}|F|,  2^{n_2 q} \| g_2  \cdot \ci_{I_0}^M\|_q^q \big) \lesssim \sum_{n_1}  2^{-n_1} \sum_{n_2} 2^{-n_2} 2^{n_2 q} \| g_2  \cdot \ci_{I_0}^M\|_q^q \\
&\lesssim \sum_{n_1}  2^{-n_1} \sum_{n_2} 2^{n_2 \left( q-1 \right)} \|g_2  \cdot \ci_{I_0}^M\|_q^q \lesssim \sum_{n_1}  2^{-n_1} 2^{\frac{n_1 \left( q'-\epsilon  \right) \left(  q-1  \right)  }{q} } |F|^{\frac{q-1}{q}} \|g_2  \cdot \ci_{I_0}^M\|_q^{-\left( q-1 \right)} \cdot \|g_2  \cdot \ci_{I_0}^M\|_q^q  \\
&\lesssim  \sum_{n_1}  2^{-n_1 \left( 1-\frac{q'-\epsilon}{q'} \right)}  |F|^{\frac{1}{q'}} \cdot \|g_2  \cdot \ci_{I_0}^M\|_q.
\end{align*}

In either case, the series in $n_1$ are summable and we obtain \eqref{eq:equation-to-prove}.
\end{itemize}
\end{proof}

\begin{remark}
The same analysis applies in the case of Carleson operator; the statement of Proposition \ref{prop:local-est-var-C-L^q-size} becomes: \emph{for any $q'>1$, and for any finite collection of tiles $\rr P(I_0)$, any set of finite measure $F$ and any locally integrable functions $f, g_1, g_2$ so that $|f(x)|\leq \one_F(x)$ for a.e. $x$, we have
\[
\big| \Lambda_{\ic C; \rr P(I_0)}(f, g_1 \cdot g_2) \big| \lesssim \big( \sssize_{\rr P(I_0)} \one_F \big)^{1-\frac{1}{q'}-\epsilon} \cdot \big(  \sssize_{\rr P(I_0)}^{q'} g_1 \big) \cdot |F|^\frac{1}{q'} \cdot \|g_2 \cdot \ci_{I_0}^M\|_q.
\]
}
And this implies, for any $1<s<\infty$, and any $\vec f=\lbrace f_k\rbrace_k$, $\vec g=\lbrace g_k\rbrace_k$ with $\ds \| \vec f\|_{\ell^s} \leq \one_F(x)$:
\begin{equation}
\big| \sum_k \Lambda_{\ic C;\rr P(I_0)}(f_k, g_k) \big| \lesssim \big( \sssize_{\rr P(I_0)} \one_F   \big)^{1-\epsilon} \cdot \big(  \sssize_{\rr P(I_0)}^1 \| \vec g \|_{\ell^{s'}} \big) \cdot |I_0|.
\end{equation}
\end{remark}

\begin{remark}
In the general case of iterated vector spaces, the corresponding inductive statement for the variational Carleson operator is: \emph{
for any $q'>r'$, and for any finite collection of multi-tiles $\rr P(I_0)$,   any sets of finite measure $F$ and $E_1$, and any vector-valued functions $\vec f, \vec g_1, \vec g_2$ so that $\| f(x, \cdot)\|_{L^{R}(\ii W, \mu)}, \| g_1(x, \cdot)\|_{L^{R'}(\ii W, \mu)}$ and $\| g_2(x, \cdot)\|_{L^{R'}(\ii W, \mu)}$ are locally integrable functions with $\| f(x, \cdot)\|_{L^{R}(\ii W, \mu)}\leq \one_{E_1}(x)$ for a.e. $x$, we have
{\fontsize{9}{10}\[
\big| \Lambda^F_{\ic C^{var, r};\rr P(I_0)}(\vec f, \overrightarrow{ (g_1 \cdot g_2)}) \big| \lesssim \big( \sssize_{\rr P(I_0)} \one_F \big)^{1-\frac{1}{q'}-\epsilon} \cdot \big(  \sssize_{\rr P(I_0)}^{q'} \| g_1(x, \cdot)\|_{L^{R'}(\ii W, \mu)} \big) \cdot |E_1|^\frac{1}{q'} \cdot \big \| \big\| g_2(x, \cdot)\big\|_{L^{R'}(\ii W, \mu)}  \cdot \ci_{I_0}^M \big\|_{L^q_x}.
\]
}}
Here $\ds \overrightarrow{ (g_1 \cdot g_2)}$ denotes the vector-valued function written component-wise as $ \ds \overrightarrow{ (g_1 \cdot g_2)}(x, w)=\vec g_1(x, w) \cdot \vec g_2(x, w)$ for a.e. $(x, w) \in \rr R \times \ii W$.
\end{remark}

\bibliographystyle{alpha}
\bibliography{HelicoidalMethodVVandSparse.bbl}

\end{document}